\documentclass[reqno]{amsart} 
\usepackage{fullpage,amssymb,epsf,psfrag,graphicx}
\usepackage{amssymb,epsf,psfrag,graphicx}

\makeatletter\@addtoreset{equation}{section}\makeatother
\makeatletter\@addtoreset{figure}{section}\makeatother
\makeatletter\@addtoreset{table}{section}\makeatother

\newtheorem{theorem}{Theorem}[section]
\newtheorem{lemma}[theorem]{Lemma}

\newtheorem{corollary}[theorem]{Corollary}

\theoremstyle{definition}
\newtheorem{definition}[theorem]{Definition}
\theoremstyle{remark}
\newtheorem{remark}[theorem]{Remark}

\newcommand{\R}{{\mathbb R}}
\newcommand{\C}{{\mathbb C}}
\newcommand{\Z}{{\mathbb Z}}

\newcommand{\Proj}{{\mathbb P}}

\newcommand{\op}[1]{\!\!\mathop{\rm ~#1}\nolimits}

\newcommand{\scriptop}[1]{\!\!\mathop{\mbox{\rm \scriptsize ~#1}}\nolimits}

\newcommand\PI{\hbox{P}_{\scriptsize
       \hbox{I}}}

\begin{document}

\title[Okamoto's space for $\PI$ in Boutroux coordinates]{Okamoto's space for the first 
Painlev\'e equation\\ in Boutroux coordinates}
\author{J.J. Duistermaat}\thanks{Johannes (Hans) Jisse Duistermaat passed away unexpectedly in March 2010. He was affiliated with the Mathematical Institute, University of Utrecht, P.O. Box 80.010, 3508 TA Utrecht, The Netherlands (from 1974 until 2010). Between 2004 and 2010, his research was stimulated by a KNAW professorship.}
\author{N. Joshi}\thanks{The collaborative research reported in this paper began at the Isaac Newton Institute for Mathematical Sciences, Cambridge University, UK in May 2009. The first version of this paper was prepared in June 2009. We met again in Leuven, Belgium in July 2009 and continued our collaboration over email. Several successive updates of the paper were prepared by Hans  with the last version being dated February 2010. While the research results reported here remain mostly unchanged from that version, minor errors have been corrected and the paper has been restructured to include an abstract and an introduction. Moreover, the detailed explicit calculations of the resolutions of singularities were moved into an appendix to ease the flow of the exposition. Some informal comments, more befitting a working document between two collaborators, have been removed. Nalini Joshi wishes to record her deep gratitude to Hans for his remarkable, revelatory, collaborative friendship. Her research was supported by the Australian Research Council grant \# DP0985615. }
\address{School of Mathematics and Statistics F07, The University of Sydney, NSW 2006, Australia\\ Tel: +61 2 9351 2172\\ Fax: +61 2 9351 4534}
\email{nalini.joshi@sydney.edu.au}.
\begin{abstract}
We study the completeness and connectedness of asymptotic behaviours of solutions of the first Painlev\'e equation  $\op{d}^2y/\op{d}\! x^2=6\, y^2+ x$, in the limit $x\to\infty$, $x\in\C$. This problem arises in various physical contexts including the critical behaviour near gradient catastrophe for the focusing nonlinear Schr\"odinger equation. We prove that the complex limit set of solutions is non\--empty, compact and invariant under the flow of the limiting autonomous Hamiltonian system, that the infinity set of the vector field is a repellor for the dynamics and obtain new proofs for solutions near the equilibrium points of the autonomous flow. The results rely on a realization of Okamoto's space, i.e., the space of  initial values compactified and regularized by embedding in $\C\Proj 2$ through an explicit construction of nine blow-ups. 
\end{abstract}

\date{}
\maketitle
\large
\section{Introduction}
In this paper, we consider the completeness and connectedness of the asymptotic behaviours of the first Painlev\'e equation
\begin{equation}
\frac{\op{d}^2y}{\op{d}\! x^2}=6\, y^2+ x, 
\label{PI}
\end{equation}
in the limit $x\to\infty$, $x\in\C$. The first Painlev\'e equation arises in many physical contexts, as a reduction of the Korteweg-de Vries equation, in the double scaling limit of random matrix models and in the critical behaviour near the point of \lq\lq gradient catastrophe\rq\rq\ of the solution to the Cauchy problem for the focusing nonlinear Schr\"odinger equation \cite{dub}. 

The asymptotic limit $x\to\infty$ of Equation (\ref{PI}) was first studied in 1913 by Boutroux \cite{boutroux}, who provided a transformation of variables that makes the asymptotic behaviours explicit. It is known that all solutions of (\ref{PI}) are meromorphic in $\C$ with double movable poles, i.e., with locations that change with initial conditions. Boutroux found that locally, in each patch near infinity,  the general solutions are given to leading-order by elliptic functions. More detailed results about how the local asymptotic behaviours of solutions change slowly as $x$ moves near infinity were provided by  Joshi and Kruskal \cite{jk88,jk92}, who constructed a complex multiple-scales method to carry out asymptotic analysis along a large circle in the complex plane for the first and second Painlev\'e equations. Such local behaviours were used in the Riemann-Hilbert method, which was applied to deduce connections between behaviours valid along special directions approaching infinity (see the review by Kitaev \cite{kitaev94}). 

In addition to the two-parameter solutions asymptotic to elliptic-function behaviours, Boutroux identified five one-parameter family of solutions asymptotic to algebraic power expansions in certain sectors of angle $4\pi/5$ in $\C$. He called them {\em tronqu\'ee} or  truncated solutions. In each family of tronqu\'ee solutions, there is a unique solution whose algebraic expansion is valid in a sector of angle $8\pi/5$. Boutroux called these {\em tritronqu\'ee} or triply truncated solutions. In the literature, this term has come to be associated with the unique tritronqu\'ee solution that is real on the real line; each of the other four such solutions can be obtained from this one by a discrete symmetry of Equation (\ref{PI}) corresponding to rotating variables in $\C$. This real tritronqu\'ee solution appears as a distinguished solution in various physical problems (see, e.g., Dubrovin et al. \cite{dub}). In the form $Y_{tt}=6 Y^{2}-t$ (for $y(x)=Y(-t)$), Joshi and Kitaev \cite{jkit} constructed a sequence of solutions that converge to the tritronqu\'ee solution on the whole positive real axis and proved that the tritronqu\'ee solution has no poles whatsoever on the positive real axis. O. and R. Costin \cite{costin2} applied Borel-summation methods to deduce many complex properties of the tronqu\'ee solutions.  

However, while it is known that the solution space of Equation (\ref{PI}) is connected, through Okamoto's \cite{okamoto79} compactification and regularization of the space of initial-values, there has been no investigation (to our knowledge) of the completeness or connectedness of the known asymptotic behaviours of the solutions of this equation (or of any of the six Painlev\'e equations). We tackle this problem by undertaking asymptotic analysis in Okamoto's  space. Our approach relies on explicit resolution of singularities in an asymptotic version of this space. We note that although we focus on the first Painlev\'e equation in this paper, our approach can be extended to the other Painlev\'e equations.

\subsection{Outline of the paper}
The paper is organized as follows. In \S\ref{boutrouxsec}, we recast Equation (\ref{PI}) as a Hamiltonian system, provide a rescaling of it under Boutroux's transformation of variables, and summarize known properties of solutions. The resolution of singularities of this Boutroux-Painlev\'e system is provided explicitly in Appendix A where we carry out the sequence of changes of variables necessary first to compactify and then blow up the nine base points of the system in $\C\Proj 2$. The construction shows that the vector field is infinite on the union $I:=\cup_{i=0}^{8} L_{j}$ of nine complex projective lines. 

We show in \S\ref{polesec} that the Boutroux-Painlev\'e vector field is regular and transverse to the last complex projective line, $L_{9}$, in $S_{9}$. It is shown here that the Taylor expansion of the flow around a point on this line provides us with the Laurent expansion of the solutions $y(x)$ near a pole. In \S\ref{repelsec}, we consider the vector field near the infinity set $I$ and show that this is a repellor for the flow. We also construct the limit set for each solution and show that it is a non-empty, compact and connected subset of $S_{9}$, which remains invariant under the autonomous flow. As a corollary, we prove that every solution of Equation (\ref{PI}) must have an infinite number of poles in the complex plane. Finally, in \S\ref{eqsec}, we consider the Boutroux-Painlev\'e system near the equilibria of the autonomous limit system and prove several results about {\em tronqu\'ee} solutions, ending with a determination of their sequence of poles near the boundaries of pole-free sectors, by using classical methods. 

\subsection{Notation}\label{notation}
The Painlev\'e flow is denoted $(y_{1}, y_{2})$, where $y=y_1$, $\op{d}\! y/\op{d}\! x=y_2$. In Boutroux's coordinates, this flow is transformed to $(u_{1}, u_{2})$, where $u_1(z)=x^{-1/2}\, y_1(x)$, $u_2(z)=x^{-3/4}\, y_2(x)$, with $(4/5)\, x^{5/4}=z$. We embed the resulting Painlev\'e\--Boutroux flow into $\C\Proj 2$  by identifying the affine coordinates $(u_{1}, u_{2})$ with homogeneous coordinates as follows
\begin{eqnarray*}
 [1:u_{1}:u_{2}]&=&[u_1^{-1}:1:u_1^{-1}\,u_2]=:[u_{021}:1:u_{022}]\\
 &=&[u_2^{-1}:u_1\,u_2^{-1}:1]=:[u_{031}:u_{032}:1]
 \end{eqnarray*}
The line at infinity given respectively by $u_{021}=0$, $u_{031}=0$ is denoted by $L_0$.  For $0\le i \le 8$, corresponding to the $i$-th stage of the regularizing (blow-up) sequence, we denote a base point by $b_{i}$, the exceptional line attached to that base point by $L_{i+1}$ and the coordinates in the $j$-th chart of the $i$-th blowup by $(u_{ij1}, u_{ij2})$. In each coordinate chart, the Jacobian of the coordinate change from $(u_1, u_2)$ to $(u_{ij1}, u_{ij2})$ is denoted by
\begin{equation*}
 w_{ij}=\frac{\partial u_{ij1}}{\partial u_1}\,\frac{\partial u_{ij2}}{\partial u_2}-\frac{\partial u_{ij1}}{\partial u_2}\,\frac{\partial u_{ij2}}{\partial u_1}.
\end{equation*}
The last space $S_{9}$ constructed by this sequence of blow-ups is Okamoto's \lq\lq space of initial values.\rq\rq\ The regularization of the limiting autonomous system obtained as $z\to\infty$ is the same as that of the related Weierstrass cubic curve. We add the superscript ${ell}$ to the above notation for the base points and lines and a superscript $0$ to the vector field to refer to the corresponding objects in the resulting anti-canonical pencil.
\section{Boutroux scaling}
\label{boutrouxsec}
Equation (\ref{PI}) can be viewed, upon the substitutions 
$y=y_1$, $\op{d}\! y/\op{d}\! x=y_2$, as a Hamiltonian system 
$\op{d}\! y_1/\op{d}\, x=\partial H/\partial y_2$, 
$\op{d}\! y_2/\op{d}\! x=\, -\partial H/\partial y_1$ with an 
$x$\--dependent Hamiltonian function 
\begin{equation}
H=H(x,\, y_1,\, y_2):={y_2}^2/2-2\, {y_1}^3-x\, y_1.
\label{HPI}
\end{equation}
The function $H$    
is a weighted homogeneous polynomial 
in the sense that if we substitute 
$x=\lambda ^4\,\xi$, $y_1=\lambda ^2\, u_1$, and 
$y_2=\lambda ^3\, u_2$, then 
$H=\lambda ^6\, ({u_2}^2/2-2\, {u_1}^3-\xi\, u_1)$. 
We have $\xi =1$ if and only if $x=\lambda ^4$, when 
$\lambda =x^{1/4}$, $y_1=x^{1/2}\, u_1$, 
$u_1=x^{-1/2}\, y_1$, $y_2=x^{3/4}\, u_2$, 
and $u_2=x^{-3/4}\, y_2$. If $x=x(z)$ and a dot means 
differentiation with respect to $z$, then 
$\dot{u}_1=\dot{x}\, (-(1/2)\, x^{-1}\, u_1+x^{1/4}\, u_2)$. 
If we choose $(4/5)\, x^{5/4}=z$ then $\dot{x}\, x^{1/4}=1$, 
and the Painlev\'e system takes the form  
\begin{equation}
\begin{array}{lll}
\dot{u}_1&=&u_2-2\, (5\, z)^{-1}\, u_1,\\
\dot{u}_2&=&6\, {u_1}^2+1-3\, (5\, z)^{-1}\, u_2.
\end{array}
\label{u01dot}
\end{equation}
This is an order $z^{-1}$ perturbation of the 
Hamiltonian system with Hamiltonian function 
equal to the $z$\--independent {\em energy function} 
\begin{equation}
\begin{array}{llll}
&E&:=&{u_2}^2/2-2\, {u_1}^3-u_1,\\
\mbox{\rm where}\quad&&&\\
&\dot{E}&=&(5\, z)^{-1}
\, (2\, u_1+12\, {u_1}^3-3\, {u_2}^2)
=\, -(5\, z)^{-1}\, (6\, E+4\, u_1).
\end{array} 
\label{qeq}
\end{equation}
It implies Boutroux's second order differential equation 
\begin{equation}
\ddot{u}_1=6\, {u_1}^2+1-z^{-1}\,\dot {u}_1+\, 4\, (5\, z)^{-2}\, u_1
\label{boutrouxeq}
\end{equation}
for $u_1$. These transformations have been used by 
Boutroux \cite{boutroux} in order to investigate 
the asymptotic behavior of the solutions of the 
Painlev\'e equation when $x\to\infty$. 

Because $\int_{z_0}^{\infty}\, z^{-1}\,\op{d}\! z=\infty$, 
we cannot straightforwardly conclude that solutions 
of the Boutroux system (\ref{u01dot}) converge to 
solutions of the autonomous Hamiltonian system when $z\to\infty$.  
Actually they don't: we will see that each solution 
of (\ref{u01dot}) converges to different solutions 
of the autonomous limit system, depending on the 
path along which $z$ runs to infinity. 

\begin{remark}
If $a$ and $b$ are nonzero complex constant complex numbers, then 
the substitutions $y(x)=a\, \eta (\xi )$, $\xi =b\, x$ turn 
(\ref{PI}) into the 
differential equation $\op{d}^2\eta /\op{d}\!\xi ^2
=\alpha\,\eta ^2+\beta\,\xi$, where 
$\alpha =6\, a/b^2$ and $\beta =1/a\, b^3$, 
or equivalently $b=(6/\alpha\,\beta )^{1/5}$ and  
$a=\alpha\, b^2/6$. 

Boutroux \cite[p. 311]{boutroux} 
took the first Painlev\'e equation with the constants 
$\alpha =6$ and $\beta =\, -6$, 
or equivalently $b=(-1/6)^{1/5}$ and $a=b^2$, 
as his point of departure, and applied the 
substitutions $X=(4/5)\,\xi ^{5/4}$ and $\eta =\xi ^{1/2}\, Y$ 
in order to arrive at the differential equation 
$\op{d}^2Y/\op{d}\! X^2=6\, Y^6-6-(1/X)\,\op{d}\! Y/\op{d}\! X
+(4/(5\, X)^2)\, Y$. Therefore the translation from 
Boutroux's notation to ours is $X=(-1/6)^{1/4}\, z$,  
$Y=(-1/6)^{-1/2}\, u=(-1/6)^{-1/2}\, u_1$, and  
$Y'=(-1/6)^{-3/4}\, \dot{u}=(-1/6)^{-3/4}\, 
(u_2-2\, (5\, z)^{-1}\, u_1)$. 
An expression which plays a central role in Boutroux 
\cite[\S 7--11]{boutroux} is  
\[
(Y')^2-4\, Y^3+12\, Y
=(-1/6)^{-3/2}\, (\dot{u}^2-\, 4\, u^3-2\, u).
\]

Joshi and Kruskal \cite{jk88}, \cite{jk92} took the 
first Painlev\'e equation with the constants 
$\alpha =3/2$ and $\beta =\, -3/2$, or equivalently 
$b=(-8/3)^{1/5}$ and $a=b^2/4$, as their point of departure, 
and applied the substitutions 
$Z=(4/5)\,\xi ^{5/4}$ and $\eta =\xi ^{1/2}\, U$, 
where they actually wrote $z$ and $u$ instead of $Z$ and $U$, 
respectively. Therefore the translation from their $z$ and $u$  
to ours is $Z=(-8/3)^{1/4}\, z$,  
$U=4\, (-8/3)^{-1/2}\, u=\,\pm\op{i}\,\sqrt{6}\, u_1$, and  
$\op{d}\! U/\op{d}\! Z=4\, (-8/3)^{-3/4}\,\dot{u}
=4\, (-8/3)^{-4/3}\, (u_2-2\, (5\, z)^{-1}\, u_1)$. 
A central role is played in \cite{jk88}, \cite{jk92} by the function   
\[
{\mathcal E}:=((\op{d}\! U/\op{d}\! Z)^2-U^3+3\, U)/2
=2^{-1/2}\, (-3)^{3/2}\, (\dot{u}^2/2-2\, u^3-u). 
\]   

The functions $(Y')^2-4\, Y^3+12\, Y$ 
and ${\mathcal E}$ are closely related to the energy function 
$E$ in (\ref{qeq}), as 
\begin{equation}
\dot{u}^2/2-2\, u^3-u
=E-2\, (5\, z)^{-1}\, u_1\, u_2+2\, (5\, z)^{-2}\, {u_1}^2.
\label{qeqdotu}
\end{equation}
\label{boutrouxjkconstantsrem}
\end{remark}

\begin{remark}
The Boutroux substitutions $x=\bigl((5/4)\, z\bigr)^{4/5}$ with 
inverse $z=(4/5)\, x^{5/4}$, and $y(x)=x^{1/2}\, u(z)$ $=x^{1/2}\, u\bigl((4/5)\, x^{5/4}\bigr)$ $=\bigl((5/4)\, z\bigr)^{2/5}\, u(z)$ with inverse 
$u(z)=x^{-1/2}\, y(x)$ are singular at $x=0$ and 
correspondingly $z=0$. These substitutions introduce 
multi\--valuedness of the solutions $u(z)$ of the 
Boutroux\--Painlev\'e equation (\ref{boutrouxeq}) when 
$z$ runs around the origin in the complex plane, 
where the solutions $y(x)$ of the Painlev\'e equation 
(\ref{PI}) 
are single\--valued. 

More precisely, every local solution $y(x)$ of (\ref{PI}) 
extends to a single\--valued 
meromorphic function 
on the whole complex $x$\--plane, where the poles 
are of order two and have leading coefficient equal to 1. 
This is the {\em Painlev\'e property} in its strongest form; 
see \cite[Remark 1.1]{d} 
for some remarks on its proofs in the literature. 

The equation $u(z)
=((5/4)\, z)^{-2/5}\, y\bigl(\bigl((5/4)\, z\bigr)^{4/5}\bigr)$, 
in combination with the single\--valuedness of $y(x)$, 
implies that the analytic continuation of $u(z)$ 
along the path $z\,\op{e}^{\,\scriptop{i}\,\theta }$,  
avoiding the poles of $u(z)$,  
returns to its opposite if $\theta\in\R$ runs from 
$0$ to $5/4$ times $2\pi$. This may be expressed by the formula 
\begin{equation}
u_1(z\, e^{5\pi\scriptop{i}/2})=\, -\,u_1(z),
\quad u_2(z\, e^{5\pi\scriptop{i}/2})=\op{i}u_2(z), 
\label{uangle5/4} 
\end{equation}
where the second equation follows from the first, 
in view of the first equation in (\ref{u01dot}). 
This observation has been used in Joshi and Kruskal 
\cite[Sec. 5]{jk88} as a consistency check for their 
asymptotic results for $u(z)$ for large $|z|$. 
\label{painleveboutrouxrem}
\end{remark} 

\begin{remark}
Each solution $y(x)$ of the Painlev\'e equation (\ref{PI}) 
has a convergent Laurent expansion 
\[
y(x)=\sum_{n=n_0}^{\infty}\, y_n\, x^n
\] 
for $0<|x|<<1$, when  
$u(z)=\bigl((5/4)\, z\bigr)^{-2/5}\, y(x)
=((5/4)\, z)^{-2/5}\, y(((5/4)\, z)^{4/5})$ implies 
the convergent power series 
\[
u(z)=\sum_{n=n_0}^{\infty}\, y_n\, \bigl((5/4)\, z\bigr)^{(-2+ 4\, n)/5}
\] 
for $0<|z|<<1$. We have the following cases. 
\begin{itemize}
\item[i)] $y(0)$ is finite, when $n_0=0$. 
The Painlev\'e equation 
$y''=6\, y^2+x$ is equivalent to the recursive equations 
\begin{eqnarray}
n\, (n-1)\, y_n&=&6\,\sum_{m=0}^{n-2}\, y_{n-2-m}\, y_m
\quad\mbox{\rm for}\quad n\geq 2,\; n\neq 3,
\label{yn}\\
y_3&=&2\, y_0\, y_1+1/6
\label{y3}
\end{eqnarray}
for the coefficients $y_n$. The mapping which assigns to 
the solution $y(x)$ the complex numbers $y_0=y(0)$ and $y_1=y'(0)$ 
is bijective from the set of all regular solutions $y(x)$ 
near $x=0$ onto $\C ^2$. 
Subcases: 
\begin{itemize}
\item[ia)] $y_0=y_1=0$, when (\ref{yn}) for $n=2$ yields that 
$y_2=0$, whereas (\ref{y3}) implies that $y_3=1/6$. 
An induction on $n$ yields that 
$y_n=0$ unless $n\in 3+5\,\Z$, as 
$n-2-m=3+5\, k$ and $m=3+5\, l$ imply that $n=3+5\, (k+l+1)$. 
Because $-2+4\, (3+5\, j)=10\, (2\, j+1)$, 
it follows that $u(z)=\sum_{j=0}^{\infty}\, y_{3+5\, j}\, 
((5/4)\, z)^{2\, (2\, j+1)}$. In particular this solution 
$u(z)$ is single\--valued. 
\item[ib)] $y_0=0$ and $y_1\neq 0$, 
when $u(z)=y_1\, ((5/4)\, z)^{2/5}+\op{O}(z^2)$ as $z\to 0$. 
\item[ic)] $y_0\neq 0$, when 
$u(z)=y_0\, ((5/4)\, z)^{-2/5}+\op{O}(z^{2/5})$ as $z\to 0$. 
\end{itemize}
\item[ii)] $y(x)$ has a pole at $x=0$, when 
$n_0=\, -2$, $y_{-2}=1$, $y_{-1}=y_0=y_1=y_2=0$, and  
$y_3=\, -1/6$. The mapping which assigns to the solution $y(x)$ 
the coefficient $y_4$ is bijective from the set of solutions 
with a pole at $x=0$ onto $\C$, see for instance 
\cite[the text following (11.3)]{d}. For $n\geq 5$ 
the Painlev\'e equation $y''=6\, y^2+x$ implies the recursive equations 
\begin{equation}
(n\, (n-1)-12)\, y_n=6\,\sum_{m=3}^{n-2}\, y_{n-2-m}\, y_m. 
\label{ynpole}
\end{equation}
Subcases: 
\begin{itemize}
\item[iia)] $y_4=0$. As in ia) it follows from 
(\ref{ynpole}) by induction on $n$ that $y_n=0$ unless $n\in 3+5\,\Z$, 
and it follows that $u(z)=\sum_{j=\, -1}^{\infty}\, y_{3+5\, j}\, 
((5/4)\, z)^{2\, (2\, j+1)}$. In particular this solution 
$u(z)$ is single\--valued. 
\item[iib)] $y_4\neq 0$, when 
$u(z)=((5/4)\, z)^{-2}
-(1/6)\, ((5/4)\, z)^{2}
+ y_4\, ((5/4)\, z)^{14/5}
+\op{O}(z^{18/5})$ as $z\to 0$. 
\end{itemize}
\end{itemize}
It follows that the solution $u(z)$ of the Boutroux\--Painlev\'e 
equation (\ref{boutrouxeq}) is not single\--valued, unless 
we are in the cases ia) or iia). 
That is, the solution $y(x)$ of (\ref{PI}) 
is either equal to the unique regular solution near $x=0$ 
for which $y(0)=y'(0)=0$, or $y(x)$ is the unique 
solution of (\ref{PI}) with a pole at 
$x=0$ such that $y_4=0$. 

If $y(x)$ is a solution of (\ref{PI}), and 
$a\in\C$ is a fifth root of unity, that is, $a^5=1$, 
then $x\mapsto a^{-1}\, y(a^2\, x)$ is also a solution 
of (\ref{PI}). The solutions 
$y(x)$ in ia) and iia), corresponding to 
the single\--valued solutions $u(z)$ of (\ref{boutrouxeq}), are exactly 
the solutions which are invariant under this 
five\--fold symmetry, that is, which satisfy 
$y(x)=a^{-1}\, y(a^2\, x)$ for every fifth root of 
unity $a$, as this means that in the Laurent expansion 
of $y(x)$ only the powers $x^j$ appear such that 
$2\, j-1\in 5\,\Z$ $\Leftrightarrow$ $2\, j-1\in 
5\, (2\,\Z +1)$ $\Leftrightarrow$ $j\in 5\, Z+3$.  
The solutions $y(x)$ in ia) and iia) appear in  
Boutroux \cite[p. 336, 337]{boutroux}. 
\label{intsymrem}
\end{remark}

\begin{remark}
If $y(x)$ is a solution of the  
Painlev\'e equation (\ref{PI}), and 
$a\in\C$ is a fifth root of unity, that is, $a^5=1$, 
then $x\mapsto a^{-1}\, y(a^2\, x)$ is also a solution 
of (\ref{PI}). 

The solutions 
$y(x)$ in ia) and iia) of Remark \ref{intsymrem}, corresponding to 
the single\--valued solutions $u(z)$ of the 
Boutroux\--Painlev\'e equation (\ref{boutrouxeq}), are exactly 
the solutions which are invariant under this 
five\--fold symmetry, that is which satisfy 
$y(x)=a^{-1}\, y(a^2\, x)$ for every fifth root of 
unity $a$, as this means that in the Laurent expansion 
of $y(x)$ only the powers $x^j$ appear such that 
$2\, j-1\in 5\,\Z$ $\Leftrightarrow$ $2\, j-1\in 
5\, (2\,\Z +1)$ $\Leftrightarrow$ $j\in 5\, Z+3$.  
This explains why Boutroux \cite[p. 336, 337]{boutroux} 
called the solutions $y(x)$ in ia) and iia) 
of Remark \ref{intsymrem} the 
{\em symmetric solutions}. 
\label{symrem}
\end{remark}

\section{The poles in Okamoto's space}
\label{polesec}
In this section, we consider the Boutroux\--Painlev\'e vector field in Okamoto's space $S_{9}$ constructed explicitly in Appendix A. This construction shows that the vector field has no base points 
in $S_9$, is infinite along 
the configuration $I:=\bigcup_{i=0}^8\, L_i^{(9-i)}$ 
of nine complex projective lines, and regular 
in $S_9\setminus I$. For this reason we will call the set 
$I$ the {\em infinity set} of the vector field.  

The change of coordinates to the coordinate chart $(u_{911},\, u_{912})$ is shown in Appendix A to be
\begin{eqnarray}
\nonumber u_{911}(z)&=&{u_1(z)}^{-9}\, u_2(z)\, 
(-32\, {u_1(z)}^7\, u_2(z)-4\, {u_1(z)}^3\, {u_2(z)}^5
+{u_2(z)}^7+256\, (5\, z)^{-1}\, {u_1(z)}^8),\\
\nonumber u_{912}(z)&=&u_1(z)\, {u_2(z)}^{-1},\\ 
u_{1}(z)&=&u_{912}(z)^{-2}\, 
(4+32\, u_{912}(z)^4+u_{911}(z)\, u_{912}(z)^6
-256\, (5\, z)^{-1}\, u_{912}(z)^5)^{-1}
\label{uu91eq}\\
\nonumber u_2(z)&=&{u_{912}(z)}^{-3}\, 
(4+32\, {u_{912}(z)}^4+u_{911}(z)\, {u_{912}(z)}^6
-256\, (5\, z)^{-1}\, {u_{912}(z)}^5)^{-1}
\end{eqnarray}

The set of points in $S_9\setminus I$ which project to 
$L_0$, the set where $(u_1,\, u_2)$ is infinite, 
is equal to $L_9\setminus I$. Because 
$L_9\cap I=L_9\cap L_8^{(1)}$ 
consists of one point, $L_9\setminus I$ is 
isomorphic to the affine complex plane. 
The regular vector field in $S_9\setminus I$ 
is nonzero at and 
transversal to $L_9\setminus I$. 
A solution crosses the complex line $L_9\setminus L_8^{(1)}$ 
at the time $z=\zeta$, if and only if 
$u(z)=u_1(z)$ becomes infinite as $z\to\zeta$. 
The whole set $L_9\setminus I$  is visible in the 
coordinate chart $(u_{911},\, u_{912})$, where it 
is the line $u_{912}=0$, parametrized by $u_{911}\in\C$. 
Because $u(z)=u_{1}(z)$ given by Equation (\ref{uu91eq})
is a rational expression in $z$, $u_{911}(z)$, and $u_{922}(z)$, 
and the solution $z\mapsto (u_{911}(z),\, u_{912}(z))$ 
of the regular non\--autonomous system is a complex analytic function in 
a neighborhood of $z=\zeta$ with 
$u_{912}(\zeta )=0$ and $a:=u_{911}(\zeta )\in\C$, 
it follows that the solution $z\mapsto u(z)$ of the 
Boutroux\--Painlev\'e equation is a meromorphic function 
in a neighborhood of $z=\zeta$, with a pole of order two. 
For this reason the line $L_9\setminus I$ is called 
the {\em pole line}. 

Consider the vector field $(\dot{u}_{911}, \dot{u}_{912})$ which is given in Section \ref{lastblowup} of Appendix A. We recall it here for ease of reference by the reader:
\begin{eqnarray*}
\dot{u}_{911}&=&(4+32\, {u_{912}}^4+u_{911}\, {u_{912}}^6
-256\, (5\, z)^{-1}\, {u_{912}}^5)^{-1}\\
&&\times\, 
[u_{912}\, 
(-2^{11}-2^6\cdot 5\, u_{911}\, {u_{912}}^2
+2^{13}\cdot 7\, {u_{912}}^4
-3^2\, {u_{911}}^2\, {u_{912}}^4\\
&&+2^{12}\, u_{911}\, {u_{912}}^6
+2^{16}\cdot 3\, {u_{912}}^8
+2^3\cdot 3^2\, {u_{911}}^2\, {u_{912}}^8\\
&&+2^{12}\cdot 5\, u_{911}\, {u_{912}}^{10}
+2^6\cdot 11\, {u_{911}}^2\, {u_{912}}^{12}
+2^3\, {u_{911}}^3\, {u_{912}}^{14})\\
&&-2\, (5\, z)^{-1}\, 
(2^2\cdot 3\, u_{911}
-2^{12}\cdot 3^2\, {u_{912}}^2
-2^5\cdot 3^2\cdot 7\, u_{911}\, {u_{912}}^4\\
&&+2^{15}\cdot 3\cdot 5\, {u_{912}}^6
+3\, {u_{911}}^2\, {u_{912}}^6
+2^{10}\cdot 17\, u_{911}\, {u_{912}}^8\\
&&+2^{17}\cdot 19\, {u_{912}}^{10}
+2^{13}\cdot 3\cdot 7\, u_{911}\, {u_{912}}^{12}
+2^7\cdot 23\, {u_{911}}^2\, {u_{912}}^{14})\\
&&+2^9\, (5\, z)^{-2}\, {u_{912}}^3
\, (-2^6\cdot 3\cdot 5
+3\, u_{911}\, {u_{912}}^2\\
&&+2^{13}\, {u_{912}}^4
+2^{14}\cdot 5\, {u_{912}}^8
+2^8\cdot 11\, u_{911}\, {u_{912}}^{10})\\
&&-2^{24}\cdot 7\, (5\, z)^{-3}\, {u_{912}}^{12}],\\
\dot{u}_{912}&=&-(4+32\, {u_{912}}^4+u_{911}\, {u_{912}}^6
-256\, (5\, z)^{-1}\, {u_{912}}^5)^{-1}\\
&&\times\, [2-2^4\, {u_{912}}^4
-u_{911}\, {u_{912}}^6+2^8\, {u_{912}}^8
+2^3\, u_{911}\, {u_{912}}^{10}\\
&&+2^{10}\, {u_{912}}^{12}
+2^6\, u_{911}\, {u_{912}}^{14}
+{u_{911}}^2\, {u_{912}}^{16}\\
&&-(5\, z)^{-1}\, u_{912}\, 
(2^2-2^5\cdot 7\, {u_{912}}^4
+u_{911}\, {u_{912}}^6\\
&&+2^{11}\, {u_{912}}^8
+2^{14}\, {u_{912}}^{12}
+2^9\, u_{911}\, {u_{912}}^{14})\\
&&+2^8\, (5\, z)^{-2}\, {u_{912}}^6\, (1+2^8\, {u_{912}}^8)].
\end{eqnarray*}

It follows from the equation for $\dot{u}_{912}$ that the 
coefficients of $(z-\zeta )^i$ in the Taylor expansion at $z=\zeta$ of 
$u_{912}(z)$ do not depend on $a$ for $1\leq i\leq 6$, when 
(\ref{uu91eq}) shows that the coefficients of 
$(z-\zeta )^j$ in the Laurent expansion at $z=\zeta$ 
of $u(z)$ do not depend on $a$ for $-2\leq j\leq 3$. 
Substitution of the Taylor expansion at $z=\zeta$ of order 
$i$ of $u_{912}(z)$ in the formula for $\dot{u}_{912}$ 
yields the Taylor expansion at $z=\zeta$ of order $i$ 
of $\dot{u}_{912}(z)$ of order $i$, hence the Taylor expansion 
at $z=\zeta$ of order $i+1$ of $u_{912}(z)$, as 
long as $i\leq 5$. Then substitution of $u_{911}(\zeta )=a$ in the 
formula for $\dot{u}_{912}$ yields the coefficient 
of $(z-\zeta )^6$ in the Taylor expansion at $z=\zeta$ of 
$\dot{u}_{912}(z)$, hence of $(z-\zeta )^7$ in the 
Taylor expansion at $z=\zeta$ of $u_{912}(z)$, 
when (\ref{uu91eq}) yields the coefficients of 
$(z-\zeta )^j$ in the Laurent expansion 
at $z=\zeta$ of $u(z)$ for $-2\leq j\leq 4$. 
This yields  
\begin{eqnarray}
u_{911}(z)&=&a+\op{O}(z-\zeta ),\nonumber\\
u_{912}(z)&=&-\frac12\, (z-\zeta )
-\frac{1}{2^2\cdot 5\cdot\zeta}\, (z-\zeta )^2
+\frac{3}{2^2\cdot 5^2\cdot\zeta ^2}\, (z-\zeta )^3
-\frac{3\cdot 7}{2^3\cdot 5^3\cdot\zeta ^3}\, (z-\zeta )^4
\nonumber\\
&&+\left(\frac{3\cdot 7\cdot 19}{2^3\cdot 5^5\cdot\zeta ^4}
+\frac{1}{2\cdot 5}\right)\, (z-\zeta )^5
-\left(\frac{3\cdot 7\cdot 19}{2\cdot 5^6\cdot\zeta ^5}
-\frac{41}{2^2\cdot 3\cdot 5^2\cdot\zeta}\right)\, (z-\zeta )^6
\nonumber\\
&&+\left(\frac{3\cdot19\cdot29}{2\cdot 5^7\cdot\zeta ^6}
-\frac{41}{2\cdot 3\cdot 5^3\cdot\zeta ^2}
+\frac{3\, a}{2^9\cdot 7}
\right)\, (z-\zeta )^7+\op{O}((z-\zeta )^8),
\nonumber
\end{eqnarray}
\begin{eqnarray}
u(z)&=&(z-\zeta )^{-2}-\frac{1}{5\cdot\zeta}\, (z-\zeta )^{-1}
+\frac{3}{2^2\cdot 5\cdot\zeta ^2}
-\frac{31}{2\cdot 5^3\cdot\zeta ^3}\, (z-\zeta )
\nonumber\\
&&+\left(\frac{19\cdot 283}{2^4\cdot 5^5\cdot\zeta ^4}-\frac{1}{2\cdot 5}\right)
\, (z-\zeta )^2
-\left(\frac{3\cdot 11\cdot 727}{2^4\cdot 5^6\cdot \zeta ^5}
+\frac{11}{2\cdot 4\cdot 5^2\cdot\zeta}\right)
\, (z-\zeta )^3
\nonumber\\
&&+\left(\frac{197\cdot 443}{2^6\cdot 5^6\cdot\zeta ^6}
+\frac{29}{2^3\cdot 3\cdot 5^2\cdot\zeta ^2}
-\frac{a}{2^8\cdot 7}\right)\, (z-\zeta )^4+\op{O}((z-\zeta )^5).  
\label{laurentu}
\end{eqnarray}

The anticanonical pencil has a base point at 
the point $^{(1)}b_8^{\,\scriptop{ell}}$ 
determined by the equations 
$u_{921}=256\, (5\, z)^{-1}$, $u_{922}=0$ which is the lift  
to $S_9$ of the point $b_8^{\,\scriptop{ell}}$ 
of the anticanonical pencil in $S_8$. 
The blowing up of $S_8$ in 
the point $b_8$, which is not the base point of the 
anticanonical pencil, causes $E\, w_{92}$ to be 
infinite along 
$L_9\setminus L_8^{(1)}$, the line determined by the equation $u_{921}=0$. 
In turn this forces the energy function $E(z)$ to have a pole 
at the point $z=\zeta$ where the solution $u(z)$ of the 
Boutroux\--Painlev\'e equation has a pole. 

The equations for $w_{91}$, $E\, w_{91}$, and 
$\dot{E}\, w_{91}$ (given in Section \ref{lastblowup}) imply in combination with the Taylor expansion 
for $u_{911}(z)$ and $u_{912}(z)$ in (\ref{laurentu}) that 
$E=\, -2^2\, (5\,\zeta )^{-1}\, (z-\zeta )^{-1}
+2^{-7}\, a-22\, (5\,\zeta )^{-2}+\op{O}(z-\zeta )$ 
and $\dot{E}=\, -2^2\, (5\,\zeta )^{-1}\, 
(z-\zeta )^{-2}\, (1+\op{O}((z-\zeta )^2))$. 
Combination of these asymptotic expansions for $E$ and $\dot{E}$ 
leads to 
\begin{equation}
E(z)=\, -2^2\, (5\,\zeta )^{-1}\, (z-\zeta )^{-1}
+2^{-7}\, a-22\, (5\,\zeta )^{-2}+\op{O}((z-\zeta )/\zeta ), 
\label{qexpansion}
\end{equation}
where the remainder term is uniform for bounded $\zeta ^{-1}$ 
and $a$. It follows that the energy $E(z)$, although it has a pole 
of order one at $z=\zeta$, is close to 
$2^{-7}\, a$ if $|z-\zeta |$ is large compared to $1/|\zeta |$, 
and $E(z)=2^{-7}\, a+\op{O}(\zeta ^{-1})$ if $z-\zeta$ and 
$(z-\zeta )^{-1}$ are bounded. That is, for large $|z|$, 
$E(z)$ is well approximated 
by $2^{-7}\, a$ as soon as $z$ leaves the disc centered 
at $z=\zeta$ with radius of small order $1/|\zeta |$, 
where the approximation improves when $|z-\zeta |$ 
increases to order one. 

\section{The solutions near the set where the vector field is infinite}
\label{repelsec}
In this section, we consider the vector field near the infinity set $I$ and show that it is repelling for the flow. We also construct the limit set for each solution and show that it is a non-empty, compact and connected subset of $S_{9}$ that remains invariant under the autonomous flow. As a corollary, we prove that every solution of Equation (\ref{PI}) must have an infinite number of poles in the complex plane. 

Let ${\mathcal S}$ denote the fiber bundle of the surfaces 
$S_9=S_9(z)$, $z\in\C\setminus\{ 0\}$, in which the 
time\--dependent Painlev\'e vector field $v_z$, in the Boutroux 
scaling, defines a regular (= holomorphic) 
one\--dimensional vector subbundle ${\mathcal P}$ of ${\mathcal S}$. 
For each $z\in\C\setminus\{ 0\}$, let    
$I(z):=\bigcup_{i=0}^8\, L_i^{(9-i)}(z)$ 
be the {\em infinity set}, the set of all points in $S_9(z)$ 
where $v_z$ is infinite, that is, where ${\mathcal P}$ 
is ``vertical'' (or  tangent to the fiber). 
If ${\mathcal I}$ denotes the union in ${\mathcal S}$ 
of all $I(z)$, $z\in\C\setminus\{ 0\}$, then 
${\mathcal S}\setminus {\mathcal I}$ is Okamoto's \lq\lq space  
of initial conditions\rq\rq , fibered by the surfaces 
$S_9(z)\setminus I(z)$, the open subset of 
${\mathcal S}$ of all points in ${\mathcal S}$ where ${\mathcal P}$ 
is transversal to the fibers, and therefore defines a 
regular infinitesimal connection in the bundle 
of the $S_9(z)\setminus I(z)$, $z\in\C\setminus\{ 0\}$. 
Instead of using the coordinate\--invariant description 
of a bundle of surfaces with a connection, we will 
analyse the asymptotic behavior, for $|z|\to\infty$, 
of the solutions of the Painlev\'e equation in the Boutroux scaling, 
by studying the $z$\--dependent vector field in the coordinate systems 
introduced in Section \ref{boutrouxsec}. The solution curve 
in ${\mathcal S}$ will we denoted by $\gamma =\gamma (z)$, 
whereas the corresponding solution of the Boutroux\--Painlev\'e 
differential equation is denoted by $u(z)$. 
Note that $u(z)$ is equal to the first coordinate $u_1(z)$ 
of $\gamma (z)$ in the $(u_1,\, u_2)$ coordinate system, 
the $01$\--coordinate system.  

In this section we begin with an asymptotic description 
of the solutions near the locus ${\mathcal I}$ 
where the vector field is infinite. In the notation we 
often drop the dependence on $z$ of the surfaces $S_9(z)$.  
All order estimates will be uniform in $z$ for $z$ bounded away 
from zero. Near the part $I\setminus L_8^{(1)}
=\bigcup_{i=0}^6\, L_i^{(9-i)}\cup 
(L_7^{(2)}\setminus L_8^{(1)})$ 
of $I$ we will use the function 
$1/E$, where $E$ is the energy, as an indicator for the distance 
to $I$, whereas near the remaining 
part $L_8^{(1)}$ of $I$ we switch to 
$w_{92}$ in the $92$\--coordinate system. See the first statement in 
Lemma \ref{E8E4simlem}. The function 
$1/E$ is no longer useful as an indicator function 
near $L_8^{(1)}$ because $L_8^{(1)}$ 
contains the lift $(b_8^{\,\scriptop{ell}})^{(1)}$ 
to $S_9$ of the base point $b_8^{\,\scriptop{ell}}$ 
of the anticanonical pencil, and 
$E$ takes all finite values near $(b_8^{\,\scriptop{ell}})^{(1)}$. 
One of the points of the proof is that $w_{92}$ is approximately constant 
when the solution runs closely along $L_8^{(1)}$.  
\begin{lemma}
Let 
\[
I^6:=\bigcup_{i=0}^6\, L_i^{(9-i)}.
\]
For every $\epsilon >0$ there exists a neighborhood 
$U$ of $I^6$ in $S_9$ such that 
$|(\dot{E}/E)/(-6/5\, z)-1|<\epsilon$ in $U$ and for all 
$z\in\C\setminus\{ 0\}$. For every compact subset 
$K$ of $L_7^{(2)}\setminus L_8^{(1)}$ there exists a 
neighborhood $V$ of $K$ in $S_9$ and a constant $C>0$ such that 
$|(\dot{E}/E)\, z|\leq C$ in $V$  and for all 
$z\in\C\setminus\{ 0\}$. 
\label{dotq/qlem}
\end{lemma}
\begin{proof}
Because $I^6$ is compact, it suffices to prove that 
every point of it has a neighborhood in $S_9$ in which the 
estimate holds. 
The quantity $r:=(5\, z\, (\dot{E}/E)+6)/8
=\, -2\, u_1/E$ is equal to 
\begin{eqnarray*}
r_{02}&=&{u_{021}}^2/(4+2\, {u_{021}}^2-u_{021}\, {u_{022}}^2),\\
r_{03}&=&{u_{031}}^2\, u_{032}
/(-u_{031}+2\, {u_{031}}^2\, u_{032}+4\, {u_{032}}^3),\\
r_{11}&=&{u_{111}}^2\, {u_{112}}^2
/(-u_{111}+4\, {u_{112}}^2+2\, {u_{111}}^2\, {u_{112}}^2),\\
r_{12}&=&{u_{121}}^2\, u_{122}
/(-1+2\, {u_{121}}^2\, u_{122}+4\, {u_{121}}^2\, {u_{122}}^3),\\
r_{21}&=&{u_{211}}^2\, {u_{212}}^3
/(-u_{211}+4\, u_{212}+2\, {u_{211}}^2\, {u_{212}}^3),\\
r_{22}&=&{u_{221}}^3\, {u_{222}}^2
/(-1+4\, u_{221}\, {u_{222}}^2+2\, {u_{221}}^3\, {u_{222}}^2),\\
r_{31}&=&{u_{311}}^2\, {u_{312}}^4
/(4-u_{311}+2\, {u_{311}}^2\, {u_{312}}^4),\\
r_{32}&=&{u_{321}}^4\, {u_{322}}^3
/(-1+4\, u_{322}+2\, {u_{321}}^4\, {u_{322}}^3),\\
r_{41}&=&{u_{412}}^3\, (4+u_{411}\, u_{412})^2
/(-u_{411}+2\, {u_{412}}^3\, (4+u_{411}\, u_{412})^2),\\
r_{42}&=&{u_{421}}^3\, (4+u_{421})^2\, {u_{422}}^4
/(-1+2\, {u_{421}}^3\, (4+u_{421})^2\, {u_{422}}^4),\\
r_{51}&=&{u_{512}}^2\, (4+u_{511}\, {u_{512}}^2)^2
/(-u_{511}+2\, {u_{512}}^2\, (4+u_{511}\, {u_{512}}^2)^2),\\
r_{52}&=&{u_{521}}^2\, {u_{522}}^3\, (4+{u_{521}}^2\, u_{522})^2
/(-1+2\, {u_{521}}^2\, {u_{522}}^3\, (4+{u_{521}}^2\, u_{522})^2),\\
r_{61}&=&u_{612}\, (4+u_{611}\, {u_{612}}^3)^2
/(-u_{611}+2\, u_{612}\, (4+u_{611}\, {u_{612}}^3)^2),\\
r_{62}&=&u_{621}\, {u_{622}}^2\, (4+{u_{621}}^3\, {u_{622}}^2)^2
/(-1+2\, u_{612}\, {u_{622}}^2\, (4+{u_{621}}^3\, {u_{622}}^2)^2),\\
r_{71}&=&(4+u_{711}\, {u_{712}}^4)^2
/(-u_{711}+2\, (4+u_{711}\, {u_{712}}^4)^2),\\
r_{72}&=&u_{722}\, (4+{u_{721}}^4\, {u_{722}}^3)^2
/(-1+2\, u_{722}\, (4+{u_{721}}^4\, {u_{722}}^3)^2)
\end{eqnarray*}
in the coordinate charts which cover $I^6$. 
The part $L_0^{(9)}\setminus L_3^{(6)}$ 
of $I^6$ is equal to the line $u_{021}=0$ 
on which $r_{02}=0$. The part $L_1^{(8)}\setminus L_2^{(7)}$ 
of $I^6$ is equal to the line $u_{121}=0$ 
on which $r_{12}=0$. The part $L_2^{(7)}\setminus L_3^{(6)}$ 
of $I^6$ is equal to the line $u_{221}=0$ 
on which $r_{22}=0$. The part $L_3^{(6)}\setminus 
(L_4^{(5)}\cup L_2^{(1)})$ of $I^6$ 
is equal to the part $u_{311}\neq 4$ 
on the line $u_{312}=0$ on which $r_{31}=0$. 
The part $L_3^{(6)}\setminus 
(L_4^{(5)}\cup L_0^{(9)})$ of $I^6$ 
is equal to the part $u_{322}\neq 1/4$ 
on the line $u_{321}=0$ on which $r_{32}=0$. 
The part $L_4^{(5)}\setminus L_5^{(4)}$ of $I^6$ 
is equal to the line $u_{421}=0$ on which $r_{42}=0$. 
The part $L_5^{(4)}\setminus L_6^{(3)}$ of $I^6$ 
is equal to the line $u_{521}=0$ on which $r_{52}=0$. 
The part $L_6^{(3)}\setminus L_7^{(2)}$ of $I^6$ 
is equal to the line $u_{621}=0$ on which $r_{62}=0$. 
The part $L_6^{(3)}\setminus L_5^{(4)}$ 
of $I^6$ 
is equal to the line $u_{722}=0$ on which $r_{72}=0$. 
This covers all of $I^6$, and the proof of the 
first statement in the lemma is complete. 

For the second statement we observe that 
$L_7^{(2)}\setminus (L_6^{(3)}\cup L_8^{(1)})$ is the line 
$u_{712}=0$, $u_{711}\neq 32$ on which $r_{71}=16/(u_{711}-32)^2$, 
whereas $u_{721}=0$, $u_{722}\neq 1/32$, on which 
$r_{72}=16\, u_{722}/(32\, u_{722}-1)^2$, is an  
open neighborhood of $L_7^{(2)}\cap L_6^{(3)}$ in $L_7^{(2)}$. 
Note that $r$ becomes infinite when approaching 
$L_7^{(2)}\cap L_6^{(3)}$ on $L_7^{(2)}$, which is why 
$L_7^{(2)}$ cannot be included in the first statement of the lemma. 
\end{proof}

The function $|d|$ in the following lemma 
will be used as a measure for the distance 
to the infinity set $I$ of the vector field. 
\begin{lemma}
Suppose $z$ is bounded away from zero. 
Let $q:=2\, E$. 
There exists a continuous complex valued function $d$ on 
a neighborhood of $I$ in $S_9$ such that 
$d=q^{-1}$ in a neighborhood in $S_9$ of $I\setminus 
L_8^{(1)}$, 
$d=w_{92}$ in a neighborhood in $S_9$ of the remaining part 
$L_8^{(1)}\setminus L_7^{(2)}$ of $I$, 
and $q\, d\to 1$, $d/w_{92}\to 1$ when approaching 
$L_8^{(1)}\setminus L_7^{(2)}$. 

If the solution at the complex time $z$ is 
sufficiently close to a point of $L_8^{(1)}\setminus L_7^{(2)}$ $($parametrized by coordinate $u_{921}$ $)$, then there exists a unique $\zeta\in\C$ such that 
$|z-\zeta |=\op{O}(|d(z)|\, |u_{921}(z)|)$, where $d(z)$ is small and $|u_{921}(z)|$ is bounded, and 
$u_{921}(\zeta )=0$, that is, the solution of the 
Boutroux\--Painlev\'e equation has a pole at $z=\zeta$. 
In the sequel we write $\delta :=d(\zeta )
=w_{92}(\zeta )=2^6\, u_{922}(\zeta )$, 
and consider $\delta \to 0$.
We have $d(z)/\delta\sim 1$. 
For large finite $R_8\in\R _{>0}$, the connected component 
of $\zeta$ in $\C$ of the set of all $z\in\C$ such that 
$|u_{921}(z)|\leq R_8$ is an approximate disc $D_8$ with center 
at $\zeta$ and radius $\sim 2^{-5}\, |\delta |\, R_8$, 
and $z\mapsto u_{921}(z)$ is a complex analytic diffeomorphism 
from $D_8$ onto $\{ u\in\C\mid |u|\leq R_8\}$. 

For $i$ decreasing from 7 to 4 we use the coordinate 
$u_{(i+1)\, 21}$ in order to parametrize 
$L_i^{(9-i)}\setminus L_{i-1}^{(10-i)}$, where 
$u_{(i+1)\, 21}=0$  
corresponds to the intersection point of $L_i^{(9-1)}$ with 
$L_{i+1}^{(8-i)}$. The point on $L_{i+1}^{(8-i)}\setminus L_i^{(9-i)}$ 
with coordinate $u_{(i+2)\, 21}$ runs to the same intersection 
point when $|u_{(i+2)\, 21}|\to\infty$. 
For large finite $R_i\in\R _{>0}$, 
the connected component of $\zeta$ in $\C$ 
of the set of all $z\in\C$ such that the solution at the 
complex time $z$ is close to $L_i^{(9-i)}\setminus L_{i-1}^{(10-i)}$, 
with $|u_{(i+1)\, 21}(z)|\leq R_i$, but not close to 
$L_{i+1}^{(8-i)}$, is the complement of 
$D_{i+1}$ in an approximate disc $D_i$ with center at $\zeta$ 
and radius $\sim (2^{3-i}\, |\delta |\, R_i)^{1/(9-i)}$, 
where we note that $|\delta |^{1/(9-i)}/|\delta |^{1/(9-(i+1))}
=|\delta |^{-1/(9-i)\, (8-i)}>>1$. More precisely, 
$z\mapsto u_{(i+1)\, 21}$ defines a $(9-i)$\--fold covering 
from the annular domain $D_i\setminus D_{i+1}$ onto 
the complement in $\{ u\in\C\mid |u|\leq R_i\}$ 
of an approximate disc with center at the origin and small 
radius $\sim (2^{-6}\, |\delta |\, {R_{i+1}}^{9-i})^{1/(8-i)}$, 
where $u_{(i+1)\, 21}(z)\sim -2^{i-3}\,\delta\, (z-\zeta )^{9-i}$. 

For all $z\in D_4$, the largest approximate disc, we have 
$|z-\zeta |<<|\zeta |$ and $d(z)/\delta\sim 1$. 
\label{E8E4simlem}
\end{lemma}
\begin{proof}
Recall that $L_8^{(1)}\setminus L_7^{(2)}$ is determined by the equation 
$u_{922}=0$ and is parametrized by $u_{921}\in\C$. Moreover,  
$L_9$ minus one point not on $L_8^{(1)}$ 
corresponds to $u_{921}=0$ and is parametrized by 
$u_{922}$. 
For the study of the solutions near the part $L_8^{(1)}\setminus L_7^{(2)}$ 
of $I$, we use the coordinates $(u_{921},\, u_{922})$. 
Asymptotically for $u_{922}\to 0$ and bounded 
$u_{921}$, $z^{-1}$ 
we have 
\begin{eqnarray}
\dot{u}_{921}&\sim&-2^{-1}\, {u_{922}}^{-1},
\label{dotu921sim}\\
w_{92}&\sim&2^6\, u_{922},
\label{w92u922sim}\\ 
\dot{w}_{92}/w_{92}&=&6\, (5\, z)^{-1}+\op{O}({u_{922}}^2)
=6\, (5\, z)^{-1}+\op{O}({w_{92}}^2),
\label{dotw92sim}\\
q\, w_{92}&\sim&1-2^8\, (5\, z)^{-1}\, {u_{921}}^{-1}.
\label{qw92sim} 
\end{eqnarray}
It follows from (\ref{dotw92sim}) that, as long as the solution 
is close to a given large compact subset of $L_8^{(1)}\setminus L_7^{(2)}$, 
$w_{92}(z)=(z/\zeta )^{6/5}\, w_{92}(\zeta )\, (1+\op{o}(1))$, 
where $z/\zeta\sim 1$ if and only if $|z-\zeta |<< |\zeta |$. 
In view of (\ref{w92u922sim}), in this situation, $u_{922}$ is 
approximately equal to a small constant, when (\ref{dotu921sim}) 
yields that $u_{921}(z)\sim u_{921}(\zeta )-2^{-1}\, 
{u_{922}}^{-1}\, (z-\zeta )$, and it follows that 
$u_{921}(z)$, the affine coordinate on $L_8^{(1)}\setminus L_7^{(2)}$, 
fills an approximate disc centered at $u_{921}(\zeta )$ with 
radius $\sim R$ if $z$ runs over an approximate disc centered at 
$\zeta$ with radius $\sim 2\, |u_{922}|\, R$. Therefore, if 
$|u_{922}(\zeta )|<<1/|\zeta |$, the solution at complex 
times $z$ in a $D$ centered at $\zeta$ with radius 
$\sim 2\, |u_{922}|\, R$ has the properties that along it 
$u_{922}(z)/u_{922}(\zeta )\sim 1$ and that $z\mapsto u_{921}(z)$ 
is a complex analytic diffeomorphism from $D$ onto an 
approximate disc centered at $u_{921}(\zeta )$ with radius 
$\sim R$. If $R$ is sufficiently large, we have 
$0\in u_{921}(D)$, that is, the solution of the 
Boutroux-Painlev\'e equation has a pole at 
a unique point in $D$. After having established this fact,  
we can arrange that $u_{921}(\zeta )=0$, that is, 
the center $\zeta$ of $D$ is equal to the pole point. 
As long as $|z-\zeta |<<|\zeta |$, we have $d(z)/d(\zeta)\sim 1$, i.e.,
$2^6\, u_{922}(z)/\delta\sim w_{92}(z)/\delta\sim 1$ and  
$u_{921}(z)\sim -2^{-1}\, 
{u_{922}}^{-1}\, (z-\zeta )\sim -2^5\,\delta ^{-1}\, (z-\zeta )$, 
where for a large finite $R_8\in\R _{>0}$ the equation  
$|u_{921}(z)|=R_8$ corresponds to 
$|z-\zeta |\sim 2^{-5}\, |\delta |\, R_8$, 
which is still small compared to $|\zeta |$ if $|\delta |$ is 
sufficently small. It follows that the connected component 
$D_8$ of $\zeta$ of the set of all $z\in\C$ such that 
$|u_{921}(z)|\leq R_8$ is an approximate disc 
with center at $\zeta$ and small radius $\sim 2^{-5}\, |\delta |\, R_8$. 
More precisely, $z\mapsto u_{921}(z)$ is a complex analytic 
diffeomorphism from $D_8$ onto $\{ u\in\C\mid |u|\leq R_8\}$, 
and $d(z)/\delta\sim 1$ for all $z\in D_8$. 
The function $q(z)$ has a simple pole at $z=\zeta$, 
but it follows from (\ref{qw92sim}) that $q(z)\, w_{92}(z)\sim 1$ 
as soon as $1>>|z^{-1}\, u_{921}(z)^{-1}|
\sim |\zeta ^{-1}\, 2\, u_{922}(\zeta )\, (z-\zeta )^{-1}|
=2^{-5}\, |\delta |/|\zeta\, (z-\zeta )|$, 
that is, when $|z-\zeta |>> 2^{-5}\, |\delta |/|\zeta |$. 
As the approximate radius of $D_8$ is $2^{-5}\, |\delta |\, R_8
>>2^{-5}\, |\delta |/|\zeta |$ because $R_8>>1/|\zeta |$, 
we have $q(z)\, w_{92}(z)\sim 1$ for $z\in D_8\setminus D_9$, 
where $D_9$ is a disc centered at $\zeta$ with small radius 
compared to the radius of $D_8$. 

The set $L_7^{(2)}\setminus L_6^{(3)}$ is visible 
in the coordinate system 
$(u_{821},\, u_{822})$, where it 
corresponds to the equation 
$u_{822}=0$ and is parametrized by $u_{821}\in\C$. 
The set $L_8^{(1)}$ minus one point 
corresponds to $u_{821}=0$ and is parametrized by 
$u_{822}\in\C$. 
It follows from the equations which express $(u_{811}, 
\, u_{812})$ and  
$(u_{821},\, u_{822})$ in terms of 
$(u_{711},\, u_{712})$ that 
$u_{822}=1/u_{811}$. Also,   
$u_{811}=u_{921}-2^8\, (5\, z)^{-1}$ which implies that $u_{921}\to\infty$ if and only 
if $u_{822}\to 0$. That is, the point near    
$L_8^{(1)}$ approaches the intersection point 
with $L_7^{(2)}$, 
when (\ref{qw92sim}) implies that $q\, w_{92}\to 1$. 
Therefore the functions $q^{-1}$ and $w_{92}$ can be glued together 
by means of a continuous interpolation 
to a continuous function $d$ as asserted in the lemma. 

Asymptotically for $u_{822}\to 0$ and 
bounded $u_{821}$ and $z^{-1}$, we have 
\begin{eqnarray}
\dot{u}_{821}&\sim&-{u_{822}}^{-1},
\label{dotu821sim}\\
\dot{u}_{822}&\sim&2^{-1}\, {u_{821}}^{-1},
\label{dotu822sim}\\
w_{82}&\sim&2^6\, u_{821}\, {u_{822}}^2,
\label{w82sim}\\
q\, w_{82}&\sim&1,
\label{qw82sim}\\
\dot{q}/q=\dot{E}/E&\sim&-6\, (5\, z)^{-1}-2^7\, (5\, z)^{-1}\, {u_{821}}^{-1}.
\label{dotq/qsimi=7} 
\end{eqnarray} 
It follows from (\ref{dotq/qsimi=7}) and (\ref{dotu822sim}) 
that $\dot{q}/q\sim -6\, (5\, z)^{-1}-2^8\, (5\, z)^{-1}\, \dot{u}_{822}$, 
hence 
\begin{eqnarray*}
\log (q(z_1)/q(z_0))&\sim&\log ((z_1/z_0)^{-6/5})\\
&&-(2^8/5)\, ({z_1}^{-1}\, u_{822}(z_1)-{z_0}^{-1}\, u_{822}(z_0)
+\int_{z_0}^{z_1}\, z^{-2}\, u_{822}(z)\,\op{d}\! z).
\end{eqnarray*}
Therefore $q(z_1)/q(z_0)\sim 1$, if for all $z$ on the segment 
from $z_0$ to $z_1$ we have $|z-z_0|<<|z_0|$ and 
$|u_{822}(z)|<<|z_0|$. We choose $z_0$ on the boundary 
of $D_8$, when $d(z_0)^{-1}\, \delta\sim 
q(z_0)\,\delta\sim q(z_0)\, w_{92}(z_0)\sim 1$, 
and $|u_{921}(z_0)|=R_8$ implies that  
$|u_{822}(z_0)|\sim {R_8}^{-1}<<1$. Furthermore, 
(\ref{w82sim}) and (\ref{qw82sim}) imply that 
$|u_{821}(z_0)|\sim 2^{-6}\, |w_{82}(z_0)|\, |u_{822}(z_0)|^{-2}
\sim 2^{-6}\, |\delta |\, {R_8}^{-2}$, which is small when 
$|\delta |$ is sufficiently small. 
Because 
$D_8$ is an approximate disc with center at $\zeta$ 
and small radius $\sim 2^{-5}\, |\delta |\, R_8$, 
and  $R_8>>|\zeta |^{-1}$, we have that $|u_{921}(z)|\geq R_8>>1$ 
hence $|u_{822}(z)|<<1$ if $z=\zeta +r\, (z_0-\zeta )$, $r\geq 1$, 
and $|z-z_0|/|z_0|=(r-1)\, |1-\zeta /z_0|<<1$ if 
$r-1$ is small compared to the large number $1/|1-\zeta /z_0|$. 

Then equations (\ref{qw82sim}), (\ref{w82sim}), and $q\sim\delta ^{-1}$  
yield ${u_{822}}^{-1}\sim (\delta ^{-1}\, 2^6\, 
u_{821})^{1/2}$, which in combination 
with (\ref{dotu821sim}) leads to 
$2\, \op{d}({u_{821}}^{1/2})/\op{d}\! z
=\, -2^3\, \delta ^{-1/2}$,
hence 
$u_{821}(z)^{1/2}
\sim u_{821}(z_0)^{1/2}-2^2
\,\delta ^{-1/2}\, (z-z_0)$, 
and therefore 
$u_{821}(z)\sim 2^4\,\delta ^{-1}\, (z-z_0)^2$ 
if $|z-z_0|>>|u_{821}(z_0)|^{1/2}$. 
For large finite $R_7\in\R _{>0}$ the equation 
$|u_{821}(z)|=R_7$ corresponds to 
$|z-z_0|\sim (2^{-4}\, |\delta |\, R_7)^{1/2}$, 
which is still small compared to $|z_0|\sim |\zeta |$, and therefore 
$|z-\zeta |\leq |z-z_0|+|z_0-\zeta |<<|\zeta |$. 
This proves the statements about the behavior of the 
solution near $L_7^{(2)}\setminus L_6^{(3)}$. 

The statements for 
$4\leq i\leq 6$ about the behavior of the solutions 
near the part $L_i^{(9-i)}\setminus L_{i-1}^{(10-i)}$ 
of $I$ will be proved by induction over decreasing $i$. 
The set $L_i^{(9-i)}\setminus L_{i-1}^{(10-i)}$ is visible 
in the coordinate system 
$(u_{(i+1)\, 21},\, u_{(i+1)\, 22})$, where it 
corresponds to the equation 
$u_{(i+1)\, 22}=0$ and is parametrized by $u_{(i+1)\, 21}\in\C$. 
The set $L_{i+1}^{(8-i)}$ minus one point 
corresponds to $u_{(i+1)\, 21}=0$ and is parametrized by 
$u_{(i+1)\, 22}\in\C$. 
It follows from the equations which express $(u_{(i+1)\, 11}, 
\, u_{(i+1)\, 12})$ and  
$(u_{(i+1)\, 21},\, u_{(i+1)\, 22})$ in terms of 
$(u_{i\, 11},\, u_{i\, 12})$ that 
$u_{(i+1)\, 22}=1/u_{(i+1)\, 11}$,  
$u_{711}=u_{821}+32$, $u_{611}=u_{721}$, and $u_{511}=u_{621}$. 
This shows that $u_{(i+2)\, 21}\to\infty$ if and only 
if $u_{(i+1)\, 22}\to 0$, that is, the point near    
$L_{i+1}^{(8-i)}$ approaches the intersection point 
with $L_i^{(9-i)}$. 

Asymptotically for $u_{(i+1)\, 22}\to 0$ and 
bounded $u_{(i+1)\, 21}$ and $z^{-1}$, we have 
\begin{eqnarray}
\dot{u}_{(i+1)\, 21}&\sim&-(9-i)\, 2^{-1}\, {u_{(i+1)\, 22}}^{-1},
\label{dotu(i+1)21sim}\\
w_{(i+1)\, 2}&\sim&2^6\, {u_{(i+1)\, 21}}^{8-i}\, {u_{(i+1)\, 22}}^{9-i},
\label{w(i+1)2sim}\\
q\, w_{(i+1)2}&\sim&1,
\label{qw(i+1)2sim}\\
\dot{q}/q=\dot{E}/E&\sim&-6\, (5\, z)^{-1}. 
\label{dotq/qsim} 
\end{eqnarray} 

Assume that $|u_{(i+2)\, 21}(z_0)|=R_{i+1}>>1$, 
where the induction hypothesis yields that 
$|z_0-\zeta |<<|\zeta |$ and 
$1/(q(z_0)\,\delta )\sim d(z_0)/\delta\sim 1$. 
It follows from 
(\ref{w(i+1)2sim}), (\ref{qw(i+1)2sim}), and $|u_{(i+1)\,22}(z_{0})|\sim 1/|u_{(i+2)\,21}(z_{0})|$,
that 
\[
|u_{(i+1)\, 21}(z_0)|^{8-i}\sim 
2^{-6}\, |q(z_0)|^{-1}\, |u_{(i+1)\, 22}|^{i-9}
\sim 2^{-6}\, |\delta|\, {R_{i+1}}^{9-i}, 
\]
which is small if $|\delta |$ is sufficiently small. 

It follows from (\ref{dotq/qsim}) that 
$q(z)/\delta \sim q(z)/q(z_0)\sim 1$ along a solution near 
$L_i^{(9-i)}\setminus L_{i-1}^{(10-i)}$, 
as long as $|z-z_0|<<|z_0|$. 
Then (\ref{qw(i+1)2sim}) and (\ref{w(i+1)2sim}) 
yield 
\[
{u_{(i+1)\, 22}}^{-1}\sim (\delta ^{-1}\, 2^6\, 
{u_{(i+1)\, 21}}^{8-i})^{1/(9-i)}, 
\]
which in combination 
with (\ref{dotu(i+1)21sim}) leads to 
\[
(9-i)\,\op{d}({u_{(i+1)\, 21}}^{1/(9-i)})/\op{d}\! z
=\, -(9-i)\, 2^{-1+6/(9-i)}\, \delta ^{-1/(9-i)},
\]
hence 
\[
u_{(i+1)\, 21}(z)^{1/(9-i)}
\sim u_{(i+1)\, 21}(z_0)^{1/(9-i)}-2^{(i-3)/(9-i)}
\,\delta ^{-1/(9-i)}\, (z-z_0), 
\]
and therefore 
$u_{(i+1)\, 21}(z)\sim 2^{i-3}\,\delta ^{-1}\, (z-z_0)^{9-i}$ 
if $|z-z_0|>>|u_{(i+1)\, 21}(z_0)|^{1/(9-i)}$. 
For large finite $R_i\in\R _{>0}$ the equation 
$|u_{(i+1)\, 21}(z)|=R_i$ corresponds to 
$|z-z_0|\sim (2^{3-i}\, |\delta |\, R_i)^{1/(9-i)}$, 
which is still small compared to $|z_0|\sim |\zeta |$, and therefore 
$|z-\zeta |\leq |z-z_0|+|z_0-\zeta |<<|\zeta |$. 
\end{proof}

The following corollary implies that the infinity set 
$I$ of the vector field is repelling. 
This in turn implies that every solution which starts 
in Okamoto's space $S_9\setminus I$ remains 
there for all complex nonzero times. 
\begin{corollary}
For every $\epsilon _1>0$, $0<\epsilon _2<6/5$, and 
$0<\epsilon _3<1$, 
there exists a $\delta\in\R _{>0}$ such that for every 
solution we have that if $|z_0|\geq\epsilon _1$ 
and $|d(z_0)|<\delta $, we have the following conclusions. 
Let $\rho$ denote the supremum of all $r>|z_0|$ 
such that $|d(z)|<\delta$ whenever $|z_0|\leq |z|\leq r$.  
Then 
\begin{itemize}
\item[i)] $\rho$ is bounded above by the inequality 
$\delta\geq |d(z_0)|\, (\rho /|z_0|)^{6/5-\epsilon _2}\, (1-\epsilon _3)$.  
\item[ii)] If $|z_0|\leq |z|\leq\rho$, then 
$d(z)=d(z_0)\, (z/z_0)^{6/5+\varepsilon _2(z)}\, (1+\varepsilon _3(z))$, 
where $|\varepsilon _2(z)|\leq\epsilon _2$ and 
$|\varepsilon _3(z)|\leq\epsilon _3$. 
\item[iii)] If $|z|\geq\rho$ then 
$|d(z)|\geq\delta\, (1-\epsilon _3)$. 
\end{itemize}
\label{6/5cor}
\end{corollary}
\begin{proof}
It follows from Lemma \ref{E8E4simlem} that for every solution 
close to $I$ the set of all not too small 
complex times $z$ such that the solution is not near 
$I^6$ is a union of approximate discs 
of radius of order $|d|^{1/2}$ where the distance 
between the discs is at least of order $|d|^{1/3}$, 
where $|d|^{1/3}/|d|^{1/2}>>1$. Therefore, if the solution 
is near $I^6$ at the complex times $z_0$ and $z_1$, 
and is near $I$ for all complex times 
$z$ such that $|z|$ is between $|z_0|$ and $|z_1|$, there 
is a path $\gamma$ from $z_0$ to $z_1$ such that for all 
$z$ on $\gamma$ we have that the solution at time $z$ 
is near $I^6$, and $\gamma$ is $\op{C}^1$ 
close to the path $[0,\, 1]\ni t\mapsto e^{l_0+t\, (l_1-l_0)}$, 
where $l_i=\log z_i$. Then Lemma \ref{dotq/qlem} 
implies that $q(z_1)=q(z_0)\, (z_1/z_0)^{-6/5+\op{o}(1)}\, 
(1+\op{o}(1))$, 
hence $d(z)=d(z_0)\, (z/z_0)^{6/5+\op{o}(1)}\, (1+\op{o}(1))$. 
Because Lemma \ref{E8E4simlem} implies that the ratio 
between the values of $d$ remains close to 1 if the solution 
stays close to $L_7^{(2)}\cup L_8^{(1)}$, it follows that 
$d(z_1)=d(z_0)\, (z_1/z_0)^{6/5+\op{o}(1)}\, (1+\op{o}(1))$ 
if the solution is close to $I$ at all 
complex times $z$ such that $|z|$ is between $|z_0|$ and $|z_1|$. 
The corollary follows from these estimates. 
\end{proof}
\begin{remark}
The substitutions 
\begin{eqnarray*}
u_1(z)&=&((5/4)\, z)^{-2/5}\, 
y_1(((5/4)\, z)^{4/5}),\\ 
u_2(z)&=&((5/4)\, z)^{-3/5}\, y_2(((5/4)\, z)^{4/5})
\end{eqnarray*} 
in the beginning of Section \ref{boutrouxsec} 
lead in combination with (\ref{qeq}) to 
\begin{eqnarray*}
E(z)&=&2^{-1}\, ((5/4)\, z)^{-6/5}\, y_2(((5/4)\, z)^{4/5})^2\\
&&-2\, ((5/4)\, z)^{-6/5}\, y_1(((5/4)\, z)^{4/5})^3
-((5/4)\, z)^{-2/5}\, y_1(((5/4)\, z)^{4/5}).
\end{eqnarray*}
Because the solution $(y_1(x),\, y_2(x))$ of the 
Painlev\'e system is single valued, we have 
the analytic continuation formula 
\begin{equation}
E(z\,\op{e}^{5\pi\scriptop{i}/2})=\, -E(z),
\label{qangle5/4}
\end{equation}
analogous to (\ref{uangle5/4}). 
Because also $(z\,\op{e}^{5\pi\scriptop{i}/2})^{-6/5}
=\, -z^{-6/5}$, the asymptotic formula 
$E(z)/E(z_0)\sim (z/z_0)^{-6/5}$ along solutions close 
to the part $I\setminus (L_8^{(1)}\cup L_7^{(2)})$ 
of the inifinity set $I$ 
is consistent with (\ref{qangle5/4}).  
\label{not6/5powerrem}
\end{remark}

Because the substitutions of coordinates in 
Section \ref{boutrouxsec} depend in a polynomial way on 
$z^{-1}$, the bundle of the complex projective algebraic 
surfaces $S_9(z)$, $z\in\C\setminus\{ 0\}$ extends 
to a complex analytic family ${\mathcal S}_9=S_9(z)$, 
$z\in\Proj^1\setminus\{ 0\}$, where the complex projective 
line $\Proj ^1$ is identified with the Riemann sphere 
$\C\cup\{\infty\}$. The surface $S_9(\infty )$ over the 
point $\infty\in\Proj ^1$ is obtained by blowing up 
$\Proj ^2$ nine times as in the definition of 
$S_9(z)$, where in the formulas for the base point $b_8(z)$ 
and the coordinate systems $(u_{911},\, u_{912})$ 
and $(u_{921},\, u_{922})$ the coefficient $1/z$ 
is replaced by zero. Because $b_8(\infty )=b_8^{\,\scriptop{ell}}$, 
the base point of the anticanonical pencil defined by $w$ and $E\, w$, 
the limit surface $S_9(\infty )$ is equal to the rational elliptic 
surface obtained by blowing up the base points of the 
anticanonical pencil. 
The Boutroux\--Painlev\'e vector field 
converges for $z\to\infty$ to the vector field of the 
autonomous Hamiltonian system 
$\dot{u}_1=u_2$, $\dot{u}_2=6\, {u_1}^2+1$ 
with Hamiltonian function equal to the energy function $E$ 
in (\ref{qeq}). 
The function $u_1(z)$ satisfies the Weierstrass 
equation $(\dot{u}_1)^2-4\, {u_1}^3-2\, u_1=2\, E$, 
which is why in the sequel we will use the function 
$q:=2\, E$ instead of the energy function $E$. 
The function $q$ defines the elliptic fibration 
$q:S_9(\infty )\to\Proj ^1$, where the fiber 
$I(\infty )=q^{-1}(\{\infty\})
=\lim_{z\to\infty}\, I(z)$ 
over $q=\infty$ is a singular fiber of Kodaira type 
~$\op{II}^*$. The vector field of the autonomous 
Hamiltonian system is regular in the limit fiber  
$S_9(\infty )\setminus I(\infty )$ 
of Okamoto's space of initial conditions, and infinite 
on $I(\infty )$. The function $q$ is constant 
on its solution curves, and each non\--singular fiber 
is an elliptic curve where the time parameter of the 
solution leads to an identification of the fiber 
with $\C /P(q)$, where $P(q)$ denotes the period lattice 
of the flow at the level $q$. The $\, -1$ curve $L_9(\infty )$ 
which appears at the last, the ninth blowup is a global 
holomorphic section for the elliptic fibration. Starting at 
the complex time $z=0$ on the unique intersection point 
of the level curve with $L_9(\infty )$, the period lattice 
$P(q)$ is equal to the set of all $z\in\C$ such that 
the solution of the autonomous Hamiltonian system 
hits $L_9(\infty )$. 
In view of the Weierstrass 
equation $(\dot{u}_1)^2-4\, {u_1}^3-2\, u_1=q$, and the fact that 
$u_1(z)$ has a pole at $z=\zeta$ if and only if 
$\zeta\in P(q)$, the $u_1$\--coordinate of this 
solution is equal to the Weierstrass $\wp$ function 
of the lattice $P(q)$, and we recover the fact that 
hitting $L_9(\infty )$ corresponds to $u_1$ having a pole.   
The equilibrium points of the autonomous 
Hamiltonian system are the points in the affine 
$(u_1,\, u_2)$\--charts determined by the equations 
$u_2=0$ and $6\, {u_1}^2+1=0$. The corresponding singular values 
are $q=\, -4\, u_1\, (-1/6)-2\, u_1
=\, -(4/3)\, u_1=\,\pm\op{i}\,\sqrt{8/27}$. 
For each of these two finite singular values of $q$ we have a 
singular fiber of Kodaira type ~$\op{I}_1$. Therefore 
the configuration of the singular fibers of 
the rational elliptic surface $S_9(\infty )$ is 
~$\op{II}+\op{I}_1+\op{I}_1$, the second item in 
Persson's list \cite[pp. 7--14]{persson}
of configurations of singular fibers of rational elliptic surfaces. 
It also occurs on p. 121 in the classification 
of Schmickler\--Hirzebruch \cite{schmickler} 
of all elliptic fibrations over $\Proj ^1$ with 
at most three singular fibers.  

The following definition is a complex version of the 
concept of limit sets in dynamical systems. 
\begin{definition}
For every solution $\C\setminus\{ 0\}
\ni z\mapsto U(z)\in S_9(z)\setminus I(z)$, 
let $\Omega _U$ 
denote the set of all $s\in S_9(\infty )
\setminus I(\infty )$ such that there 
exists a sequence $z_j\in\C$ with the property that 
$z_j\to\infty$ and $U(z_j)\to s$ as $j\to\infty$. 
The subset $\Omega _U$ of $S_9(\infty )
\setminus I(\infty )$ is called the 
{\em limit set of the solution $U$}. 
\end{definition}
Corollary \ref{limitsetcor} below is analogous to Coddington and Levinson 
\cite[Th. 1.1 and 1.2 in Ch.16]{cl}. 
\begin{corollary}
There exists a compact subset $K$ of 
$S_9(\infty )\setminus I(\infty )$ 
such that for every solution $U$ the limit set $\Omega _U$ 
is contained in $K$. The limit set $\Omega _U$ 
is a non\--empty, compact and connected subset of $K$, 
invariant under the flow of the autonomous Hamiltonian 
system on $S_9(\infty )\setminus I(\infty )$.
For every neighborhood $A$ of $\Omega _U$ in ${\mathcal S}_9$ 
there exists an $r>0$ such that 
$U(z)\in A$ for every $z\in\C$ such that $|z|>r$. 
If $z_j$ is any sequence in $\C\setminus\{ 0\}$ 
such that $z_j\to\infty$ as $j\to\infty$, then there 
is a subsequence $j=j(k)\to\infty$ as $k\to\infty$ and an $s\in\Omega _U$ 
such that $U(z_{j(k)})\to s$ as $k\to\infty$. 
Finally, for every solution $U$ the limit set 
$\Omega _U$  is invariant under the transformation $T$ 
of $S_9(\infty )$ which in the coordinate system 
$(u_1,\, u_2)$ is given by 
$(u_1,\, u_2)\mapsto (-u_1,\,\op{i} u_2)$, when 
$q\mapsto -q$ and $E\mapsto -E$. 
\label{limitsetcor}
\end{corollary}
\begin{proof}
For any $\delta ,\, r\in\R _{>0}$, let 
$K_{\delta ,\, r}$ denote the set of all 
$s\in S_9(z)$ such that $|z|\geq r$ and 
$|d(s)|\geq\delta$. Because ${\mathcal S}_9$ 
is a complex analytic family over 
$\Proj ^1\setminus\{ 0\}$ of compact 
surfaces $S_9(z)$, $z\in\Proj ^1\setminus\{ 0\}$, 
$K_{\delta ,\, r}$ is a compact subset 
${\mathcal S}_9$. Furthermore $K_{\delta ,\, r}$ is disjoint 
from union of the infinity sets $I(z)$, 
$z\in\Proj ^1\setminus\{ 0\}$, and therefore 
$K_{\delta ,\, r}$ is a compact subset of 
Okamoto's space ${\mathcal S}_9\setminus {\mathcal S}_{9,\,\infty}$, 
where the latter is viewed as a complex analytic family of 
non\--compact surfaces over $\Proj ^1\setminus\{ 0\}$. 
When $r\uparrow\infty$, the sets 
$K_{\delta ,\, r}$ shrink to the set 
$K_{\delta ,\,\infty}$ of all $s\in S_9(\infty )$ 
such that $|d(s)|\geq\delta$, which is a compact subset 
of $S_9(\infty )\setminus I(\infty )$. 

It follows from Corollary \ref{6/5cor} 
that there exists $\delta\in\R _{>0}$ such that 
for every solution $U$ there exists $r_0\in\R _{>0}$ 
with the property that $U(z)\in K_{\delta ,\, r_0}$ for every 
$z\in\C$ such that $|z|\geq r_0$. In the sequel, let 
$r\geq r_0$, when it follows from the definition 
of $K_{\delta ,\, r}$ that $U(z)\in K_{\delta ,\, r}$ 
whenever $|z|\geq r$. 
Let $Z_r:=\{ z\in\C\mid |z|\geq r\}$ and let 
$\Omega _{U,\, r}$ denote the closure of $U(Z_r)$ 
in ${\mathcal S}_9$. Because $Z_r$ is connected 
and $U$ is continuous, $U(Z_r)$ is connected, 
hence its closure $\Omega _{U,\, r}$ is connected. 
Because $U(Z_r)$ is contained in the compact 
subset $K_{\delta ,\, r}$, its closure $\Omega _{U,\, r}$ 
is contained in $K_{\delta ,\, r}$, and therefore 
$\Omega _{U,\, r}$ is a non\--empty compact 
and connected subset of ${\mathcal S}_9\setminus 
{\mathcal S}_{9,\,\infty}$. Because the intersection 
of a decreasing sequence of non\--empty compact and connected 
sets is non\--empty, compact, and connected, and the 
sets $\Omega _{U,\, r}$ decrease to $\Omega _U$ as 
$r\uparrow\infty$, it follows that $\Omega _U$ is a 
non\--empty, compact and connected subset of 
${\mathcal S}_9$. Because $\Omega _{U,\, r}\subset K_{\delta ,\, r}$ 
for all $r\geq r_0$, and the sets $K_{\delta ,\, r}$ shrink to 
the compact subset $K_{\delta ,\,\infty}$ of 
$S_9(\infty )\setminus I(\infty )$ as $r\uparrow\infty$, 
it follows that $\Omega _U\subset K_{\delta ,\,\infty}$. 
This proves the first statement in the corollary 
with $K=K_{\delta ,\,\infty}$. Because $\Omega _U$ is 
the intersection of the decreasing family of compact sets $
\Omega _{U,\, r}$, there exists for every neighborhood 
$A$ of $\Omega _U$ in ${\mathcal S}_9$ an $r>0$ such that 
$\Omega _{U,\, r}\subset A$, hence $U(z)\in A$ for every 
$z\in\C$ such that $|z|\geq r$. If $z_j$ is any sequence 
in $\C\setminus\{ 0\}$ such that $|z_j|\to\infty$, then 
the compactness of $K_{\delta ,\, r}$, in combination with 
$U(Z_r)\subset K_{\delta ,\, r}$, implies that 
there is a subsequence 
$j=j(k)\to\infty$ as $k\to\infty$ and an $s\in K_{\delta ,\, r}$ 
such that $U(z_{j(k)})\to s$ as $k\to\infty$, when it follows 
from the definition of $\Omega _U$ that $s\in\Omega _U$. 

We next prove that $\Omega _U$ is invariant under the flow 
$\Phi ^t$ of the autonomous system Hamiltonian system.
Let $s\in\Omega _U$, when there is a sequence $z_j\in\C\setminus\{ 0\}$ 
such that $z_j\to\infty$ and $U(z_j)\to s$ as $j\to\infty$. 
Because the $z$\--dependent vector field of the 
Boutroux\--Painlev\'e system converges in $\op{C}^1$ to 
the vector field of the autonomous Hamiltonian system 
as $z\to\infty$, 
it follows from the continuous dependence on initial 
data and parameters for first order ordinary differential 
equations, see for instance Coddington and Levinson 
\cite[Th. 7.4 in Ch. 1]{cl}, that the distance between  
$U(z_j+t)$ and $\Phi ^t(U(z_j))$ converges to zero as 
$j\to\infty$. Because $\Phi ^t(U(z_j))\to\Phi ^t(s)$ and 
$z_j\to\infty$ as $j\to\infty$, it follows that 
$U(z_j+t)\to \Phi ^t(s)$ and $z_j+t\to\infty$ as $j\to\infty$, 
hence $\Phi ^t(s)\in\Omega _U$. 

The invariance of $\Omega _U$ under the transformation $T$ 
follows from (\ref{uangle5/4}) and 
(\ref{qeq}). 
\end{proof}
\begin{corollary}
Every solution of the first Painlev\'e equation has 
infinitely many poles. 
\label{infpolecor}
\end{corollary}
\begin{proof}
Let $u(z)$ be a solution  of  the Boutroux-Painlev\'e  equation 
with only finitely many poles, 
$U(z)$ the corresponding solution of the system in
$S_9\setminus S_{9, \infty}$, and $\Omega _U$ the limit set
of $U$. According to Corollary \ref{limitsetcor}, $\Omega _U$ is a compact
subset of $S_9\setminus S_{9, \infty}$. If $\Omega _U$ intersects
the pole line $L_9$ in a point $p$, then there exist $z$ with
$|z|$ arbitrarily large such that $U(z)$ is arbitrarily close to $p$,
when the transversality of the vector field to the pole line implies
that $U(\zeta) \in L_9$ for a unique $\zeta$ near $z$, which means
that $u(z)$ has a pole at $z =\zeta$. As this would imply that $u(z)$ has 
infinitely many poles, it follows that $\Omega _U$ is a compact
subset of $S_9 \setminus (S_{9, \infty} \cup L_9)$. However,
$S_{9, \infty} \cup L_9$ is equal to the set of all points in $S_9$
which lie over the line $L_0$ at infinity in the complex projective
plane, and therefore $S_9 \setminus (S_{9, \infty} \cup L_9)$
is the affine $(u_1, u_2)$ coordinate chart, of which $\Omega _U$
is a compact subset, which implies that 
$u_1(z)= u(z)$ and $u_2(z)$ remain bounded for large $|z|$.
In view of the theorem on removable singularities it follows that
$u_1(z)$ and $u_2(z)$ are equal to holomorphic functions of
$1/z$ in a neighborhood of $1/z = 0$, which in turn implies
that there are complex numbers $u_1(\infty ), u_2(\infty ) $
such that $u_1(z) \to u_1(\infty )$ and $u_2(z) \to u_2(\infty )$
as $|z| \to \infty$. In other words,
$\Omega_U = \{ (u_1(\infty ),\, u_2(\infty ))\}$.
Because the limit set $\Omega_U$ is invariant under the
autonomous Hamiltonian system and contains only one point,
this point is an equilibrium point of the autonomous Hamiltonian
system. That is, $u_2(\infty ) = 0$ and $u_1(\infty )$ is equal to
one of the two zeros $c$ of $u \mapsto 6\, u^2+1$. 
According to the last statement in Corollary \ref{limitsetcor}, 
$(-c,\, 0)\in\Omega _U$ if $(c,\, 0)\in\Omega _U$, 
where $(-c,\, 0)\neq (c,\, 0)$ because $c\neq 0$. 
This contradiction with $\Omega_U =\{ (c,\, 0)\}$ 
completes the proof. 
\end{proof}
In Lemma \ref{polelem}  
more information will be given about the asymptotic 
distribution for large $|\xi |$ of the poles $\xi$ of 
the solutions $y(x)$ of the first Painlev\'e equation. 
In the remainder of this section we discuss, for the solutions 
$U(z)$ close to the infinity set $I$, 
the asymptotic behavior of the set of $z$ such that 
$U(z)$ is close to $L_i^{(9-i)}$ for $0\leq i\leq 3$, 
extending the description for $4\leq i\leq 8$ in 
Lemma \ref{E8E4simlem}. 

As in Lemma \ref{E8E4simlem}, one finds 
concentric approximate discs $D_1$, $D_2$, and 
approximate discs $D_0$ of small radii 
such that the connected component of the set 
of all $z\in\C$ such that the solution in $S_9$ is close to 
$L_1^{(8)}\setminus L_2^{(7)}$, $L_2^{(7)}\setminus 
(L_1^{(8)}\cup L_3^{(6)})$, and $L_0^{(9)}\setminus 
L_3^{(6)}$ is equal to $D_1$, $D_2\setminus D_1$, and 
$D_0$, respectively. 

More precisely, $L_1^{(8)}\setminus L_2^{(7)}$ is visible 
in $(u_{121},\, u_{122})$ chart, where it is defined by 
$u_{121}=0$ and parametrized by $u_{122}$. 
We have $\dot{u}_{122}\sim {u_{121}}^{-1}={w_{12}}^{-1/2}
\sim q^{1/2}=d^{-1/2}$, where $\dot{q}/q
=\dot{E}/E\sim -6\, (5\, z)^{-1}$, 
hence $d$ is approximately constant. Therefore each connected 
component of the set of all $z\in\C$ such that 
$|u_{122}(z)|\leq R_1$ is an approximate disc $D_1$ of radius 
$\sim |d|^{1/2}\, R_1$, and $z\mapsto u_{122}(z)$ is a 
complex analytic diffeomorphism 
from $D_1$ onto $\{ u\in\C\mid |u|\leq R_1\}$. 
Furthermore, $L_2^{(7)}\setminus L_3^{(6)}$ is 
visible in $(u_{221},\, u_{222})$ chart, where it is defined by 
$u_{221}=0$ and parametrized by $u_{222}$, whereas 
the part of $L_1^{(8)}$ in this chart is defined by 
$u_{222}=0$ and parametrized by $u_{221}=u_{111}
=u_{031}\, {u_{032}}^{-1}={u_{122}}^{-1}$, hence 
$|u_{122}|>R_1$ corresponds to $|u_{221}|<{R_1}^{-1}$. 
We have $\dot{u}_{222}\sim 2\, {u_{221}}^{-1}
=2\, {w_{22}}^{-1/4}\, {u_{222}}^{1/2}$, where 
$w_{22}\sim q^{-1}=d$ and $\dot{q}/q=\dot{E}/E\sim -6\, (5\, z)^{-1}$ 
yields that $({u_{222}}^{1/2})^{\bullet}
\sim d^{-1/4}$ with approximately constant $d$, 
hence $u_{222}(z)\sim (u_{222}(z_0)^{1/2}+d^{-1/4}\, (z-z_0))^2$.  
Let $R_2$ be a large finite positive real number. 
As $|u_{221}|<{R_1}^{-1}$ corresponds to 
$|u_{222}|\sim |d|^{1/2}\, |u_{221}|^{-2}> |d|^{1/2}\, {R_1}^2$, 
and $|d|^{1/4}/|d|^{1/2}>>1$, 
the mapping $z\mapsto u_{222}(z)$ is a twofold covering 
from the complement of $D_1$ in an approximate disc $D_2$ 
of radius $\sim |d|^{1/4}\, {R_2}^{1/2}$ onto 
$\{ u\in\C\mid |d|^{1/2}\, {R_1}^2<|u|\leq R_2\}$. 
Finally $L_0^{(9)}\setminus L_3^{(6)}$ is visible 
in the $(u_{021},\, u_{022})$ chart, where it is defined by 
$u_{021}=0$ and parametrized by $u_{022}$. 
We have $\dot{u}_{022}\sim 6\, {u_{021}}^{-1}
=\, -6\, {w_{02}}^{-1/3}$, 
where $w_{02}\sim 4\, q^{-1}=\, 4\, d$, and 
$\dot{q}/q=\dot{E}/E\sim -6\, (5\, z)^{-1}$ yields that  
$d$ is approximately constant, with 
$\dot{u}_{022}\sim -2^{1/3}\, 3\, d^{-1/3}$. 
Let $R_0$ be a large 
finite positive real constant. Then 
each connected 
component of the set of all $z\in\C$ such that 
$|u_{122}(z)|\leq R_0$ is an approximate disc 
$D_0$ of radius 
$\sim 2^{-1/3}\, 3^{-1}\, |d|^{1/3}\, R_0$, 
and $z\mapsto u_{122}(z)$ is a 
complex analytic diffeomorphism 
from $D_0$ onto $\{ u\in\C\mid |u|\leq R_0\}$. 

In order to understand the location of the 
concentric discs $D_8\subset D_7\subset D_6\subset D_5\subset D_4$, 
the concentric discs $D_1\subset D_2$ and the 
discs $D_0$ in the complex $z$\--plane, we 
first describe the situation for the solutions of the autonomous 
Hamiltonian system, which have a similar behavior near the infinity set 
$I$, but in addition has the function $q=2\, E$ 
as a constant of motion. Recall that for each  
$q\in\C\setminus\{\pm\op{i}\,\sqrt{8/27}\}$ the level set 
in $S_9(\infty )$ of the function $q$ is an  
elliptic curve $C_q$ such that $z\mapsto\Phi ^z(\sigma (q))$ 
defines an isomorphism from $\C /P(q)$ onto $C_q$. 
Here $\Phi ^z$ denotes the flow of the autonomous 
Hamiltonian system defined by the function $E=q/2$, 
$\sigma (q)$ is the unique point in $C_q\cap L_9(\infty )$, 
and $P(q)$ is the period lattice of the flow at the level $q$. 
In the $(u_1,\, u_2)$ coordinate system, the 
first coordinate $u(z)=u_1(\Phi ^z(\sigma (q)))$ of $\Phi ^z(\sigma (q))$ 
is the solution of the autonomous differential equation 
$\op{d}^2u/\op{d}\! z^2=6\, u^2+1$ such that 
$\dot{u}^2-4\, u^3-2\, u=q$ and $u(z)$ has a pole 
at $z\in P(q)$, and no other poles.  
Therefore $u(z)=\wp _{_{P(q)}}(z)$, the Weierstrass 
$\wp$ function defined by the lattice $P(q)$. 

Locally the period lattice has a $\Z$\--basis $p_1(q)$, 
$p_2(q)$ depending in a complex analytic fashion on $q$. 
The period functions $p(q)=p_i(q)$ 
satisfy the homogeneous linear second order differential equation 
\begin{equation}
\frac{\op{d}^2p}{\op{d}\! q^2}
+\frac{54\, q}{8+27\, q^2}\,\frac{\op{d}\! p}{\op{d}\! q} 
+\frac{15}{4\, (8+27\, q^2)}\, p=0,  
\label{pf}
\end{equation}
as an application of Bruns \cite[p. 237, 238]{bruns} 
to $g_2(q)=\, -2$ and $g_3(q)=\, -q$. 
\begin{lemma}
For large $|q|$ the period lattice $P(q)$ has a 
$\Z$\--basis of the form 
\begin{equation}
\begin{array}{lll}
p_1(q)&=&q^{-1/6}\, a(1/q)+q^{-5/6}\, b(1/q),\\ 
p_2(q)&=&q^{-1/6}\,\op{e}^{2\pi\scriptop{i}/6}\, a(1/q)
+q^{-5/6}\, \op{e}^{-2\pi\scriptop{i}/6}\, b(1/q), 
\end{array}
\label{p1p2II*}
\end{equation}
where $a\,$ and $b\,$ are complex analytic 
functions on an open neighborhood of the origin in the 
complex plane such that 
$a(0)=\, -\op{i}\, 2^{-1}\, \pi ^{-1}\,\Gamma (1/3)^3$ and 
$b(0)=\op{i}\,  2^{4}\, 3^{-3/2}\,\pi ^2\,\Gamma (1/3)^{-3}$. 
Here $\Gamma$ denotes Euler's Gamma function. 
The differential equation 
(\ref{pf}) in combination with the explicit values of 
$a(0)$ and $b(0)$ leads to a successive determination of the coefficients 
in the Taylor expansions at the origin of the 
functions $a$ and $b$. 

For $0\leq i\leq 8$ and $i\neq 3$ 
the point $\Phi ^z(\sigma (q))$ is near $L_i^{(9-i)}$ if 
and only if $z$ belongs to the aforementioned approximate discs 
$D_i$ with $\delta =q^{-1}$. The centers of the concentric discs 
$D_8\subset D_7\subset D_6\subset D_5\subset D_4$ 
are at the points of the period lattice $P(q)$ that are 
the pole points of the solution $u(z)=\wp _{_{P(q}}(z)$ 
of the autonomous differential equation $\ddot{u}=6\, u^2+1$. 
The centers of the concentric discs $D_1\subset D_2$ 
are at the zeros of $u(z)=\wp _{_{P(q)}}(z)$. The zeros of $u(z)$ 
are close of order smaller than 
$|q|^{-1/6}$ to the points $(p_1(q)+p_2(q))/3$ and 
$2\, (p_1(q)+p_2(q))/3$ modulo $P(q)$. 
The centers of the discs $D_0$ are at the zeros of the 
derivative $\dot{u}(z)=\op{d}\wp _{_{P(q)}}(z)/\op{d}\! z$, and close 
of order smaller than $q^{-1/6}$ to the points 
$p_1(q)/2$, $p_2(q)/2$, and 
$(p_1(q)+p_2(q))/2$ modulo $P(q)$. 
\label{qinftyperiodlem}
\end{lemma}
\begin{proof}
Any period along a closed path $\gamma _q$
on the curve $\dot{u}^2=4\, u^3+2\, u+q$
is equal to 
\[
p=\oint\, (4\, u^3+2\, u+q)^{-1/2}\,\op{d}\! u
=4^{-1/3}\, q^{1/3}\, q^{-1/2}\,\int_{\gamma}\, (v^3+2\cdot 4^{-1/3}\, 
q^{-2/3}\, v+1)^{-1/2}\,\op{d}\! v.
\]
Asymptotically for $q\to\infty$ we have 
\begin{eqnarray*}
(v^3+2\cdot 4^{-1/3}\, q^{-2/3}\, v+1)^{-1/2}
&=&(v^3+1)^{-1/2}\, (1+2\cdot 4^{-1/3}\, q^{-2/3}\, v/(v^3+1))^{-1/2}\\
&=&(v^3+1)^{-1/2}-4^{-1/3}\, q^{-2/3}\, (v^3+1)^{-3/2}\, v
+\op{O}(q^{-4/3}). 
\end{eqnarray*}
Furthermore $\op{d}(v^3+1)^{-1/2}/\op{d}\! v
=\, -\frac12\, (v^3+1)^{-3/2}\, 3\, v^2$, 
and therefore an integration by parts yields 
\[ 
\oint\, (v^3+1)^{-3/2}\, v\,\op{d}\! v
=\, -(2/3)\, \oint\, ((v^3+1)^{-1/2})'\, v^{-1}\,\op{d}\! v
=-(2/3)\,\oint\, (v^3+1)^{-1/2}\, v^{-2}\,\op{d}\! v.
\]
Therefore $p(q)=q^{-1/6}\,\alpha +q^{-5/6}\,\beta
+\op{O}(q^{-3/2})$, where 
\[
\alpha =2^{-2/3}\,\oint\, (v^3+1)^{-1/2}\,\op{d}\! v
\quad\mbox{\rm and}\quad
\beta =\, 2^{-1/3}\, 3^{-1}
\, \oint\, (v^3+1)^{-1/2}\, v^{-2}\,\op{d}\! v, 
\]
where the integration is over a closed 
path on the elliptic curve $\dot{v}^2=v^3+1$ 
homotopic to $\gamma _q$. Here the elliptic curve is 
the compact one obtained by adding one point at infinity to the 
curve $\dot{v}^2=v^3+1$ in the affine $(v,\,\dot{v})$ plane. 
In the next computations we use the well\--known formulas 
\begin{eqnarray}
\Gamma (p)&:=&\int_0^{\infty}\,\op{e}^{-t}\, t^{p-1}\,\op{d}\! t,
\label{Gammadef}\\
\Gamma (p+1)&=&p\,\Gamma (p),
\label{Gammap+1}\\
\Gamma (\frac12)&=&\pi ^{1/2},
\label{Gamma1/2}\\
\op{B}(p_1,\, p_2)&:=&
\int_0^1\, t^{p_1-1}\, (1-t)^{p_2-1}\,\op{d}\! t
\nonumber\\
&=&\int_0^{\infty}\, s^{p_1-1}\, (s+1)^{-p_1-p_2}\,\op{d}\! s
=\frac{\Gamma (p_1)\,\Gamma (p_2)}{\Gamma (p_1+p_2)},
\label{Beta}\\
\Gamma (2\, p)&=&2^{2\, p-1}\,\pi ^{-1/2}\,\Gamma (p)
\,\Gamma (p+\frac12),\quad\mbox{\rm and} 
\label{legendreduplication}\\
\Gamma (p)\,\Gamma (1-p)&=&\frac{\pi}{\sin (\pi\, p)}
\label{Gammap1-p}
\end{eqnarray}
for Euler's Gamma function, where 
(\ref{Beta}), (\ref{legendreduplication}), and 
(\ref{Gammap1-p}) are Euler's Beta function, 
Legendre's duplication formula, and the reflection formula for the 
Gamma function, respectively. The second identity in 
(\ref{Beta}) follows from the substitution of variables 
$t=s/(s+1)$

As our first loop we take the closed path 
which doubly covers the real $v$\--interval from $-\infty$ to $-1$. 
The substitution of variables $v=\, -(s+1)^{1/3}$ 
then yields in view of (\ref{Beta}) and the other identities 
for the Gamma function that 
$\alpha =\,\pm\op{i}\, 2^{1/3}\, 3^{-1}\,\op{B}(1/2,\, 1/6)
=\,\pm\op{i}\, 2^{-1}\, \pi ^{-1}\,\Gamma (1/3)^3$ 
and $\beta =\,\mp\op{i}\, 2^{2/3}\, 3^{-2}\,\op{B}(1/2,\, 5/6)
=\,\mp\op{i}\,  2^{4}\, 3^{-3/2}\,\pi ^2\,\Gamma (1/3)^{-3}$. 

As our second loop we take the closed path in $\dot{v}^2=v^3+1$ 
which doubly covers the real $v$\--interval from 
$-1$ to $0$ followed by the straight line in the 
$v$\--plane from $0$ to $\op{e}^{-2\pi\scriptop{i}/6}$, 
where the substitution of variables $w=\op{e}^{2\pi\scriptop{i}/3}\, v$ 
shows that the integral of $(w^3+1)^{-1/2}\,\op{d}\! w$ over the 
the second interval is equal to minus the integral 
of $(v^3+1)^{-1/2}\,\op{d}\! v$ over the first interval. 
The first and the second loop in $\dot{v}^2=v^3+1$ intersect each 
other once, at the point $(v,\,\dot{v})=(-1,\, 0)$, where the 
intersection is transversal, and therefore the intersection 
number of the first loop with the second loop is equal to 
$\pm 1$, where the sign depends on the choices of the 
orientations of the loops. as the elliptic curve is 
a real two\--dimensional torus, its first homology group 
is isomorphic to $\Z ^2$, when the fact that the intersection 
number of the two loops is $\pm 1$ implies that the 
homology classes of the two loops form a $\Z$\--basis 
of the first homology group of the elliptic curve. This implies 
in turn that the corresponding periods, asymptotically equal to 
the integrals of $4^{-1/3}\, q^{-1/6}\,
\dot{v}^{-1}\,\op{d}\! v$ over these loops, form 
a $\Z$\--basis of the period lattice $P(q)$. 

The substitution of variables $v=\, -s^{1/3}$ 
yields that the integral of $(v^3+1)^{-1/2}\,\op{d}\! v$ 
over the real $v$\--interval from $-1$ to $0$ is equal to 
$3^{-1}\,\int_0^1\, (1-s)^{-1/2}\, s^{-2/3}\,\op{d}\! s
=3^{-1}\,\op{B}(1/2,\, 1/3)$, and therefore the integral 
over the second loop leads to 
$\alpha =(1-\op{e}^{2\pi\op{i}/3})\, 2^{1/3}\, 3^{-1}
\,\op{B}(1/2,\, 1/3)$. As 
\[
\frac{\op{B}(1/2,\, 1/3)}{\op{B}(1/2,\, 1/6)}
=\frac{\Gamma (1/2)\,\Gamma (1/3)\, \Gamma (2/3)}
{\Gamma (1/2)\,\Gamma (1/6)
\,\Gamma (5/6)}=\frac{\sin (\pi /6)}{\sin (\pi /3)}=3^{-1/2}, 
\] 
and $(1-\op{e}^{2\pi\op{i}/3})\, 3^{-1/2}=
\, -\op{i}\,\op{e}^{2\pi\scriptop{i}/6}$, we arrive 
at the conclusion that, if in the above $\pm$ we choose the 
minus sign, the second period is asymptotically equal to 
$\op{e}^{2\pi\scriptop{i}/6}$ times the first period. 

Because the differential equation (\ref{pf}) for the periods 
has $q=\infty$ as a regular singular point, and 
the two solutions 
$\lambda =\, -1/6$ and $\lambda =\, -5/6$ of its 
indicial equation $\lambda\, (\lambda -1)
+2\,\lambda +5/36=0$ do not 
differ by an integer, each solution of (\ref{pf}) is 
of the form $q^{-1/6}\, a(1/q)+q^{-5/6}\, b(1/q)$, 
where $a\,$ and $b\,$ are complex analytic functions on 
a neighborhood of the origin, and the solution is uniquely 
determined by $a(0)$ and $b(0)$.  
See for instance Coddington and Levinson \cite[Ch. 4, Sec. 4]{cl}. 
It follows that our first period fits the description for 
$p_1(q)$. The analytic continuation of $p_1(q)$, when the small 
$1/q$ runs around the origin once, 
is equal to $p_2(q)$. Because $p_2(q)$ 
asymptotically agrees with our second period, it 
is equal to it. Therefore the periods $p_1(q)$ and $p_2(q)$ 
described in the lemma form 
a $\Z$\--basis of the period lattice $P(q)$. This completes the proof 
of the first paragraph in the lemma. 

For the second paragraph we observe that we took as the 
center of the discs $D_8$ the pole points. 
The line $L_1^{(8)}\setminus L_2^{(7)}$ is visible in the 
$(u_{121},\, u_{122})$ chart, where it is defined by 
$u_{121}=0$ and is parametrized by $u_{122}$. 
As $u_{121}={u_2}^{-1}$ and $u_{122}=u_1$, it follows 
that the centers of the discs $D_1$, which correspond 
to $u_{122}=0$, correspond to the zeros of $u(z)=\wp _{_{P(q)}}(z)$. 
The line $L_0^{(9)}\setminus L_3^{(6)}$ is 
visible in the $(u_{021},\, u_{022})$ chart, where it is 
defined by $u_{021}=0$ and parametrized by $u_{022}$. 
As $u_{021}={u_1}^{-1}$ and $u_{022}={u_1}^{-1}\, u_2$, 
it follows that the centers of the discs $D_0$ correspond to the 
zeros of the derivative 
$u_2(z)=\dot{u}(z)=\op{d}\wp _{_{P(q)}}(z)/\op{d}\! z$ of the 
solution $u(z)=\wp _{_{P(q)}}(z)$ of the autonomous 
differential equation $\ddot{u}=6\, u^2+1$. 

For large $|q|$ and $z$ not in one of the aforementioned 
small discs $D_i$, the solution of the autonomous Hamiltonian system 
in $S_9\setminus I$ is close to the part 
$L_3^{(6)}$ of $I$. In the $(u_{311},\, u_{312})$ chart, 
$L_3^{(6)}$ is defined by $u_{312}=0$ and parametrized by $u_{311}$. 
The base point $(u_{311},\, u_{312})=(4,\, 0)$ corresponds 
to $L_3^{(6)}\cap L_4^{(5)}$, the origin 
$(u_{311},\, u_{312})=(0,\, 0)$ to $L_3^{(6)}\cap L_0^{(6)}$, 
whereas $(u_{311},\, u_{312})=(\infty ,\, 0)$ corresponds to 
$(u_{321},\, u_{322})=(0,\, 0)$ hence to 
$L_3^{(6)}\cap L_2^{(7)}$. We have 
$u_{311}(z)=u(z)^{-3}\,\dot{u}(z)^2$ and 
$\dot{u}^2-4\, u^3-2\, u=q=$ constant, which suggests 
the rescaling $z=z_0+q^{-1/6}\, t$ and 
$u(z)=q^{1/3}\, v(q^{1/6}\, (z-z_0))$. 
Then $u_{311}=v^{-3}\, (v')^2$, 
$v''=6\, v^2+q^{-2/3}$, and 
$(v')^2=4\, v^3+1+2\, q^{-2/3}\, v$, 
hence $u_{311}=4+v^{-3}+2\, q^{-2/3}\, v^{-2}$. 
In the limit $q=\infty$ this leads to 
$u_{311}(t)=v(t)^{-3}\, v'(t)^2=4+v(t)^{-3}$. The equation 
$(v')^2=4\, v^3+1$ has a regular hexagonal period lattice $P$, 
and we arrange that $v(t)=\wp (t)$ is the solution 
with its poles at the points of $P$. As the poles 
have order two, it follows that the mapping 
$\C /P\to L_3^{(6)}:t+P\mapsto u_{311}(t)$ is a 
sixfold branched covering, where near the point $t+P=0+P$ 
the mapping behaves as $u_{311}(t)-4\sim t^6$. 
At this ramification point $t+P$ all the six branches 
come together, where the image point $u_{311}=4$ 
corresponds to the centers of the discs $D_8$. 
There are three ramification points $t$ where   
$v\in\C\setminus\{ 0\}$, $v'=0$, at each of which 
$u_{311}'=0$ and $u_{311}''=2\, v^{-3}\, v''=12\, v^{-1}\neq 0$, 
which means that at each of these ramification points 
two branches come together. Both ramification points 
ly over $u_{311}=0$, corresponding to the centers of the 
discs $D_0$. The only other ramification points 
$t$ occur when $v=0$. There are two of these and at each 
one three branches come together. Both these 
ramification points ly over $u_{311}=\infty$, corresponding to 
the centers of the discs $D_1$. 

The group of deck transformations of the aforementioned 
branched covering is isomorphic to $\Z /6\,\Z$ and 
generated by $T:z+P\mapsto\op{e}^{2\pi\scriptop{i}/6}\, z+P$, 
which is an automorphism of $\C /P$ because 
$P$ has a $\Z$\--basis consisting of $p_1\in\C\setminus\{ 0\}$ 
and $p_2=\op{e}^{2\pi\scriptop{i}/6}\, p_1$. 
The ramification points of order 6 are the 
fixed points in $\C /P$ of $T$, which is the 
single point $0+P$. The ramification points 
of order 3 are the fixed points of $T^2$ which are 
no fixed points of $T$, which are the two 
points $(p_1+p_2)/3+P$ and $2\, (p_1+p_2)/3+P$. 
The ramification points of order 2 are the fixed points 
of $T^3:t+P\mapsto -t+P$ which are no fixed points of $T$, 
which are the three points $p_1/2+P$, $p_2/2+P$, and 
$(p_1+p_2)/2+P$. 
\end{proof}

Figure \ref{latticefig} exhibits the 
complex times $z$ for which the solutions 
of the autonomous Hamiltonian system for large $q$ 
are near the various irreducible components of 
the singular fiber $I(\infty )$ over $q=\infty$. 
The solid dots represent the concentric discs 
$D_i$, $4\leq i\leq 8$, centered at the points of the period 
lattice $P(q)$ where $u_1$ has a pole, with respective radii 
$\sim |q|^{-1/(9-i)}\, r_i$ for large finite $r_i$, 
where the distance between the points of $P(q)$ the period lattice 
is of order $|q|^{-1/6}\gg |q|^{-1/5}\gg \ldots \gg |q|^{-1}$. 
The solution is near $L_8^{(1)}$ when $z\in D_8$, 
and near $L_i^{(9-i)}\setminus L_{i+1}^{(10-i)}$ 
when $z\in D_i\setminus D_{i+1}$ for 
$4\leq i\leq 7$.  
The double circles represent the concentric disks 
$D_1$ and $D_2$, centered at the zeros of $u_1$, 
approximately equal to the points  
$(p_1(q)+p_2(q))/3$ and $2\, (p_1(q)+p_2(q))/3$ modulo $P(q)$,  
and of radius $\sim |q|^{-1/2}\, r_1$ and $\sim |q|^{-1/4}\, r_2$ 
for large finite $r_i$, respectively. The solution is near 
$L_1^{(8)}$ when $z\in D_1$, and near $L_2^{(7)}\setminus L_1^{(8)}$ 
when $z\in D_2\setminus D_1$. 
The single circles represent the concentric discs 
$D_0$ centered at the zeros of $u_2=\dot{u}_1$, approximately equal to 
the points $p_1(q)/2$, $p_2(q)/2$, $(p_1(q)+p_2(q))/2$ 
modulo $P(q)$ with radius $\sim |q|^{-1/3}\, r$ for a large finite $r$. 
The solution is near $L_0^{(9)}$ when 
$z\in D_0$. When $z$ is in the complement of all the 
aforementioned discs, which happens for most of the complex times, 
the solution is near the component $L_3^{(6)}$ 
of $I(\infty )$. 

Figure \ref{wplevelsfig} is a contour plot of the absolute 
value of the Weierstrass $\wp$ function defined by the 
regular hexagonal lattice generated by $1$ and 
$\frac12+\frac12\,\sqrt{3}\,\op{i}$. The areas close to the 
lattice points = the pole points of $\wp (z)$ should have 
been blacker than black, because the second order 
poles are very big. Instead the computer program has left these areas blank. 
The solid level curve is the one of the level at the 
zeros of $\wp '(z)$, the saddle points of $|\wp (z)|$. 
For our $u(z)=\wp _{_{P(q)}}(z)$ the lattice $P(z)$ is asymptotically 
equal to $q^{-1/6}$ times a standard regular hexagonal lattice 
$P_0$, and therefore $u(z)\sim q^{1/3}\,\wp _{_{P_0}}(z)$. 
It follows that the pits in the mountain landscape of $|u(z)|$ 
in which the zeros of $u(z)$ ly are separated from each other by 
mountain ridges in which the heights of the passes, situated 
at the zeros of $\op{d}\! u(z)/\op{d}\! z$, are of 
large order $|q|^{1/3}$.  

\begin{figure}[p]
\centering
\includegraphics[width=9cm]
{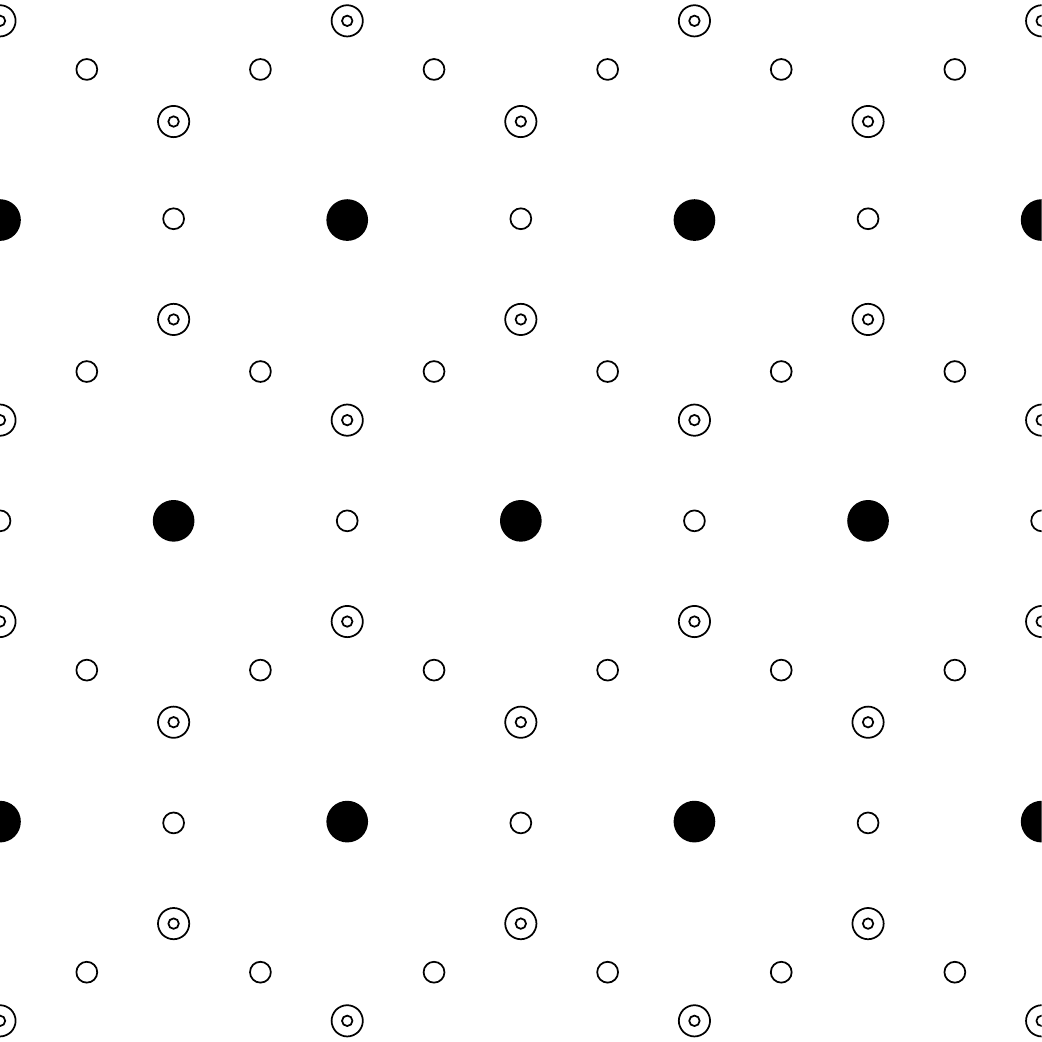}
\caption{Near the components of the infinity set} 
\label{latticefig}
\end{figure}

\begin{figure}[p]
\centering
\includegraphics[width=9cm]
{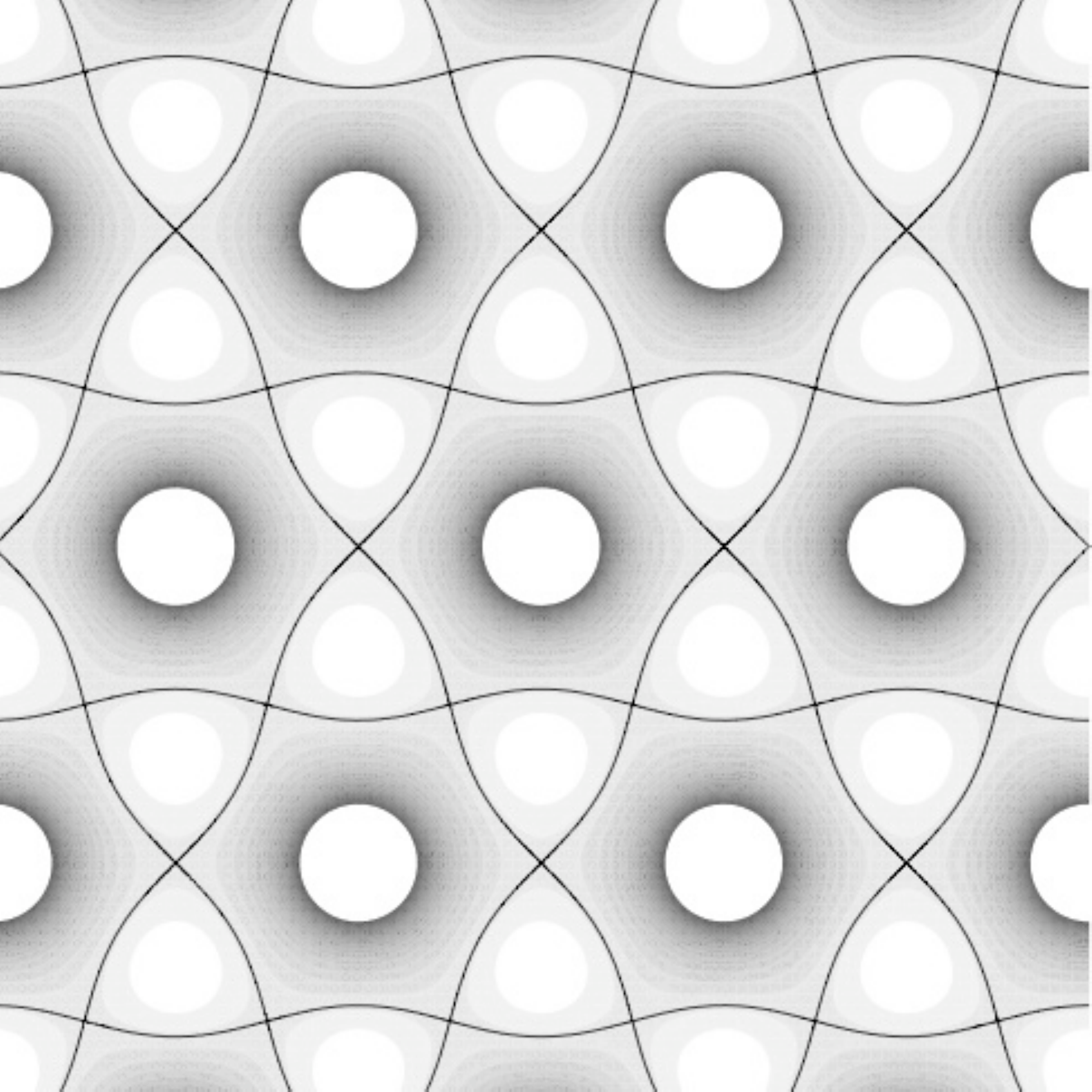}
\caption{The absolute value of the Weierstrass $\wp$ function} 
\label{wplevelsfig}
\end{figure}

The following lemma implies that the solutions near the part 
$L_3^{(6)}\setminus (L_0^{(9)}\cup L_2^{(7)}\cup 
L_4^{(5)})$ of $I$ 
of the non\--autonomous Boutroux\--Painlev\'e system 
closely follow the solutions of the autonomous 
Hamiltonian system. 
\begin{lemma}
Let $K$ be a compact subset of 
$L_3^{(6)}\setminus (L_0^{(9)}\cup L_2^{(7)}\cup 
L_4^{(5)})$ and $R\in\R _{>0}$. Then there exists a neighborhood 
$U$ of $K$ in $S_9$ and a constant $C$ such that 
the distance between $x(z)$ and $x_0(z)$ is 
$\leq C\, |z_0|^{-1}\, |q|^{-1/6}$ if $|z_0|\geq C$, 
$|z-z_0|\leq R$,  
$x_0$ is a solution of the autonomous Hamiltonian system 
such that $x_0(z_0)\in U$ and $x_0(z)\in U$, and 
$x$ is the solution of the Boutroux\--Painlev\'e system 
such that $x(z_0)=x_0(z_0)$. 
\label{L3lem}
\end{lemma}
\begin{proof}
$L_3^{(6)}\setminus (L_0^{(9)}\cup L_2^{(7)}\cup 
L_4^{(5)})$ is visible in the $(u_{311},\, u_{312})$ chart, 
where it is determined by $u_{312}=0$ and parametrized 
by $u_{311}$. As $u_{311}=\infty$, $u_{311}=0$, and $u_{311}=4$ 
correspond to $L_2^{(9)}$, $L_0^{(9)}$, 
and $L_4^{(5)}$, respectively, we keep $u_{311}$ 
bounded and bounded away from $0$ and $4$. 
We multiply the vector field by the scalar factor 
$q^{-1/6}={u_{311}}^{1/2}\, u_{312}\, 
(u_{311}-4-2\, {u_{311}}^2\, {u_{312}}^4)^{-1/6}$, 
which amounts to a time reparametrization 
along each solution curve. Lemma \ref{qinftyperiodlem} 
implies that the distance between the periods 
of the solutions of the autonomous system is of order 
$|q|^{-1/6}$. Lemma \ref{dotq/qlem} implies that,  
along the solutions near $L_3^{(6)}$, we have  
$q(z)=q(z_0)\, (z/z_0)^{-6/5+\op{o}(1)}$, which  
is asymptotically constant for $|z-z_0|<<|z_0|$.  
This leads to the time rescaled system 
\begin{equation}
\begin{array}{lll}
\op{d}\! u_{311}/\op{d}\! t
&=&{u_{311}}^{1/2}\, (12-3\, u_{311}+2\, {u_{311}}^2\, {u_{312}}^4)
\, (u_{311}-4-2\, {u_{311}}^2\, {u_{312}}^4)^{-1/6}\\
&\sim& -3\, {u_{311}}^{1/2}\, (u_{311}-4)^{5/6},\\
\op{d}\! u_{312}/\op{d}\! t
&=&{u_{311}}^{-1/2}\, u_{312}\, 
(-6+u_{311}-{u_{311}}^2\, {u_{312}}^4)
\, (u_{311}-4-2\, {u_{311}}^2\, {u_{312}}^4)^{-1/6}\\
&&+(5\, z)^{-1}\, {u_{311}}^{1/2}\, {u_{312}}^2
\, (u_{311}-4-2\, {u_{311}}^2\, {u_{312}}^4)^{-1/6}. 
\end{array}
\label{31rescaledsystem}
\end{equation}
Actually, we have six vector fields, one for each of the sixth 
roots of $1/q$, corresponding to the fact that the curves 
$q$ equal to a large constant pass six times near 
$L_3^{(6)}\setminus (L_0^{(9)}\cup L_2^{(7)}\cup L_4^{(5)})$. 
The autonomous Hamiltonian system is obtained 
from (\ref{31rescaledsystem} by deleting 
the $z$\--dependent term 
from the equation for  $\op{d}\! u_{312}/\op{d}\! t$. 
As $u_{312}\sim q^{-1/6}\, {u_{311}}^{-1/2}\, (u_{311}-4)^{1/6}$, 
the difference between the two vector fields 
is $\sim (5\, z)^{-1}\, q^{-1/3}\, 
{u_{311}}^{-1/2}\, (u_{311}-4)^{1/6}
=\op{O}(z^{-1}\, q^{-1/3})$, and therefore 
$|x(z)-x_0(z)|=\op{O}(q^{-1/3}\, z^{-1}\, (t-t_0))$, 
where $t-t_0=\op{O}(q^{1/6}\, (z-z_0))$. 
\end{proof}

It follows from the first equation in 
(\ref{31rescaledsystem}) that, for the solutions 
near $L_3^{(6)}\setminus (L_0^{(9)}\cup L_2^{(7)}\cup L_4^{(5)})$ 
of both the non\--autonomous and the autonomous system, 
the $t$\--times needed to go from a position with 
first coordinate $u_{311}$ to a position near 
$L_3^{(6)}\cap L_0^{(9)}$, $L_3^{(6)}\cap L_2^{(7)}$, 
and $L_3^{(6)}\cap L_4^{(5)}$ are asymptotically equal to 
the integral of 
$\, -3^{-1}\, U^{-1/2}\, (U-4)^{-5/6}$ 
over the $U$\--interval from $u_{311}$ to $0$, $\infty$, 
and $4$, respectively. Therefore the distances 
between the discs $D_i$ for $x$ and the discs $D_i$ 
for $x_0$ are of smaller order than $|q|^{1/6}$. 
As the distances between the discs for $x_0$ 
are of order $|q|^{1/6}$, where near $L_3^{(6)}$ we 
have $q=d$, see Lemma \ref{E8E4simlem},  the discs for $x$ have 
asymptotically the same relative position as the discs 
for $x_0$. Therefore Lemma \ref{L3lem} leads 
to a description of the solutions of the 
Boutroux\--Painlev\'e system near $I$ 
which closely resembles the description of the 
the solutions of the autonomous Hamiltonian 
system near $S_{9,\infty}$. Lemma \ref{dotq/qlem} implies 
for the solutions 
of the Boutroux\--Painlev\'e system near $L_3^{(6)}$ we have 
$q(z)\sim q(z_0)\, (z/z_0)^{-6/5+\op{o}(1)}$ hence the order 
$|q(z)|^{-1/6}\sim |q(z_0)|^{-1/6}\, (|z|/|z_0|)^{1/5+\op{o}(1)}$ 
of the distances between the discs 
increases slowly with growing $|z|\geq |z_0|>>1$, 
until the solution has left the neighborhood of 
$I$ in $S_9$ where the estimates hold.  

\begin{remark}
Boutroux \cite[bottom of p. 310]{boutroux} 
claimed that the quantity $(Y')^2-4\, Y^3+12\, Y$, 
see Remark \ref{boutrouxjkconstantsrem}, remains bounded 
when $X$ runs to infinity along any path with bounded 
argument not passing through any pole point of $Y(X)$. 
As the function $q$ has a simple pole at every pole point 
of $u(z)$, where $u_1(z)\, u_2(z)$ and 
$u_1(z)^2$ have poles of order 5 and 4, respectively, 
it follows from (\ref{qeqdotu}) that $(Y')^2-4\, Y^3+12\, Y$ 
has a pole of order 5 at every pole point of $Y(X)$, and therefore 
the claim of Boutroux can only be valid if the final $X$ 
stays sufficiently far away from the pole points of $Y(X)$. 
Boutroux referred for the proof to \cite[\S 10]{boutroux}, 
which in turn refers back to the estimates on 
\cite[p. 296, 297]{boutroux}. 
No proof is given for 
the existence of paths along which estimates of the form 
\cite[(17) on p. 296]{boutroux} hold, and of which the 
concatenation is the desired path in the $X$\--plane 
running to infinity. In particular no analysis is given of 
the existence of paths along which estimates of the form 
\cite[(17) on p. 296]{boutroux} hold when $|D|$ is large. 
Boutroux did not give an explicit $6/5$ power law as 
in Corollary \ref{6/5cor}, although the last term in the 
right hand side of \cite[(36) on p. 321]{boutroux} contains a 
factor $X^{-6/5}$ which Boutroux used in order to argue that 
the quantity $(Y')^2-4\, Y^3+12\, Y$ remains bounded. 

The asymptotic formula of Joshi and Kruskal 
\cite[(5.18) with $c=6/5$]{jk92} 
implies that the solution ${\mathcal E}$ 
of the averaged equation 
for their quantity $E$, which corresponds to our energy 
function $E$, satisfies 
${\mathcal E}(z)\sim {\mathcal E}(z_0)\, (z/z_0)^{-6/5+\op{o}(1)}$ 
for large $|{\mathcal E}|$. Joshi and Kruskal \cite{jk92} did not 
provide estimates for $E$ in terms 
of ${\mathcal E}$. It follows from 
(\ref{qexpansion}) that $E$ has a pole of order one 
at the pole points $z=\zeta$, 
and Lemma \ref{L3lem} together with the paragraph 
preceding it imply that for the solutions near the 
infinity set the pole points form an approximate regular hexagonal 
lattice with distance between the pole points of order $|E|^{-1/6}$. 
Therefore we cannot have 
$E(z)\sim E(z_0)\, (z/z_0)^{-6/5+\op{o}(1)}$ 
for all large $|E|$. However, Lemma \ref{E8E4simlem} and 
Corollary \ref{6/5cor} imply this estimate  
if $z$ stays away from the pole points at a distance 
of larger order than $|d|$, when $|d|\sim |2\, E|^{-1}<<|E|^{-1/6}$. 
\label{infqlitrem}
\end{remark}

\section{Near the equilibria of the limit system}
\label{eqsec}
In this section, we consider the Boutroux-Painlev\'e system near the equilibria of its autonomous limit and prove several results about the solutions near these equilibria. These are the solutions called {\em tronqu\'ee} by Boutroux that are asymptotically free of poles near infinity in certain sectors. We use classical methods to determine these properties and end  with a determination of their sequences of poles near the boundaries of pole-free sectors.

The limit system 
\begin{equation}
\begin{array}{lll}
\dot{u}_1&=&u_2,\\
\dot{u}_2&=&6\, {u_1}^2+1,
\end{array}
\label{u01dotlim}
\end{equation}
obtained from (\ref{u01dot}) by substituting  
$z=\infty$, has two equilibrium points, 
determined by $u_1=\epsilon\,\op{i}/\sqrt{6}$, $u_2=0$, 
where $\epsilon\in\{ -1,\, 1\}$. 
The linearization of the vector field 
at these equilibrium points is 
given by $(\delta u_1,\,\delta u_2)\mapsto 
(\delta u_2,\, 12\, u_1\, \delta u_1)
=(\delta u_2,\, \epsilon\, 2\, \sqrt{6}\,\op{i}\,\delta u_1)$, 
which has the eigenvalues 
\begin{equation}
\lambda _{\pm}=\pm\, (24)^{1/4}\,\op{e}^{\pi\scriptop{i}\, 
(1/2-\epsilon /4)}, 
\label{lambdapm}
\end{equation}
with the corresponding eigenspaces determined by 
$\delta u_2=\lambda _{\pm}\, \delta u_1$. 

Because the vector field in $S_9(\infty )\setminus 
I(\infty )$ has no zeros on the 
pole line $L_9(\infty )$, and $S_9(\infty )\setminus 
(I(\infty )\cup L_9(\infty ))$ 
is equal to the coordinate neighborhood where 
$u_1$ and $u_2$ are finite, these are the only 
equilibrium points of the autonomous Hamiltonian 
system in $S_9(\infty )\setminus 
I(\infty )$. The values of the function $q:=2\, E$ 
in (\ref{qeq}) at the equilibrium points are equal to 
$-\epsilon\, (2/3)^{3/2}\op{i}$, and the curves 
$q=\, -\,\epsilon\, (2/3)^{3/2}\op{i}$ are the only singular 
level curves of the function $q$ in $S_9(\infty )\setminus 
I(\infty )$. Both these singular fibers 
of the elliptic fibration $q:S_9(\infty )\to\Proj ^1$ 
are of Kodaira type $\op{I}_1$. The fiber 
$q=\infty$, equal to the infinity set $I(\infty )$ 
of the Hamiltonian vector field defined by the Hamiltonian function 
$E=q/2$, is of Kodaira type $\op{II}^*$.  
The configuration of the singular fibers 
of the elliptic fibration $q:S_9(\infty )\to\Proj ^1$ 
is $\,\op{II}^*+\op{I}_1+\op{I}_1$.  

\subsection{Perturbation of a system with a hyperbolic 
equilibrium point}
\label{perturbss} 
With the substitution $t=\lambda_+\, z$ and an affine 
change of coordinates $(u_1,\, u_2)\mapsto (p^+,\, p^-)$ 
which maps $(\epsilon\,\op{i}/\sqrt{6},\, 0)$ to $(0,\, 0)$ 
and the eigenvectors of the linearization of (\ref{u01dotlim}) 
at $(\epsilon\,\op{i}/\sqrt{6},\, 0)$ for the eigenvalues 
$\lambda _+$ and $\lambda _-$ to $(1,\, 0)$ and $(0,\, 1)$, 
respectively, the system (\ref{u01dot}) is transformed 
to a system 
\begin{equation}
\op{d}\! p/\op{d}\! t=v(t^{-1},\, p)   
\label{cssystem}
\end{equation} 
such that the right hand side satisfies the conditions 
in Lemma \ref{csestlem} below. 
In the remainder of this section we discuss arbitrary systems 
(\ref{cssystem}) which satisfy the conditions 
in Lemma \ref{csestlem}.  
In the lemmas \ref{stablelem} -- \ref{bdyasymlem}, we  
describe the solutions $p(t)$ of (\ref{cssystem}) 
which remain bounded for all $t$ 
in an unbounded domain in the complex plane, 
where the domain is increased step by step. 

The conclusions of the lemmas \ref{stablelem} -- \ref{bdyasymlem} 
follow from O. and R. Costin \cite[Th. 1 and 2]{costin2}, 
which deals with systems (\ref{cssystem}) in arbitrary dimensions 
$n$, where $v$ is complex analytic on an 
open neighborhood of $(0,\, 0)$ in $\C\times\C ^n$, 
$v(0,\, 0)=0$, and less special assumptions are made 
on $L_0:=\partial v(0,\, p)/\partial p|_{p=0}$. 
The proof of \cite[Th. 1]{costin2} uses Borel summation. 
Our proofs, a sequence  
of variations on the method of Cotton, 
are more classical, where our step by step approach 
may be helpful in understanding all the aspects of 
the final picture. 
As we only need the case 
that $n=2$ and $L_0$ has the eigenvalues $\pm 1$, 
and a wide generalization would require more elaborate 
notations and proofs, we did not attempt to write down 
the latter. In the next paragraph we summarize the 
results of the lemmas \ref{stablelem} -- \ref{bdyasymlem}.  

In Lemma \ref{stablelem} we prove that the solutions $p(t)$ 
of (\ref{cssystem}) which remain small on half lines 
$t_0+\R _{\geq 0}$ form a 
one\--parameter family parametrized by the complex number $p(t_0)^-$. 
Lemma \ref{centerlem} yields unique solutions 
$p_{\uparrow}(t)$ and $p_{\downarrow}(t)$ which remain small 
on horizontal axes 
in the upper and lower complex half plane, respectively. 
In Lemma \ref{centerVlem} it is established that for every 
$\epsilon >0$ there exists an $r>0$ such that the $p_{\uparrow}(t)$ 
have a common complex analytic extension to the domain defined 
by the inequalities  
$|t|>r$ and $\, -\pi/2 +\epsilon <\op{arg}t
<3\pi/2-\epsilon$, and that this common extension has 
an asymptotic expansion in strictly negative powers of $t$ 
as $|t|\to\infty$. Lemma \ref{alphalem} states that there exists  
an $\alpha\in\C$ such that for each 
solution $p(t)$ in Lemma \ref{stablelem} there exist  
$\eta, r>0$ such that $p(t)$ has a complex 
analytic extension to a small solution of (\ref{cssystem}) 
on the domain $R_{\eta ,\, r}$ 
defined by $|t|>r$, $\, -\pi <\op{arg}t<\pi$, 
and $|\tau (t)|<\eta$, where $\tau (t):=\op{e}^{-t}\, t^{\alpha}$. 
Furthermore, there exist unique $C\in\C$ and $d_j\in\C ^2$, 
$j\in\Z _{\geq 0}$, such that on every subdomain $\Sigma$ where 
$\op{e}^{-t}$ is of smaller order than every negative power 
of $t$ as $|t|\to\infty$ the function 
$p(t)-p_{\uparrow}(t)$ has the asymptotic expansion 
$p(t)-p_{\uparrow}(t)\sim C\,\op{e}^{-t}\, t^{\alpha}
\,\sum_{j=0}^{\infty}\, t^{-j}\, d_j$. Here the $d_j$ 
do not depend on the choice of the solution $p(t)$ in 
Lemma \ref{stablelem}. With a similar argument as in the 
proof of Lemma \ref{alphalem}, we obtain that in a subdomain 
where $C_1\, |t|^{-\epsilon _1}<|\op{e}^{-t}|<C_2\, |t|^{-\epsilon _2}$ 
for strictly positive $C_1,\, C_2,\, \epsilon _1,\,\epsilon _2$, 
we have an asymptotic expansion of the form 
$p(t)\sim\sum_{h,\, i\in\Z _{\geq 0}}\,\tau (t)^h\, t^{-i}\, p_{h,\, i}$  
as $|t|\to\infty$. Here the coefficients 
$p_{h,\, i}$ satisfy $p_{h,\, i}=C^h\, c_{h,\, i}$, 
where the coefficients $c_{h,\, i}$ do not depend on the choice 
of the solution $p(t)$ in Lemma \ref{cssystem}. 
Lemma \ref{Fjlem} states that 
for each $i$ the series $F_i(\tau ):=\sum_{h\geq 0}\,\tau ^h\, p_{h,\, i}$ 
converges for sufficiently small $|\tau |$, when Lemma 
\ref{bdyasymlem} implies that we have an asymptotic expansion 
of the form $p(t)\sim\sum_{i\geq 0}\, F_i(\tau (t))\, t^{-i}$ 
as $|t|\to\infty$ in a domain where $\op{Im}t\geq 0$ and 
$F_0(\tau (t))$ remains bounded, extending far into the domain where 
$p(t)$ is bounded away from zero. This expansion will 
lead to the asymptotic determination of 
\cite[Proposition 15]{costin2} of the sequence of 
poles of the truncated solutions of (\ref{PI}) in 
Lemma \ref{polelem}. At some places our statements 
are more precise and our proofs more complete than 
those in \cite{costin2}. 

\begin{lemma}
Assume that $v=(v^+,\, v^-)$ is a $\C ^2$\--valued 
complex analytic function on an open neighborhood $D$ of the 
origin in $\C ^3$ such that 
$v^{\pm}(u,\, p)=\,\pm\, p^{\pm}+w(u,\, p)$, 
$w(0,\, 0)=0$, and $\partial w(0,\, p)/\partial p|_{p=0}=0$.  
Here $p=(p^+,\, p^-)\in\C ^2$ and 
$\| p\| :=\max\{ |p^+|,\, |p^-|\}$. 
Then there exist  
strictly positive real numbers $\delta _0$, $\epsilon _0$, 
$C_1$, $C_2$, $C_3$, and $C_4$ 
such that 
$\| w(u,\, p)\|\leq 
C_1\,\| p\| ^2+C_2\, |u|$, and $\|\partial w(u,\, p)/\partial p\| 
\leq C_3\,\| p\| +C_4\, |u|$  
if $\| p\| <\delta _0$ and $|u|<\epsilon _0$. Here the 
last condition implies the preceding one for $C_1=C_3/2$ 
and a suitable $C_2$. In the sequel we will take 
$D$ equal to the set of all $(u,\, p)\in\C\times\C ^2$ 
such that $|u|<\epsilon _0$ and $\| p\| <\delta _0$. 
For solutions $p$ of (\ref{cssystem}) we will 
always require that $|t|>1/\epsilon _0$ and 
$\| p(t)\| <\delta _0$ for all $t$ in the domain of definition 
of $p$.  
\label{csestlem}
\end{lemma} 
\begin{proof}
We have $\| w(u, p)\|\leq \| w(u,\, p)-w(0,\, p)\| 
+\| w(0,\, p)\|$ and, with the notation 
$\partial _2w(u, p)=\partial w(u,\, p)/\partial p$, 
$\| \partial _2w(u,\, p)\|\leq 
\|\partial _2w(u,\, p)-\partial _2w(0,\, p)\| 
+\|\partial _2w(0,\, p)\|$. Because $w(0,\, 0)=0$ 
and $\partial _2w(0,\, 0)=0$, an application of a 
Taylor expansion with 
estimate for the remainder term to each of the 
terms between the norm signs yields the 
estimates in Lemma \ref{csestlem}. 
\end{proof} 
\begin{lemma}
In the situation of Lemma \ref{csestlem}, 
let $0<\delta <\delta _0$, $0<\epsilon <\epsilon _0$, $t_0,\, a^-\in\C$, 
$|t|\geq 1/\epsilon$ for every $t\in T:=t_0+\R _{\geq 0}$, 
$|a^-|+C_1\,\delta ^2+C_2\,\epsilon <\delta$, and 
$C_3\,\delta +C_4\,\epsilon <1$. 
Then there exists 
a unique $a^+=a^+_{t_0,\, a^-}\in\C$ 
such that the solution $p(t)$ 
of the differential equation (\ref{cssystem}) 
with $p(t_0)=(a^+,\, a^-)$ is defined for all $t\in T$ and 
satisfies $\| p(t)\|\leq\delta$ for every $t\in T$. 
For any $0<\delta <\min\{ \delta _0,\, 1/C_1,\, 1/C_3\}$, 
the set of all $(t_0,\, a^-)$ for which the 
above conclusions hold is a non\--empty open subset of 
$\C ^2$ on which the function  $(t_0,\, a^-)\mapsto a^+_{t_0,\, a^-}$ 
is complex analytic. Similar conclusions hold with 
$t_0+\R _{\geq 0}$ replaced by a curve 
$\tau +\op{i}\,\sigma (\tau )$, where 
$\tau\in\left[\op{Re}t_0,\,\infty\right[$, 
$\sigma :\left[\op{Re}t_0,\,\infty\right[\to\R$ 
is continuously differentiable with 
a bounded derivative, and the curve stays sufficiently 
far away from the origin. 
\label{stablelem}
\end{lemma}
\begin{proof}
We apply the {\em method of Cotton} \cite{cotton}. 
 
The system (\ref{cssystem}) can be viewed as  
an inhomogeneous system of linear differential equations 
with $w(t^{-1},\, p(t))$ as the inhomogeneous term, 
and therefore it is equivalent to 
the system of integral equations 
\begin{equation}
p(t)^{\pm}=\op{e}^{\pm (t-\tau )}\, p(\tau )^{\pm}
+\int_{\tau}^{t}\, 
\op{e}^{\pm (t-s)}\, w(s^{-1},\, p(s))^{\pm}\,\op{d}\! s.
\label{inteqpm}
\end{equation}
Let ${\mathcal X}$ denote the set of all continuous 
functions $p:T\to\C ^2$ such that  
$(t^{-1},\, p(t))\in D$ for every $t\in T$ 
and $t\mapsto w(t^{-1},\, p(t))^+$ is bounded on $T$. 
If the solution $p(t)$ of (\ref{cssystem}) belongs to ${\mathcal X}$,  
then we can let $\tau\in T$ run to infinity in the equation 
(\ref{inteqpm}) for $\pm = +$, and obtain 
\begin{equation}
p(t)^+=
\, -\int_t^{+\infty}\,\op{e}^{t-s}\, w(s^{-1},\, p(s))^+\,\op{d}\! s,
\label{inteq+2}
\end{equation}
where $\int_t^{+\infty}$ indicates the limit for 
$T\ni \tau\to\infty$ of $\int_t^{\tau}$. 
Conversely, if $a^-\in\C$ and $p\in {\mathcal X}$, 
then the equations (\ref{inteq+2}) and  
\begin{equation}
p(t)^-=\op{e}^{t
_0-t}\, a^-
+\int_{t_0}^{t}\, 
\op{e}^{s-t}\, w(s^{-1},\, p(s))^-\,\op{d}\! s 
\label{inteq-2}
\end{equation}
for all $t\in T$ imply that $p(t)$ is a solution of 
(\ref{cssystem}) such that $p(t_0)^-=a^-$. 

Let $F$ denote the integral operator which 
assigns to each $p\in {\mathcal X}$ the function 
$F(p):T\to\C ^2$
 such that, for each $t\in T$, 
$F(p)(t)^+$ and $F(p)(t)^-$ are equal to the right 
hand side of (\ref{inteq+2}) and (\ref{inteq-2}), respectively.  
Let $X$ denote the set of 
all continuous functions 
$p:T\to\C ^2$ such that 
$\| p(t)\|\leq\delta$ for every $t\in T$. 
Then $X\subset {\mathcal X}$, 
$X$ is a complete space with respect to the metric 
$d(p_1,\, p_2):=\sup_{t\in T}\,\| p_1(t)-p_2(t)\|$, 
and the assumed estimates imply that 
\begin{eqnarray}
|F(p)(t)^+|&\leq&\int_0^{\infty}
\,\op{e}^{-\tau}\, 
(C_1\,\delta ^2+C_2\,\epsilon )\,\op{d}\!\tau\leq\delta ,
\label{F+est}\\
|F(p)(t)^-|&\leq&\op{e}^{t_0-t}\, |a^-|+\int_0^{\infty}
\,\op{e}^{-\tau}\, 
(C_1\,\delta ^2+
C_2\,\epsilon )\,\op{d}\!\tau\leq\delta ,
\label{F-est}\\ 
|F(p_1)(t)^{\pm}-F(p_2)(t)^{\pm}|
&\leq&\int_0^{\infty}
\,\op{e}^{-\tau}\, (C_3\,\delta +C_4\,\epsilon )\, 
\| p_1(t_0+\tau )-p_2(t_0+\tau )\| \op{d}\!\tau 
\nonumber\\
&\leq& (C_3\,\delta +C_4\,\epsilon )\, d(p_1,\, p_2)
\label{dFest}
\end{eqnarray}
for every $p,\, p_1,\, p_2\in X$ and $t\in T$. 
Therefore $F(X)\subset X$  
and $F:X\to X$ is a contraction, with contraction 
factor $\leq C_3\,\delta +C_4\,\epsilon <1$, when 
the contraction mapping theorem 
implies that $F$ has a unique fixed point, 
a unique $p\in X$ such that $F(p)=p$. 
It follows that there is a unique solution  
$p(t)=p_{t_0,\, a^-}(t)$ of (\ref{cssystem}) on $T$ such that 
$p(t_0)^-=a^-$ and $\| p(t)\|\leq\delta$ for every $t\in T$. 
As the arbitrary solutions $p(t)$ of (\ref{cssystem}) 
are uniquely determined by their value $p(t_0)$ at $t=t_0$, 
it follows that for every solution $p(t)$ of (\ref{cssystem}) 
on $T$ with 
$p(t_0)^-=a^-$ and $p(t_0)^+\neq a^+(t_0,\, a^-):=p_{t_0,\, a^-}(t_0)^+$ 
there exists $t\in T$ such that $\| p(t)\| >\delta$. 
The assumptions remain verified upon small perturbations of 
$t_0$ and $a^-$, and an application of the implicit function theorem 
yields that the solution $p=p_{t_0,\, a^-}$ depends in a 
complex analytic way on $(t_0,\, a^-)$. 
This completes the proof of Lemma \ref{stablelem}. 
\end{proof}
\begin{lemma}
In the situation of Lemma \ref{csestlem}, 
let $0<\delta <\delta _0$, $0<\epsilon <\epsilon _0$, 
$C_1\,\delta ^2+C_2\,\epsilon <\delta$, and 
$C_3\,\delta +C_4\,\epsilon <1$. Then there 
exists a unique solution $p(t)=p_{\uparrow}(t)$ 
and $p=p_{\downarrow}(t)$ of (\ref{cssystem}) 
on $\R +\op{i}/\epsilon$ and $\R -\op{i}/\epsilon$ such that 
$\| p(t)\|\leq\delta$ for every $t\in\R +\op{i}/\epsilon$ 
and $t\in\R -\op{i}/\epsilon$, respectively. 
Similar conclusions hold with 
$\R\pm\op{i}/\epsilon$ replaced by a curve 
$\tau +\op{i}\,\sigma (\tau )$, where 
$\tau\in\R$, 
$\sigma :\R\to\R$ is continuously differentiable with 
a bounded derivative, and the curve stays sufficiently 
far away from the origin. 
\label{centerlem}
\end{lemma}
\begin{proof}
We apply the variation on \cite[p. 483]{cotton} 
of Cotton's method. 

Let $I=\R +\op{i}/\epsilon$, 
and ${\mathcal Y}$ the set of all continuous functions 
$p:I\to\C ^2$ such that $(t^{-1},\, p(t))\in D$ for every 
$t\in I$ and $t\mapsto w(t^{-1},\, p(t))$ is bounded on $I$. 
If the solution $p(t)$ of (\ref{cssystem}) 
belongs to ${\mathcal Y}$, then we can let 
$I\ni \tau\to +\infty$ in the equation (\ref{inteqpm}) 
for $\pm =+$ {\em and} 
let $I\ni\tau\to -\infty$ 
in (\ref{inteqpm}) for $\pm =\, -$, and obtain (\ref{inteq+2}) and 
\begin{equation}
p(t)^-=\int_{-\infty}^{t}\, 
\op{e}^{s-t}\, w(s^{-1},\, p(s))^-\,\op{d}\! s, 
\label{inteq-3}
\end{equation}
respectively. 
Conversely, if $p\in {\mathcal Y}$ satisfies 
(\ref{inteq+2}) and (\ref{inteq-3}), then it 
is a solution of (\ref{cssystem}). 

Let $G$ denote the integral operator which 
assigns to each $p\in {\mathcal Y}$ the function 
$G(p):I\to\C ^2$ such that, for each $t\in I$, 
$G(p)(t)^+$ and $G(p)(t)^-$ are equal to the right 
hand side of (\ref{inteq+2}) and (\ref{inteq-3}), respectively.  
Let $Y$ be the set of 
all continuous functions 
$p:I\to\C ^2$ such that 
$\| p(t)\|\leq\delta$ for every $t\in I$. 
Then $Y\subset {\mathcal Y}$, 
$Y$ is a complete space with respect to the metric 
$d(p_1,\, p_2):=\sup_{t\in U}\,\| p_1(t)-p_2(t)\|$, 
and estimates analogous to  
(\ref{F+est}), (\ref{F-est}), and (\ref{dFest}) 
imply that $G(Y)\subset Y$ and 
$G$ is a contraction in $Y$ 
with contraction factor $\leq C_3\,\delta +C_4\,\epsilon <1$. 
This time the contraction mapping theorem 
yields a unique solution 
$p:I\to\C ^2$ of (\ref{cssystem}) such that 
$\| p(t)\|\leq\delta$ for every $t\in I$. 
This completes the proof of Lemma \ref{centerlem}. 
\end{proof}

\begin{remark}
According to Anosov \cite[\S 4]{anosov}, 
\lq\lq Every five years or so, someone \lq\lq discovers\rq\rq\  the 
theorem of Hadamard and Perron, proving it either by 
Hadamard's method or by Perron's.\rq\rq\  The theorem alluded  
to is the stable manifold theorem in dynamical systems, 
\lq\lq Hadamard's method\rq\rq\  refers to Hadamard \cite{hadamard},  
and \lq\lq Perron's method\rq\rq\  to Perron \cite{perron}. However, 
Perron \cite[p. 130]{perron} made perfectly clear that he used 
the method of Cotton \cite{cotton}. 
The paper \cite{cotton} seems to be little known, 
for no good reason. 
\label{anosovrem} 
\end{remark}

\begin{lemma}
For each $m,\, n,\, r\in\R _{>0}$ such that $m>n$, 
let $\eta =\eta _{m,\, n,\, r}:\R\to\R$ be 
defined by 
$\eta (\tau )=(r^2-\tau ^2)^{1/2}$ when 
$|\tau |\leq r\, m\, (1+m^2)^{-1/2}$, 
$\eta (\tau )=r\, (1+m^2)^{1/2}-m\, |\tau |$ 
when $r\, m\, (1+m^2)^{-1/2}\leq |\tau |\leq 
r\, (1+m^2)^{1/2}/(m-n)$, 
and $\eta (\tau )=\, -n\, |\tau |$ when 
$|\tau |\geq r\, (1+m^2)^{1/2}/(m-n)$. 
Let 
$V:=\{ t\in\C\mid\op{Im}t\geq\eta (\op{Re}t)\}$, 
see Figure \ref{domainfig}. 

If $r$ is sufficiently large, then all solutions 
$p_{\uparrow}(t)$ in Lemma \ref{centerlem} have a 
common extension to a unique solution of 
(\ref{cssystem}), also denoted by $p_{\uparrow}(t)$, 
on the union ${\mathcal V}$ of $V$ with all their 
domains of definition. 
There exists $j\mapsto c_j:\Z _{>0}\to\C ^2$, 
such that the solution $p(t)=p_{\uparrow}(t)$ of 
(\ref{cssystem}) on ${\mathcal V}$ 
has the asymptotic expansion 
\begin{equation}
p(t)\sim\sum_{j=1}^{\infty}\, t^{-j}\, c_j
\label{powerasym}
\end{equation}
as $t\in {\mathcal V}$, $|t|\to\infty$. 
The asymptotic expansion (\ref{powerasym}) 
can be differentiated termwise in the sense that 
\begin{equation}
\frac{\op{d}\! p(t)}{\op{d}\! t}\sim
\sum_{j=1}^{\infty}\, -j\, t^{-j-1}\, c_j
 \label{powerasym'}
\end{equation}
as $t\in {\mathcal V}$,\, $|t|\to\infty$. 
All the coefficients $c_j$ are uniquely determined 
from the equations obtained by substituting 
(\ref{powerasym'}) and (\ref{powerasym}) 
in the left and right hand side of (\ref{cssystem}),  
respectively, and using the power series expansion 
\begin{equation}
v(t^{-1},\, p)=\sum_{i,\, j_+,\, j_-\in\Z _{\geq 0}}
\, t^{-i}\, (p^+)^{j_+}\, (p^-)^{j_-}\, v_{i,\, j_+,\, j_-}, 
\quad v_{i,\, j_+,\, j_-}\in\C ^2.
\label{vexpansion}
\end{equation} 
We have 
\begin{equation}
c_1=\, -\left(\left.\frac{\partial v(0, p)}{\partial p}\right| _{p=0}
\right) ^{-1}\,\left.\frac{\partial v(u,\, 0)}{\partial u}\right| _{u=0}.
\label{c1eq}
\end{equation}
Applying the complex conjugation $t\mapsto\overline{t}$, 
we obtain similar conclusions for the solutions 
$p_{\downarrow}(t)$ in Lemma \ref{centerlem}. 
\label{centerVlem}
\end{lemma}
\begin{proof}
For each $t\in V$ we define half\--lines $L_t^{\pm}$ 
emanating from $t$, as follows. 
If $\op{Re}t\geq 0$, then $L_t^+:=t+\R _{\geq 0}$. 
If $0\leq -\op{Re}t\leq m\,\op{Im}t$, 
then $L_t^+:=\{ t+\tau -\op{i}\, (\op{Re}t/\op{Im}t)\,\tau
\mid\tau\in\R _{\geq 0}\}$. 
If $\op{Re}t\leq 0$ and 
$m\,\op{Im}t\leq -\op{Re}t$, then 
$L_t^+:=\{ t+\tau +\op{i}\, m\,\tau\mid\tau\in\R _{\geq 0}\}$. 
Finally, $L_t^-:=R(L_{R(t)}^+)$ for every $t\in V$, 
where $R$ is the reflection $t\mapsto -\op{Re}t+\op{i}\,\op{Im}t$. 

It follows that, for every $t\in V$ and $s\in L_t^{\pm}$, 
$|t|\geq r$ and $|s|\geq (m-n)\, (1+m^2)^{-1/2}\, |t|$. 
See Figure \ref{domainfig}. 

\begin{figure}[ht]
\centering
\includegraphics[width=9cm]
{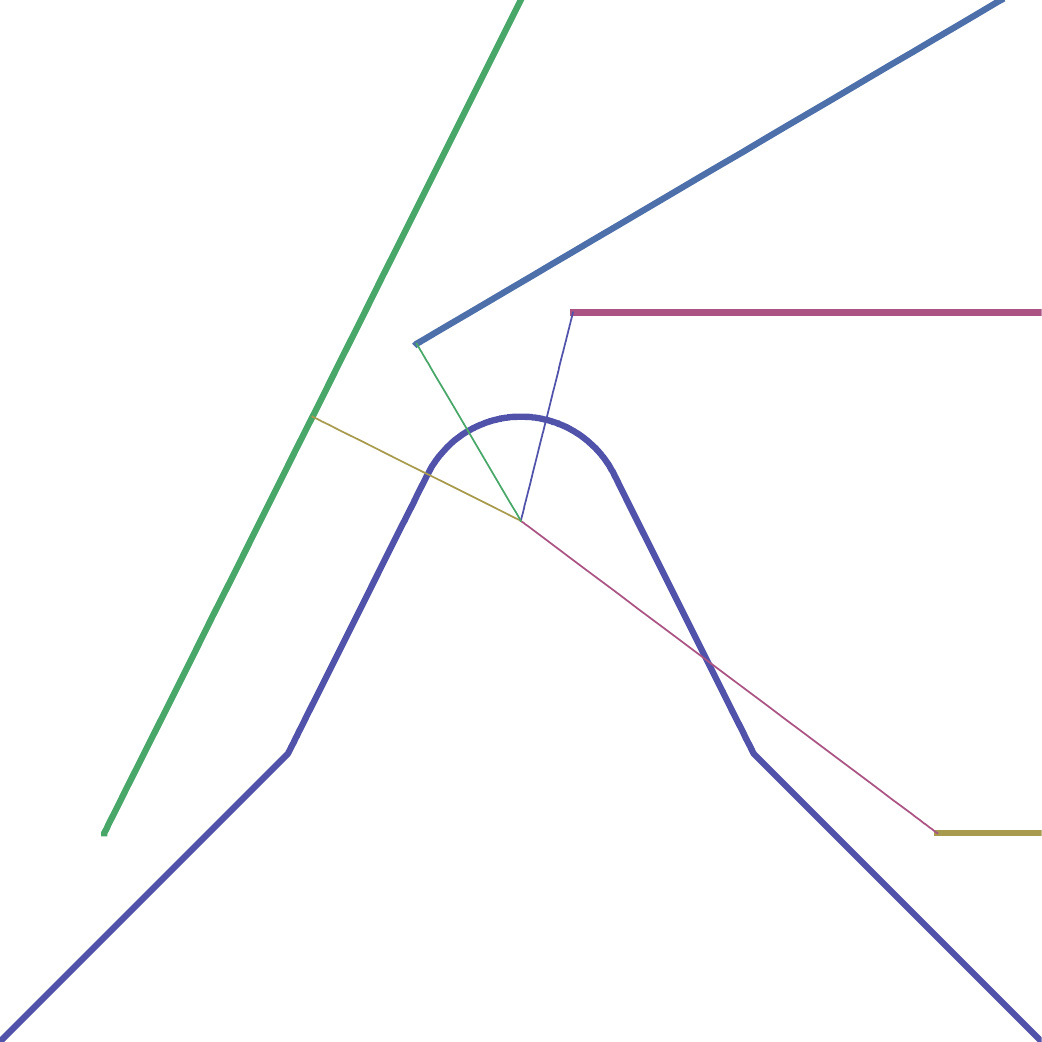}
\caption{The domain $V$, four half\--lines $L_t^+\subset V$, 
and their distances to the origin} 
\label{domainfig}
\end{figure}

On the space of functions $p:V\to\C ^2$ such that 
$p(t)=\op{O}(t^{-1/2})$ as $t\in V$, $|t|\to\infty$, 
we use the norm $\| p\| :=\sup_{t\in V}\, |t|^{1/2}\,\| p(t)\|$. 
Let $r>1/\epsilon _0$ and $0<\delta <r^{1/2}\,\delta _0$, 
and $Z$ the space of continuous functions $p:V\to\C ^2$, 
complex analytic on the interior of $V$, such that 
$\| p\|\leq\delta$. For each $p\in Z$ we define $G(p)$ as above, 
where the integrals in the right hand sides of 
(\ref{inteq+2}) and (\ref{inteq-3}) are over $L_t^+$ and 
$L_t^-$, respectively. Using the Cauchy integral theorem, 
we obtain that $G(p)$ is a complex analytic function 
on the interior of $V$, with derivative equal to  
$\op{d}\! G(p)(t)^{\pm}/\op{d}\! t=
\pm G(p)(t)^{\pm}+w^{\pm}(t^{-1},\, p(t))$. 
Because the velocities of the 
parametrizations of $L_t^{\pm}$ in the previous paragraph 
are at most $(1+m^2)^{1/2}$, we have the estimates 
\begin{eqnarray*}
|G(p)(t)^{\pm}|&\leq&  
(1+m^2)^{1/2}\, \sup_{s\in L_t^{\pm}}\, (C_1\,\delta ^2+C_2)\, |s|^{-1},\\
|G(p_1)(t)^{\pm}-G(p_2)(t)^{\pm}|&\leq&
(1+m^2)^{1/2}\, \,\sup_{s\in L_t^{\pm}}\, 
(C_3\, |s|^{-1/2}+C_4\, |s|^{-1})\,\| p_1-p_2\|\, |s|^{-1/2}
\end{eqnarray*} 
for every $p,\, p_1,\, p_2\in Z$ and $t\in V$. 
Because $|s|^{-1}\leq (m-n)^{-1}\, (1+m^2)^{1/2}\, |t|^{-1}$ 
and $|s|^{-1}\leq r^{-1}$ for every $s\in L_t^{\pm}$, it follows that 
$\| G(p)\|\leq (1+m^2)^{3/4}\, (m-n)^{-1/2}\, (C_1\,\delta ^2+C_2)\, r^{-1/2}$ 
and $\| G(p_1)-G(p_2)\|\leq 
(1+m^2)^{3/4}\, (m-n)^{-1/2}\, (C_3\, r^{-1/2}+C_4\, r^{-1})\,\| p_1-p_2\|$. 
Therefore, if $r>1/\epsilon _0$, $0<\delta <r^{1/2}\,\delta _0$
$(1+m^2)^{3/4}\, (m-n)^{-1/2}\, (C_1\,\delta ^2+C_2)\, r^{-1/2}\leq\delta$ 
and $(1+m^2)^{3/4}\, (m-n)^{-1/2}\, (C_3\, r^{-1/2}+C_4\, r^{-1})<1$, 
which for any $m,\, n,\,\delta\in\R _{>0}$ 
can be arranged by taking $r$ sufficiently 
large, we have $G(Z)\subset Z$, and $G:Z\to Z$ is a contraction. 
This time the contraction theorem implies that there is a unique solution 
$p=p_V:V\to\C ^2$ of (\ref{cssystem}) such that 
$\| p(t)\|\leq\delta\, |t|^{-1/2}$ for every $t\in V$. 

For any $a\in\C$, successive integrations by parts 
yield the asymptotic expansions 
\[
\int_t^{\infty}\,\op{e}^{t-s}\, s^{-a}\,\op{d}\! s 
\sim \sum_{k=0}^{\infty}\, (-1)^k\, \pi _k\, t^{-a-k}
\quad\mbox{\rm and}\quad
\int_{-\infty}^t\,\op{e}^{s-t}\, s^{-a}\,\op{d}\! s 
\sim \sum_{k=0}^{\infty}\, \pi _k\, t^{-a-k}
\]
for $|t|\to\infty$, where $\pi _k:=\prod_{j=0}^{k-1}\, (a+j)$. 
This is one of the oldest examples of 
asymptotic expansions which do not converge.  
Substituting the 
inequality $\| w(s^{-1},\, p(s))\|\leq C_1\, \| p(s)\| ^2
+C_2\, |s|^{-1}\leq 
(C_1\,\delta ^2+C_2)\, |s|^{-1}$ in the integrals in the right 
hand sides of $p(t)^{\pm}=G(p)(t)^{\pm}$, we obtain that 
$p(t)=\op{O}(t^{-1})$ as $t\in V$, $|t|\to\infty$. 
Substituting 

\begin{equation}
p(t)=\sum_{j=1}^{N-1}\, t^{-j}\, c_j+\op{O}(t^{-N})
\label{powerasymN}
\end{equation}  
for $t\in V$,  
the power series expansion $w(u,\, p)=\sum_{i,\, j_+,\, j_-\geq 0}
\, w_{i,\, j_+,\, j_-}\, u^i\, (p^+)^{j_+}\, (p^-)^{j_-}$, 
and $u=t^{-1}$, $p=p(t)$ in the integrals in the right 
hand sides of $p(t)^{\pm}=G(p)(t)^{\pm}$, we obtain 
(\ref{powerasymN}) with $N$ replaced by $N+1$ and a 
unique $c_N\in\C ^2$. 
This inductive procedure yields that $p(t)$ has 
an asymptotic expansion of the form (\ref{powerasym}) 
for $t\in V$, $|t|\to\infty$. 
The substitution of $u=t^{-1}$ and $p=p(t)$  in the power series 
expansion of $v(u,\, p)$ implies that there exist $b_i\in\C ^2$
such that $p'(t)=v(t^{-1},\, p(t))\sim\sum_{i=1}^{\infty}\, t^{-i}\, b_i$ 
as $t\in V$, $|t|\to\infty$, 
when $\int_{t_0}^t\, p'(s)\,\op{d}\! s+p(t_0)=p(t)
\sim\sum_{j=1}^{\infty}\, t^{-j}\, c_j$ yields that 
$b_i=\, -(i-1)\, c_{i-1}$ for every $i\in\Z _{\geq 1}$. 
This proves (\ref{powerasym'}) for $p=p_V$. 

We will prove next that all solutions 
$p_{\uparrow}$ in Lemma \ref{centerlem} 
and $p_V$ have a common extension to a solution 
of (\ref{cssystem}) on ${\mathcal V}$.   
The conditions in Lemma \ref{centerlem} hold if and only if 
\[
\epsilon <\min\{\epsilon _0,\, 1/C_4,\, 1/4\, C_1\, C_2\}
\quad\mbox{\rm and}\quad  
\delta _-(\epsilon )<\delta <\min\{ \delta _0,\, \delta _+(\epsilon )\} ,
\] 
where $\delta _{\pm}(\epsilon ):=
(1\pm (1-4\, C_1\, C_2\,\epsilon )^{1/2})/2\, C_1$. 
If in the proof of Lemma \ref{centerlem} we replace $Y$ by the space 
$Y_{\epsilon ,\,\delta }$ of all continuous functions 
$p$ on the upper half plane 
$U_{\epsilon}:=\{ t\in\C\mid\op{Im}t\geq 1/\epsilon\}$, 
complex analytic in the interior of $U_{\epsilon}$ and 
with $\| p(t)\|\leq\delta$ for all $t\in U_{\epsilon}$, 
then $G$ is a contraction on $Y_{\epsilon ,\,\delta}$ 
and therefore has a unique fixed point 
$p_{\epsilon ,\,\delta}\in Y_{\epsilon ,\,\delta}$. 
The Cauchy integral theorem implies that 
the restriction of $p_{\epsilon ,\,\delta}$ to the interior 
of $U_{\epsilon}$ 
is a complex analytic solution of (\ref{cssystem}). 
As the restriction of $p_{\epsilon ,\,\delta}$ 
to $\R +\op{i}/\epsilon$ is a fixed point of 
$G:Y\to Y$, and the latter is unique, and equal to the 
solution $p_{\uparrow}$ in Lemma \ref{centerlem}, it follows that 
the solution $p_{\uparrow}(t)$ in Lemma \ref{centerlem} 
extends to a solution $p=p_{\uparrow ,\,\epsilon}$ 
of (\ref{cssystem}) on $U_{\epsilon}$, complex analytic 
on the interior of $U_{\epsilon}$, and satisfying 
$\| p(t)\|\leq\delta$ for every $t\in U_{\epsilon}$. 
Note that if $0<\epsilon '<\epsilon$, then the conditions 
in Lemma \ref{centerlem} hold with $\epsilon$ replaced by $\epsilon '$, 
and the restriction to $U_{\epsilon '}$ of 
$p_{\uparrow ,\,\epsilon}$ is equal to 
$p_{\uparrow ,\,\epsilon '}$. 

If $p(t)$ is any 
solution of (\ref{cssystem}) on some upper half plane 
$U$, and $p(t)\to 0$ as $t\in U$, 
$|t|\to\infty$, then there exist $\delta ,\, \epsilon$ 
as above such that $U_{\epsilon}\subset U$ and 
$\| p(t)\|\leq\delta$ for every $t\in U_{\epsilon}$, 
and therefore $p|_{U_{\epsilon}}=p_{\uparrow ,\,\epsilon}$. 
As the sector\--like domain $V$ contains an upper 
half plane $U_{\epsilon}$, and $p_V(t)\to 0$ 
as $t\in V$, $|t|\to\infty$, it follows that 
$(p_V)|_{U_{\epsilon}}=p_{\uparrow ,\,\epsilon}$. 
This completes the proof that all solutions 
$p_{\uparrow}$ in Lemma \ref{centerlem} extend to 
a common solution $p(t)$ of (\ref{cssystem}) on ${\mathcal V}$, 
of which the restriction to $V$ is equal to $p_V$, and 
therefore $p$ has all the properties mentioned in 
Lemma \ref{centerVlem}. 
\end{proof}
\begin{lemma}
In the situation of Lemma \ref{csestlem}, 
let $\alpha$ denote the right lower 
corner in the $2\times 2$\--matrix 
\begin{equation}
L_1=\left.\frac{\partial ^2v(u,\, p)}{\partial u\,\partial p}
\right| _{u=0,\, p=0}
+\left.\frac{\partial ^2v(0,\, p)}{\partial p^2}
\right| _{p=0}\, c_1,  
\label{L1eq}
\end{equation}
with $c_1$ defined by (\ref{c1eq}).  
For $\eta ,\, r\in\R _{>0}$, let  
\begin{equation}
R_{\eta ,\, r}:=\{ t\in\C\mid\; |t|\geq r\;\mbox{\rm and}\;
|\op{e}^{-t}\, t^{\alpha}|\leq\eta
\} , 
\label{Rdef}
\end{equation}
where $t^{\alpha}=\op{e}^{\alpha\log t}$, 
$\log t=\log |t|+\op{i}\op{arg}t$, and 
$\, -\pi\leq\op{arg}t\leq\pi$. 
For every solution $p(t)=p_{t_0,\, a^-}(t)$ 
and $p(t)=p_{\uparrow}(t)$ of (\ref{cssystem}) 
in Lemma \ref{stablelem} and Lemma \ref{centerVlem}, 
there exist $\eta ,\, r\in\R _{>0}$ such that $p(t)$ extends  
to a solution of (\ref{cssystem}) on $R_{\eta ,\, r}$,  
which extension is again denoted by $p_{t_0,\, a^-}$ 
and $p_{\uparrow}$, respectively. 
There exist $j\mapsto d_j:\Z _{\geq 0}\to\C ^2$ with 
$d_0=(0,\, 1)$, and $C=C_{t_0,\, a^-}\in\C$, such that 
on any subdomain $\Sigma$ of $R_{\eta ,\, r}$ on which 
$\op{Re}t/\log |t|\to +\infty$ as $t\in\Sigma$, $|t|\to\infty$, 
we have the asymptotic expansion 
\begin{equation}
p(t)-p_{\uparrow}(t)\sim C\, \op{e}^{-t}\, t^{\alpha}
\sum_{j=0}^{\infty}\, t^{-j}\, d_j
\label{expasym}
\end{equation} 
as $t\in\Sigma$, $|t|\to\infty$. The asymptotic expansion 
(\ref{expasym}) can be differentiated termwise 
in the sense that 
\begin{equation}
p'(t)-p_{\uparrow}'(t)\sim C\, \op{e}^{-t}\, t^{\alpha}
\sum_{j=0}\, t^{-j}\, (-d_j+(\alpha -j+1)\, d_{j-1})
\label{expasym'}
\end{equation} 
as $t\in S$, $|t|\to\infty$. All the coefficients $d_j$ 
are uniquely determined from the equations $d_0=(0,\, 1)$ 
and $-d_j+(\alpha -j+1)\, d_{j-1}=\sum_{i=0}^j\, 
L_i\, d_{j-i}$ for all $j\geq 1$, where the $2\times 2$\--matrices 
$L_i$ are determined from the asymptotic expansion 
\begin{equation}
L(t):=\left.\frac{\partial v(t^{-1},\, p)}{\partial p}
\right|_{p=p_{\uparrow}(t)}\sim\sum_{i=0}^{\infty}
\, t^{-i}\, L_i
\label{Lt}
\end{equation}
as $t\in V$, $|t|\to\infty$. 
Here 
\begin{equation}
L_0:=\left.\frac{\partial v(0,\, p)}{\partial p}\right| _{p=0}
=\left(\begin{array}{cc}1&0\\0&-1\end{array}\right) ,
\label{L0eq}
\end{equation} 
and $L_1$ is given by (\ref{L1eq}). 
The asymptotic expansion (\ref{powerasym}) for 
$p(t)=p_{\uparrow}(t)$ holds for $t\in\Sigma$, $|t|\to\infty$, 
and can be differentiated termwise there. 
In combination with (\ref{expasym}), 
(\ref{expasym'}), it follows that 
$p(t)=p_{t_0,\, a^-}(t)$ has the same asymptotic expansions  
(\ref{powerasym}), (\ref{powerasym'}) for 
$t\in\Sigma$, $|t|\to\infty$.  

The complex number $C=C_{t_0,\, a^-}$, determined by 
\[
\lim_{t\in\Sigma ,\, |t|\to\infty}\, 
\op{e}^t\, t^{-\alpha}\, (p(t)-p_{\uparrow}(t))=(0,\, C),
\]
depends in a complex analytic way on $(t_0,\, a^-)$. 
Conversely, for every $C\in\C$ there exist $\eta ,\, r\in\R _{>0}$ 
and a unique solution $p=p_C:R_{\eta,\, r}\to\C ^2$ of 
(\ref{cssystem}) such that (\ref{expasym}) holds, 
where $p_C(t)$ depends in a complex analytic way on $C$.  
If we choose $t_0$ such that $\| p(t_0)\|$ is sufficiently small, 
then $p=p_{t_0,\, a^-}$ with $a^-=p(t_0)^-$, 
and $a^-\mapsto C_{t_0,\, a^-}$
is a complex analytic diffeomorphism from 
its open domain of definition onto an open subset of $\C$.
In particular we have $p(t)\equiv p_{\uparrow}(t)$ 
if and only if $C=0$ if and only if 
there exists a sequence $t_j$ in $\C$ such that 
$|\op{arg}(t_j)|<\pi$, 
$\op{e} ^{-t_j}\, {t_j}^{\alpha}\to 0$ and 
$\op{e} ^{t_j}\, {t_j}^{-\alpha}\, (p(t_j)-p_{\uparrow}(t_j))\to 0$  
as $j\to\infty$. 

Finally, for every $C\in\C$ and $\varepsilon >0$ there exist 
$\eta ,\, r\in\R _{>0}$ such that 
\begin{equation}
\left\|\op{e}^t\, t^{-\alpha}\, (p_C(t)-p_{\uparrow}(t))-(0,\, C)\right\| 
\leq\varepsilon
\label{nonsmall}
\end{equation}
for every $t\in R_{\eta ,\, r}$. 
In particular, if $C\neq 0$ and we choose 
$0<\varepsilon <|C|$, then $p_C(t)$ is bounded away 
from zero on the set of all $t\in\C$ such that 
$|\op{e}^{-t}\, t^{\alpha}|=\eta$, $|t|\geq r$, and 
$\op{Im}t>0$, the part of the boundary of $V_{\eta ,\, r}$ 
in the upper half plane and sufficiently far away from the origin. 
A similar statement holds in the lower half plane 
with $p_{\uparrow}(t)$ replaced by $p_{\downarrow}(t)$. 
\label{alphalem}
\end{lemma}
\begin{proof}
Let $W$ be a suitable subdomain of $V$ and $p:W\to\C ^2$ 
a solution of (\ref{cssystem}). Then $y(t):=p(t)-p_{\uparrow}(t)$ 
is a solution of the differential equation 
\begin{equation}
\op{d}\! y/\op{d}\! t=f(t^{-1},\, y):=
v(t^{-1},\, p_{\uparrow}(t)+y)-v(t^{-1},\, p_{\uparrow}(t)), 
\label{ysystem}
\end{equation} 
where $(t,\, y)\mapsto f(t^{-1},\, y)$ is a $\C ^2$\--valued 
complex analytic function on 
$W\times B$, if  
$B:=\{ y\in\C ^2\mid \| y\| <\delta_0'\}$, where 
$\delta _0':=\delta _0-\sup_{t\in W}\, \| p_{\uparrow}(t)\| 
\geq\delta _0-\delta\, r^{-1/2}$. 
Furthermore, 
\begin{equation}
f(t^{-1},\, y)\sim\sum_{j=0}^{\infty}\, t^{-j}\, f_j(y)
\label{fpowerasym}
\end{equation}
for $t\in V$, $|t|\to\infty$, where each of the functions 
$f_j$ is complex analytic on $B$. 
We have $f(t^{-1},\, 0)\equiv 0$, hence $f_j(0)=0$ for every $j$. 
It follows from 
(\ref{Lt}) that $\partial f_j(y)/\partial y|_{y=0}=L_j$ 
for every $j\in\Z _{\geq 0}$. 
All aforementioned asymptotic expansions are 
uniform in $y\in B$, and can be 
differentiated termwise, arbitrarily often, 
with respect to $t$, $y^+$, and $y^-$. 

If $A(t)=\sum_{j=0}^{N-1}\, t^{-j}\, A_j$, where 
the $A_j$ are $2\times 2$\--matrices with $A_0=1$, 
then the substitution $y=A(t)\, z$ transforms the 
linear differential equation $\op{d}\! y/\op{d}\! t=L(t)\, y$ 
into the linear equation $\op{d}\! z/\op{d}\! t=M(t)\, z$, 
where $L(t)\, A(t)=A'(t)+A(t)\, M(t)$. 
It follows that $M(t)\sim\sum_{k=0}^{\infty}\, t^{-k}\, M_k$, 
where, for each $l\in\Z _{\geq 0}$ 
\[
\sum_{j+k=l}\, L_j\, A_k=(1-l)\, A_{l-1}+\sum_{j+k=l}\, A_j\, M_k.
\]
Here $A_j=0$ when $j<0$ or $j\geq N$. 
It follows that $M_0=L_0$ as in (\ref{L0eq}), 
when the linear mapping $A_l\mapsto L_0\, A_l-A_l\, M_0$ 
is surjective from the space of all $2\times 2$\--matrices 
$A_l$ onto the space of all antidiagonal $2\times 2$\--matrices, 
with kernel equal to the space of all diagonal $2\times 2$\--matrices. 
It follows by induction on $l\leq N-1$ that, given 
the $A_j$ and $M_j$ for $j\leq l-1$, there is a 
unique antidiagonal matrix $A_l$ such that the matrix 
$M_l$ is diagonal. In other words, with the substitution 
$y=A(t)\, z$ the differential equation (\ref{ysystem}) 
is equivalent to the differential equation 
\begin{equation}
\op{d}\! z/\op{d}\! t=g(t^{-1},\, z):=A(t)^{-1}
\, f(t^{-1},\, A(t)\, z)-A(t)^{-1}\, A'(t)\, z,  
\label{zsystem}
\end{equation}
where $g(t^{-1},\, z)$ has an asymptotic 
expansion for $|t|\to\infty$ of the same nature as 
$f(t^{-1},\, y)$,  
\begin{eqnarray}
g(t^{-1},\, z)&=&\sum_{j=0}^{N-1}\, t^{-j}\, M_j\, z
+h(t^{-1},\, z), 
\label{Mdiagonal}\\
h(t^{-1},\, z)&=&\op{O}(|t|^{-N}\,\| z\|) 
+\op{O}(\| z\| ^2),\quad\mbox{\rm and}
\label{hest}\\
\frac{\partial h(t^{-1},\, z)}{\partial z}&=&\op{O}(|t|^{-N}) 
+\op{O}(\| z\|)
\label{hzest} 
\end{eqnarray}
as $t\in V$ and $|t|\to\infty$,  
where $M_j$ is a {\em diagonal} 
$2\times 2$\--matrix for each $0\leq j\leq N-1$. 
We have $M_0=L_0$ is as in (\ref{L0eq}), and 
$M_1$ is equal to the diagonal part of the $2\times 2$\--matrix 
$L_1$ in (\ref{L1eq}). 

Let $\mu _j^+$ and $\mu _j^-$ denote the left upper and 
right lower corner of the diagonal matrix $M_j$, 
where $\mu _0^{\pm}=\,\pm 1$ and $\mu _1^-=\alpha$. 
Write 
\begin{eqnarray}
\Theta ^{\pm}(t)&:=&\op{e}^{\pm t}\, t^{\mu _1^{\pm}}\, \theta ^{\pm}(t),
\quad\mbox{\rm where}\nonumber\\
\theta ^{\pm}(t)&:=&\op{exp}\left(
\sum_{j=2}^{N-1}\,\mu _j^{\pm}\, t^{1-j}/(1-j)\right) .
\label{thetapm}
\end{eqnarray}
Note that (\ref{thetapm}) implies that 
$\Theta ^{\pm}(t)\sim\op{e}^{\pm t}\, t^{\mu _1^{\pm}}$ 
for large $|t|$. 
The solutions $\zeta (t)$ 
of the homogeneous linear 
differential equation $\frac{\scriptop{d}\!\zeta}{\scriptop{d}\! t}
=\sum_{j=0}^{N-1}
\, t^{-j}\, M_j\,\zeta$ are given by 
$\zeta (t)^{\pm}=\Theta ^{\pm}(t)\,\Theta ^{\pm}(\tau )^{-1}\, 
\zeta (\tau )^{\pm}$. 
Substituting (\ref{Mdiagonal}) in (\ref{zsystem}), 
we obtain an inhomogeneous linear 
differential equation for $z$ with 
$h(t^{-1},\, z(t))$ as the inhomogeneous term, when 
Lagrange's method of variation of constants yields the integral 
equations 
\[
z(t)^{\pm}=\frac{\Theta ^{\pm}(t)}{\Theta ^{\pm}(\tau )}\, z(\tau )^{\pm}
+\int_{\tau}^t\,\frac{\Theta ^{\pm}(t)}{\Theta ^{\pm}(s)}
\, h(s^{-1}, z(s))^{\pm}\,\op{d}\! s.
\]
The reasoning leading to (\ref{inteq+2}) and (\ref{inteq-2}) 
this time leads to the conclusion that $z(t)$ is a 
uniformly small solution of (\ref{zsystem}) on $T:=t_0+\R _{\geq 0}$ 
such that $z(t_0)^-=b^-$, if and only if $z(t)$ is a uniformly 
small continuous function which satisfies the integral 
equations 
\begin{equation}
z(t)^+=H(z)(t)^+:=
\, -\int_t^{+\infty}\,\frac{\Theta ^+(t)}{\Theta ^+(s)}
\, h(s^{-1},\, z(s))^+\,\op{d}\! s 
\label{inteq+2z}
\end{equation}
and
\begin{equation}
z(t)^-=H(z)(t)^-:=\frac{\Theta ^-(t)}{\Theta ^-(t_0)}\, b^-
+\int_{t_0}^{t}\,\frac{\Theta ^-(t)}{\Theta ^-(s)} 
\, h(s^{-1},\, z(s))^-\,\op{d}\! s. 
\label{inteq-2z}
\end{equation}
Assume that $N>1$, $\gamma >|b^-/\Theta ^-(t_0)|$. 
In the next paragraphs we will prove 
that, if $\sup_{t\in T}\, |t^{-1}|$ and 
$\sup_{t\in T}\, |\op{e}^{-t}\, t^{\alpha}|$ are sufficiently  
small, then $H$ is a 
contraction in the set ${\mathcal Z}$  of all continuous 
functions $z:T\to\C ^2$ such that 
$\| z\| :=\sup_{t\in T}\, |\op{e}^t\, t^{-\alpha}|\,\| z(t\|\leq\gamma$, 
with respect to the metric $(z_1,\, z_2)\mapsto\| z_1-z_2\|$. 

If in the integrand in (\ref{inteq+2z}) 
we substitute $s=t+\tau$, $\tau\in\R _{\geq 0}$, then 
(\ref{hest}) implies that $|\op{e}^t\, t^{-\alpha}|$ times 
the absolute value of the integrand is estimated from above by a 
uniform constant times 
\[
\op{e}^{-2\,\tau}
\, | (\frac{t+\tau}{t})^{\alpha -\mu_1^+}|\, 
\left( |t+\tau |^{-N}
+|\op{e}^{-(t+\tau )}\, (t+\tau )^{\alpha}|\,\| z\|\right)
\, \| z\| .
\]
In view of (\ref{hzest}), a similar estimate with the last factor 
$\| z\|$ replaced by $\| z_1-z_2\|$ holds for 
$|\op{e}^t\, t^{-\alpha}|\, |H(z_1)(t)^+-H(z_2)(t)^+|$. 
The integral of $\tau\mapsto\op{e}^{-2\,\tau}
\, | (\frac{t+\tau}{t})^{\alpha -\mu_1^+}|$ over $\R _{\geq 0}$ 
is uniformly bounded if the distance from $T$ to the origin 
is bounded away from zero. As 
$|s^{-1}|$ and $|\op{e}^{-s}\, s^{\alpha}|$ are uniformly 
small for all $s\in T$, there exists a $0<\beta <1$ such that 
$|\op{e}^t\, t^{-\alpha}|\, |H(z)(t)^+|\leq\gamma$ 
and $|\op{e}^t\, t^{-\alpha}|\, |H(z_1)(t)^+-H(z_2)(t)^+|\leq 
\beta\,\| z_1-z_2\|$ for all $z,\, z_1,\, z_2\in {\mathcal Z}$ 
and all $t\in T$.  

In (\ref{inteq-2z}), $|\op{e}^t\, t^{-\alpha}|$ times 
the absolute value of the integrand is estimated from above by a 
uniform constant times 
$(|s|^{-N}+|\op{e}^{-s}\, s^{\alpha}|)\,\| z\|$. 
A similar estimate, with the last factor $\| z\|$ 
replaced by $\| z_1-z_2\|$, holds for the integrand 
in the integral formula for 
$|\op{e}^t\, t^{-\alpha}|\, |H(z_1)(t)^--H(z_2)(t)^-|$.  
Let $N>1$ and $\gamma >|b^-/\Theta ^-(t_0)|$. As  
$|s^{-1}|$ and $|\op{e}^{-s}\, s^{\alpha}|$ are uniformly 
small for all $s\in T$, there exists a $0<\beta <1$ such that 
$|\op{e}^t\, t^{-\alpha}|\, |H(z)(t)^-|\leq\gamma$ 
and $|\op{e}^t\, t^{-\alpha}|\, |H(z_1)(t)^--H(z_2)(t)^-|\leq 
\beta\,\| z_1-z_2\|$ for all $z,\, z_1,\, z_2\in {\mathcal Z}$ 
and all $t\in T$. This completes the proof 
that $H({\mathcal Z})\subset {\mathcal Z}$ and 
$H:{\mathcal Z}\to {\mathcal Z}$ is a contraction 
with respect to the metric 
$(z_1,\, z_2)\mapsto 
\sup_{t\in T}\, |\op{e}^t\, t^{-\alpha}|\, \| z_1(t)-z_2(t)\|$. 
Because ${\mathcal Z}$ is complete with respect to this metric, 
it follows that there is a unique solution 
$z=z_{t_0,\, b^-}:T\to\C ^2$ of (\ref{zsystem}) such that 
$z(t_0)^-=b^-$ and $\| z(t)\|\leq\gamma\, |\op{e}^{-t}\, t^{\alpha}|$ 
for every $t\in T$. Because the latter estimate implies that 
$\| z(t)\|$ is uniformly small, the solution 
$p(t)=p_{\uparrow}(t)+A(t)\, z(t)$ of (\ref{cssystem}) 
is uniformly small on $T$, and therefore equal to 
the solution $p_{t_0,\, a^-}(t)$ in ii), with 
$a^-=(p_{\uparrow}(t_0)^-+A(t_0)\, z_{t_0,\, b^-}(t_0))^-$ 
and $b^-=A(t_0)^{-1}\, (p_{t_0,\, a^-}(t_0)-p_{\uparrow}(t_0))^-$. 
The assumptions remain valid upon 
small perturbations of $t_0$ and $b^-$, and an application 
of the implicit function theorem yields that the solution 
$z=z_{t_0,\, b^-}$ depends in a complex analytic way on 
$(t_0,\, b^-)$, when $b^-\mapsto a^-$ is a local complex 
analytic diffeomorphism depending in a complex analytic way 
on $t_0$. 

The estimates in the previous paragraph imply that 
the integral $\int_{t_0}^{\infty}\,\Theta ^-(s)^{-1}
\, h(s^{-1},\, z(s))\,\op{d}\! s$ is absolutely 
convergent, when (\ref{inteq-2z}) implies that 
$\Theta ^-(t)^{-1}\, z(t)^-$ converges to 
\begin{equation}
C:=\Theta ^-(t_0)\, b^-+\int_{t_0}^{\infty}
\,\Theta ^-(s)^{-1}\, h(s^{-1},\, z(s))\,\op{d}\! s
\label{Ceq}
\end{equation}
as  $t=t_0+\tau$, $\tau\to +\infty$. Because 
$\Theta ^-(t)\sim\op{e}^{-t}\, t^{\alpha}$, 
the function $\op{e}^t\, t^{-\alpha}\, z(t)^-$ 
has the same limit. 
The previous estimates for $z(t)^+=H(z)(t)^+$ yield that 
$\op{e}^t\, t^{-\alpha}\, z(t)^+$ converges to zero, 
and it follows that $\op{e}^t\, t^{-\alpha}\, z_{t_0,\, b^-}(t)$ 
converges to the vector $(0,\, C)\in\C ^2$ which depends in a complex 
analytic way on $(t_0,\, b^-)$ and on $(t_0,\, a^-)$. 
Because $A(t)$ converges to the identity matrix, 
also $\op{e}^t\, t^{-\alpha}\, (p_{t_0,\, a^-}(t)-p_{\uparrow}(t))
=A(t)\, (\op{e}^t\, t^{-\alpha}\, z_{t_0,\, b^-}(t))$ 
converges to $(0,\, C_{t_0,\, a^-})$ as $t=t_0+\tau$, $\tau\to +\infty$. 
 
As $b^-=z(t_0)^-$, and the previous remains true if 
we replace $t_0$ by any $t\in T=t_0+\R _{\geq 0}$, 
the equation (\ref{inteq-2z}) implies the integral equation 
\begin{equation}
z(t)^-={\mathcal H}(z)(t)^-:=\Theta ^-(t)\,\left( C-\int_t^{\infty}
\,\Theta ^-(s)^{-1}\, h(s^{-1},\, z(s))^-\,\op{d}\! s\right) .
\label{inteq-2zl}
\end{equation}
If ${\mathcal H}(z)(t)^+:=H(z)(t)^+$, then 
then for {\em every} $C\in\C$ the integral operator 
${\mathcal H}={\mathcal H}_{C}$ is a contraction in ${\mathcal Z}$, 
if $\gamma >|C|$ and the numbers $\sup_{t\in T}\, |t^{-1}|$ and 
$\sup_{t\in T}\, |\op{e}^{-t}\, t^{\alpha}|$ are sufficiently  
small. The unique fixed point $z=z_{C}$ of ${\mathcal H}_{C}$ 
is equal to the unique solution $z(t)$ of (\ref{zsystem}) 
such that $\sup_{t\in T}\, |\op{e}^t\, t^{-\alpha}|\, \| z(t)\|
\leq\gamma$, and $\op{e}^t\, t^{\alpha}\, z(t)\to (0,\, l)$ 
as $t\in T$, $|t|\to\infty$.  

Let $R=R_{\eta ,\, r}$ be as in (\ref{Rdef}). 
If $\gamma >|C|$ and $\eta ,\, r^{-1}$ are sufficiently small, 
then the integral operator ${\mathcal H}$ in the previous paragraph 
defines a contraction in the complete space of 
all continuous functions $z:R\cap V\to\C ^2$, complex analytic 
in the interior of $R$,  
such that $|\op{e}^t\, t^{-\alpha}|\, 
\| z(t)\|\leq\gamma$ for every $t\in R\cap V$, 
with respect to the metric $(z_1,\, z_2)\mapsto 
\sup_{t\in R\cap V}\, |\op{e}^t\, t^{-\alpha}|\, 
\| z_1(t)-z_2(t)\|$. For every $t_0\in R\cap V$ the restriction 
of $z$ to $T:=t_0+\R _{\geq 0}$ is equal to the unique 
fixed point $z_{C}$ of the integral operator 
${\mathcal H}:{\mathcal Z}\to {\mathcal Z}$ 
in the previous paragraph, and therefore 
$z_{C}:T\to\C ^2$ extends to 
a solution of (\ref{zsystem}) on the much larger domain 
$R\cap V$, which is denoted by the same letter and 
satisfies $\sup_{t\in R\cap V}|\op{e}^t\, t^{-\alpha}|\, 
\| z(t)\|\leq\gamma$. 

Substituting an asymptotic expansion 
$z(t)\sim \op{e}^{-t}\, t^{\alpha}\,\sum_{j=0}^{k-1}
\, t^{-j}\, e_j$ for $t\in\Sigma$, $|t|\to\infty$ in the right hand 
side of $z(t)={\mathcal H}(z)(t)$ yields a similar expansion 
with $k$ replaced by $k+1$, which procedure stops at $k=N-1$. 
In the induction step it is used that in $\Sigma$ we have 
$\| z(t)\|=\op{O}(|\op{e}^t\, t^{\alpha}|)
=\op{O}(|t|^{-M})$ for every $M>0$. This leads to 
an asymptotic expansion 
$y(t)\sim C\,\op{e}^{-t}\, t^{\alpha}\,\sum_{j=0}^{N-1}
\, t^{-j}\, d_j$, where the $d_j$ satisfy 
the equations $d_0=(0,\, 1)$ and 
$-d_j+(\alpha -j+1)\, d_{j-1}=\sum_{i=0}^j\, 
L_i\, d_{j-i}$ for $0\leq j\leq N-1$. 
For $j=1$ we have $-d_1+\alpha\, d_0=L_0\, d_1+L_1\, d_0$, 
which implies that $(L_1-\alpha )\, (0,\, 1)$ is in the image 
$\C\times\{ 0\}$ of $L_0+1$, and we recover that $\alpha$ 
has to be equal to the right lower corner of $L_1$. 
If $j\geq 1$, then the equation 
$-d_j+(\alpha -j+1)\, d_{j-1}=L_0\, d_j+\sum_{i=1}^j\, L_i\, d_{j-i}$ 
determines $d_j$ only modulo the kernel $\{ 0\}\times\C$ 
of $L_0+1$, but the equation with $j$ replaced by $j+1$ 
implies that $(L_1-\alpha +j)\, d_j+\sum_{i=2}^j\, L_i\, d_{j-i}$ 
belongs to the image $\C \times\{ 0\}$ of $L_0+1$. 
As the right lower corner of $L_1-\alpha +j$ is equal to $j\neq 0$, 
the space $(L_1-\alpha +j)(\{ 0\}\times\C )$ is transversal 
to $\C\times\{ 0\}$, and it follows that $d_j$ is uniquely 
determined by the $j$\--th equation and the $(j+1)$-st equation. 
Because we can take $N$ arbitrarily large, and the $d_j$, 
$0\leq j\leq N-1$ do not depend on $N$, this implies 
(\ref{expasym}), where (\ref{expasym'}) follows from a 
substitution of (\ref{expasym'}) in 
$y'(t)=f(t^{-1},\, y(t))$, $y(t)=p(t)-p_{\uparrow}(t)$. 

This completes the proof of the lemma in the upper half plane, where 
(\ref{nonsmall}) follows from the aforementioned estimates 
for the integrals $\int _t^{\infty}
\, \Theta ^{\pm}(s)^{-1}\, h(s^{-1},\, z(s))^{\pm}\,\op{d}\! s$. 
In turn (\ref{nonsmall}), in combination with 
$p_{\uparrow}(t)=\op{O}(t^{-1})$, implies that 
$p_C(t)$ is bounded away from zero on the part 
of the boundary of $V_{\eta ,\, r}$ in the upper half plane 
and sufficiently far away from the origin. 
The statements regarding the behavior in the lower half plane 
follow by replacing $p_{\uparrow}(t)$ by $p_{\downarrow}(t)$. 
Note that in the lower half plane $p_{\uparrow}(t)$ 
is equal to one of the solutions $p_{t_0,\, a^-}$, and 
therefore $p_{\uparrow}(t)$ extends from $V$ to a domain 
of the form $V\cup R_{\eta ,\, r}$. 
\end{proof}
\begin{remark}
Because $p_{\downarrow}(t)$ is a solution  
of (\ref{cssystem}) as in Lemma \ref{stablelem}, 
it follows from Lemma \ref{alphalem} that 
there is a unique constant $S\in\C$ such that 
\begin{equation}
p_{\downarrow}(t)-p_{\uparrow}(t)
\sim\op{e}^{-t}\, t^{\alpha}\, (0,\, S)
\label{Sdef}
\end{equation} 
as $\op{Re}t/\log |t|\to\infty$, $|t|\to\infty$, 
where $S=0$ implies that $p_{\downarrow}(t)=p_{\uparrow}(t)$ 
for all $t\in R_{\eta ,\, r}$ with $\eta$ and $1/r$ sufficiently small. 
Because this is a nonlinear analogue of the phenomena 
described by Stokes \cite{stokes}, we follow 
Costin \cite[Th. 1]{costinduke} in calling 
$S$ the {\em Stokes constant} 
pertaining to the comparison in the right  
half plane of the solutions 
$p_{\downarrow}(t)$ and $p_{\uparrow}$ of (\ref{cssystem}). 
Generically the Stokes constant is nonzero. 
This happens in particular for the Boutroux\--Painlev\'e system, 
see Lemma \ref{Slem} below. 

Furthermore, for every 
solution $p(t)=p_{t_0,\, a^-}(t)$ 
of (\ref{cssystem}) in Lemma \ref{stablelem} 
there are unique $C_{\uparrow},\, C_{\downarrow}\in\C$ 
such that 
\begin{equation}
p(t)-p_{\uparrow}(t)\sim\op{e}^{-t}\, t^{\alpha}\, 
(0,\, C_{\uparrow}) \quad\mbox{\rm and}\quad 
p(t)-p_{\downarrow}(t)\sim \op{e}^{-t}\, t^{\alpha}
\, (0,\, C_{\downarrow})
\label{Cupdown}
\end{equation} 
as $\op{Re}t/\log |t|\to\infty$, $|t|\to\infty$, 
where conversely the solution $p(t)$ of (\ref{cssystem}) is uniquely  
by any of the two asymptotic identities in (\ref{Cupdown}). 
Because $(p(t)-p_{\uparrow}(t))-(p(t)-p_{\downarrow}(t))
=p_{\downarrow}(t)-p_{\uparrow}(t)$, it follows that 
\begin{equation}
C_{\uparrow}-C_{\downarrow}=S.
\label{CupdownS}
\end{equation} 
Conversely, every complex number occurs as $C_{\uparrow}$ 
in (\ref{Cupdown}) for a unique solution $p(t)$ 
of (\ref{cssystem}) in Lemma \ref{stablelem}. 
The following statements are equivalent:
\begin{itemize}
\item[i)] 
$C_{\downarrow}=C_{\uparrow}$ for some 
solution $p(t)$ of (\ref{cssystem}) in Lemma \ref{stablelem}. 
\item[ii)] $S=0$. 
\item[iii)] $C_{\downarrow}=C_{\uparrow}$ for every  
solution $p(t)$ of (\ref{cssystem}) in Lemma \ref{stablelem}. 
\item[iv)] 
$p_{\downarrow}(t)=p_{\uparrow}(t)$ for all $t$ in a 
domain of the form $R_{\eta,\, r}$. 
\end{itemize} 

Recall the definition of $\mu _1^+$ 
as the upper left corner of the matrix 
$L_1$ in (\ref{L1eq}). In analogy with (\ref{Sdef}), 
the substitution $t\mapsto -t$ leads to the existence of 
a unique constant $S_-\in\C$, the Stokes constant pertaining to the 
comparison of the solutions $p_{\downarrow}(t)$ and 
$p_{\uparrow}(t)$ in the left half plane, such that 
\begin{equation}
p_{\downarrow}(t)-p_{\uparrow}(t)
\sim\op{e}^{t}\, t^{\mu _1^+}\, (S_-,\, 0)
\label{Sleftdef}
\end{equation} 
as $\op{Re}t/\log |t|\to -\infty$, $|t|\to\infty$.  
Here $S_-=0$ implies that $p_{\downarrow}(t)=p_{\uparrow}(t)$ 
for all $t\in -R_{\eta ,\, r}$ with $\eta$ and $1/r$ sufficiently small. 
We have $S_-=S=0$ if and only if 
$p_{\uparrow}(t)$ and $p_{\downarrow}(t)$ have a common extension 
to a single\--valued solution $p(t)$ of (\ref{cssystem}) 
on a neighborhood of $t=\infty$ such that $p(t)\to 0$ as $t\to\infty$. 
\label{stokesconstantrem}
\end{remark}

Write $\tau (t)=\op{e}^{-t}\, t^{\alpha}$, and let 
\begin{equation}
p_{\,\scriptop{formal}}(t)=\sum_{h,\, i\in\Z_{\geq 0}}
\, \tau (t)^h\, t^{-i}\, p_{h,\, i}
\label{formal}
\end{equation} 
be a formal power series in $\tau (t)$ and $t^{-1}$, 
with coefficients $p_{h,\, i}\in\C ^2$. 
The termwise derivative of the right hand side of (\ref{formal}), 
where $\op{d}(\tau (t)^h\, t^{-i})/\op{d}\! t
=\, -h\, \tau (t)^h\, t^{-i}+(h\,\alpha -i)\, \tau (t)^h\, t^{-i-1}$, 
is a formal power series in $\tau (t)$ and $t^{-1}$, 
which by definition is the derivative of $p_{\,\scriptop{formal}}(t)$ 
with respect to $t$. Substitution of $p_{\,\scriptop{formal}}(t)$ 
in the expansion (\ref{vexpansion}) yields 
a formal power series in $\tau (t)$ and $t^{-1}$, 
which by definition is $v(t^{-1},\,p_{\,\scriptop{formal}}(t))$. 
The formal solutions $p(t)=p_{\,\scriptop{formal}}(t)$ of the 
differential equation (\ref{cssystem}) are obtained by 
equating the coefficients of $\tau (t)^h\, t^{-i}$. 

Let $T$ be an unbounded subdomain in the complex upper half plane 
of the domain $R_{\eta ,\,r}$ in Lemma \ref{alphalem},  
and assume that there exist strictly positive constants 
$C_1,\, C_2,\,\epsilon _1,\,\epsilon _2$ such that 
\[
C_1\, |t|^{-\epsilon _1}\leq |\tau (t)|\leq C_2\, |t|^{-\epsilon _2}
\] 
for every $t\in T$. In such a domain any formal 
series (\ref{formal}) is an asymptotic series. 
The proof of Lemma \ref{alphalem} yields that 
the solutions $p_C(t)$ in Lemma \ref{alphalem} 
have a formal series $p_{\,\scriptop{formal},\, C}(t)$ 
as their asymptotic expansion in $T$, in the sense that 
for every $N$ there exists an $m$ such that 
$p_C(t)-\sum_{h<m,\, i<m}\,\tau (t)^h\, t^{-i}\, p_{h,\, i}
=\op{O}(t^{-N})$ as $t\in T$, $|t|\to\infty$. 
The formal series $p_{\,\scriptop{formal},\, C}(t)$ 
is a formal series solution of the differential 
equation (\ref{cssystem}) such that $p_{0,\, 0}=0$ and  
$p_{0,\, 1}=(0,\, C)$. As will be verified in the proof of Lemma 
\ref{Fjlem} below, a formal series solution 
(\ref{formal}) of (\ref{cssystem}) is uniquely determined by 
the initial consitions $p_{0,\, 0}=0$, $p_{0,\, 1}=(0,\, C)$, 
and $p_{h,\, i}=C^h\, c_{h,\, i}$ for uniquely determined 
universal coefficients $c_{h,\, i}\in\C ^2$. 
This is the formal solution in O. and R. Costin 
\cite[(4), (5)]{costin2} for $n=2$ and 
${\bf C}=(0,\, C)$, where the latter implies that only the terms 
with ${\bf k}=(0, h)$ appear in loc. cit. 
The proof of Lemma \ref{Fjlem} also yields an independent 
verification of the existence of the formal series solution of 
(\ref{cssystem}) such that $p_{0,\, 0}=0$ and $p_{0,\, 1}=(0,\, C)$.

In the subdomains $\Sigma$ in Lemma \ref{alphalem}, the function $\tau (t)$ 
is of smaller order than $t^{-i}$ for every $i\in\Z _{\geq 0}$, 
when (\ref{formal}) only is an asymptotic series in the above sense if 
all terms with $h>0$ are deleted, and  
$p_{\,\scriptop{formal},\, C}(t)$ reduces to the right hand side of 
(\ref{powerasym}), which is independent of $C$. 
For $p_C(t)-p_{\uparrow}(t)$ we have the asymptotic expansion 
(\ref{expasym}) in $\Sigma$, where the right hand side 
is equal to (\ref{formal}) with only the terms with $h=1$ retained. 

If $U$ is an unbounded subdomain of $R_{\eta ,\, r}$ 
in the upper half plane on which $\tau (t)\to 0$ 
and $t^{-i}=\op{o}(\tau (t))$ as $t\in U$, $|t|\to\infty$, 
then $p_{\,\scriptop{formal},\, C}(t)$ is only an asymptotic series 
for $t\in U$, $|t|\to\infty$ if all terms with $i>0$ 
are deleted, when $p_C(t)$ is asymptotically equal to 
this asymptotic series for $t\in U$, $|t|\to\infty$. 
For $t$ on the boundary of $R_{\eta ,\, r}$ and $|t|\geq r$, 
the absolute value of $\tau (t)$ is equal to the constant $\eta$, 
$p_{\,\scriptop{formal},\, C}(t)$ is not an asymptotic series 
in the above sense, and we only have the estimate (\ref{nonsmall}). 
As $T$ is a transitional region between the subdomains 
$S$ on the one hand and the subdomains $U$ and the 
boundary of $R_{\eta ,\, r}$ on the other, the formal series 
$p_{\,\scriptop{formal},\, C}(t)$ and the asymptotic expansion 
$p_C(t)\sim p_{\,\scriptop{formal},\, C}(t)$ as $t\in T$, $|t|\to\infty$ 
could be called the {\em transitional series} and 
the {\em transitional expansion}, respectively. 

Lemma \ref{Fjlem} below, which follows from 
O. and R. Costin \cite[Th. 2(i) and Sec. 6.9]{costin2}, 
yields that for each $i\in\Z _{\geq 0}$ the series 
(\ref{Fidef}) converges for small $|\tau |$. 
This allows to formulate the asymptotic expansion 
(\ref{bdyasym}) for the solution $p_C(t)$ of 
(\ref{cssystem}). The domain of $t$'s where the 
asymptotic expansion (\ref{bdyasym}) holds  
extends well beyond the part $R'_{\eta ,\, r}$ of the boundary of 
$R_{\eta ,\, r}$ where $\op{Im}t\geq 0$, 
$\tau (t)=\eta$, and $|t|\geq r$. 
Along $R'_{\eta ,\, r}$, 
Lemma \ref{alphalem} yielded that $p_C(t)$ remains 
bounded away from zero, but did not provide an asymptotic 
expansion for $|t|\to\infty$. 

\begin{lemma}
Let $C\in\C$. Then there is a unique 
formal solution (\ref{formal}) of 
(\ref{cssystem}) such that 
$p_{0,\, 0}=0$ and $p_{1,\, 0}=(0,\, C)$. 
For each $i\in\Z _{\geq 0}$ the series 
\begin{equation}
F_i(\tau )=\sum_{h=0}^{\infty}\,\tau ^h\, p_{h,\, i}
\label{Fidef}
\end{equation}
converges for $\tau$ in a neighborhood of $0$ in $\C$, 
where the complex analytic functions $F_i$ 
satisfy 
\begin{equation}
-\tau\,\frac{\op{d}\! F_0(\tau )}{\op{d}\!\tau}=
v_0(F_0(\tau )):=v(0,\, F_0(\tau ))
\quad\mbox{\rm and}
\label{F0eq}
\end{equation}
\begin{equation}
-\tau\,\frac{\op{d}\! F_i(\tau )}{\op{d}\!\tau} 
=P_i(F_0(\tau ),\,\ldots ,\, F_{i}(\tau ))
-\alpha\,\tau\,\frac{\op{d}\! F_{i-1}(\tau )}{\op{d}\!\tau}
+(i-1)\, F_{i-1}(\tau ) 
\label{Fieq}
\end{equation}
for $i\in\Z _{>0}$.  
Here $P_i(F_0,\,\ldots F_{i})$ denotes the coefficient 
of $t^{-i}$ in the expansion of 
\[
v(t^{-1},\,\sum_{j\geq 0}\, t^{-j}\, F_j)
\] 
in nonnegative integral 
powers of $t^{-1}$; a finite sum of weighted homogeneous polynomials 
of degree $\leq i$ in the $F_j$ with $0\leq j\leq i$, where 
each $F_j$ has the weight $j$. In particular 
\begin{equation}
P_i(F_0,\,\ldots ,\, F_i)=\frac{\partial v(0,\, F_0)}{\partial F_0}\, F_i
+Q_i(F_0,\,\ldots F_{i-1}),
\label{PiFi}
\end{equation}
where $Q_i(F_0,\,\ldots ,\, F_{i-1})$ 
is a similar polynomial in the $F_j$ with $0\leq j\leq i-1$. 
We have 
\begin{equation}
F_i(0)=p_{0,\, i}=c_i, 
\label{Fi0}
\end{equation}
where the $c_i$ are the coefficients in (\ref{powerasym}). 

Conversely, the system (\ref{F0eq}), (\ref{Fieq}) 
has a unique solution $F_i(\tau )$ which is complex analytic 
on a neighborhood of $\tau =0$ in $\C$ such that 
$F_0(0)=0$ and $F_0'(0)=(0,\, C)$. 
If $j\in\Z _{>0}$ then, given the functions 
$F_k(\tau )$ for $0\leq k\leq j-1$, the function $F_j(\tau )$ 
is uniquely determined by the conditions that it is 
a complex analytic solution on a neighborhood of 
$\tau =0$ of the equation (\ref{Fieq}) for $i=j$, and 
that the equation (\ref{Fieq}) for $i=j+1$ admits a 
complex analytic solution $F_{j+1}(\tau )$ on a neighborhood 
of $\tau =0$.  
If (\ref{Fidef}) denotes the 
power series expansion of $F_i(\tau )$, then 
(\ref{formal}) is the formal solution 
of (\ref{cssystem}) in the previous paragraph. 

If $F_{i,\, C}(\tau )$ denotes the solution of 
(\ref{F0eq}), (\ref{Fieq}) such that $F_{0,\, C}(0)=0$ and  
${F_{0,\, C}}'(0)=(0,\, C)$, then 
$F_{i,\, C}(\tau )=F_{i,\, 1}(C\,\tau )$. 
Let ${\mathcal T}_C$ be the maximal domain of definition of 
the, possibly multi\--valued, solution $F_0=F_{0,\, C}$ of (\ref{F0eq}). 
Then ${\mathcal T}_0=\C$, ${\mathcal T}_C=C^{-1}\, {\mathcal T}_1$ if $C\neq 0$, 
and for each $i\in\Z _{>0}$ the function $F_{i,\, C}$ extends to a, 
possibly multi\--valued, solution of 
(\ref{Fieq}) on ${\mathcal T}_C$.  
\label{Fjlem}
\end{lemma}
\begin{proof}
With the substitutions $\tau =\op{e}^{-s}$, 
$F_0(\tau )=p_0(s)$, the equation 
(\ref{F0eq}) is equivalent to the autonomous limit system 
\begin{equation}
\frac{\op{d}\! p_0(s)}{\op{d}\! s}=v_0(p_0(s)) 
\label{G0eq}
\end{equation}
of the system (\ref{cssystem}). 
Because $v_0(0)=v(0,\, 0)=0$ and ${v_0}'(0)=L_0$, 
an application of Lemma \ref{alphalem} with $v$ replaced by 
$v_0$, when $c_1=0$, $L_1=0$, and $\alpha =0$, leads to 
$p_{0,\,\uparrow}(s)=p_{0,\,\downarrow}(s)=0$ and 
a unique solution $p_0(s)$ of (\ref{G0eq}) on 
$\op{Re}s\geq\sigma >>0$ such that 
$p_0(s)\sim\op{e}^{-s}\, (0,\, C)$ as $\op{Re}s\geq\sigma >>0$, 
$|s|\to\infty$. The corresponding function 
$F_0(\tau )=p_0(s )$ is complex analytic on the punctured disc   
in the complex plane determined by the inequalities 
$0<|\tau |<\epsilon :=\op{e}^{-\sigma}$, and satisfies 
$F_0(\tau )\sim \tau\, (0,\, C)$ as $\tau\to 0$. It therefore 
follows from the theorem on removable singularities 
that $F_0(\tau )$ extends to a unique complex analytic 
function on the disc $|\tau |<\epsilon$, denoted again by 
$F_0(\tau )$, where $F_0(\tau )$ is a solution of 
(\ref{F0eq}) on the disc $|\tau |<\epsilon$ such that 
$F_0(0)=0$ and ${F_0}'(0)=(0,\, C)$. As the formal power series 
solution $\sum_{h\geq 0}\,\tau ^h\, p_{h,\, 0}$ of 
(\ref{F0eq}) with $p_{0,\, 0}=0$, $p_{1,\, 0}=(0,\, C)$ is 
unique, the above complex analytic solution $F_0(\tau )$ 
is unique as well. 

Let $i\in\Z _{>0}$. For given complex analytic functions 
$F_j$, $0\leq j\leq i-1$, in a neighborhood of the origin, 
(\ref{Fieq}) is a linear inhomogeneous differential equations  
of the form 
\begin{equation}
-\tau\,\frac{\op{d}\! F_i}{\op{d}\!\tau}=
\Lambda (\tau )\, F_i+G_i(\tau ),
\label{FLambdaG}
\end{equation}
where $\Lambda (\tau ):=\partial v(0,\, p)/\partial p|_{p=F_0(\tau )}$ 
and $G_i(\tau )$ are complex analytic functions of $\tau$ 
in a neighborhood of 
$\tau =0$. With the substitutions 
$\Lambda (\tau )=\sum_{k\geq 0}\,\tau ^k\,\Lambda _k$, 
where $\Lambda _0=L_0$ as in (\ref{L0eq}),  
and $G_i(\tau )=\sum_{k\geq 0}\,\tau ^k\, G_{i,\, k}$, 
the formal power series $F_i=\sum_{k\geq 0}\, F_{i,\, k}$ 
is a solution of the differential equation if and only if 
\[
-k\, F_{i,\, k}=L_0\, F_{i,\, k}+\sum_{l=1}^{k}
\,\Lambda _l\, F_{i,\, k-l}+G_{i,\, k}
\]
for every $k\in\Z _{\geq 0}$. These equations determine 
the coefficients $F_{i,\, k}$ in terms of the $F_{i,\, j}$ 
with $j<k$ and $G_{i,\, k}$, with the exception of the equation 
for $k=1$, where the $+$ part yields 
$F_{i,\, 1}^+=\, -(1/2)\, (\Lambda _1\, F_{i,\, 0}+G_{i,\, 1})^+$, 
but the resonance in the $-$ part yields no equation for 
$F_{i,\, 1}^-$, but instead the solvability condition 
$(\Lambda _1\, F_{i,\, 0}+G_{i,\, 1})^-=0$. As 
$F_{i,\, 0}=\, -{L_0}^{-1}\, G_{i,\, 0}$, this solvability 
condition is equivalent to the equation 
$G_{i,\, 1}^-=(\Lambda _1\, {L_0}^{-1}\, G_{i,\, 0})^-$ 
for the inhomogeneous term $G_i(\tau )$ in the differential 
equation. Because $G_{i,\, 0}=Q_i(F_{0,\, 0},\,\ldots ,\, F_{i-1,\, 0})
+(i-1)\, F_{i-1,\, 0}$ is determined, the solvability 
condition determines $G_{i,\, 1}^-$.  

For $i=1$ we have $P_1(F_0,\, F_1)=\partial_1v(0,\, F_0)
+\partial_2v(0,\, F_0)\, F_1$, 
hence $G_1(\tau )=\partial_1v(0,\, F_0(\tau ))
-\alpha\,\tau\, F_0'(\tau )=\partial_1v(0,\, F_0(\tau ))
+\alpha\, v(0,\, F_0(\tau ))$. Therefore 
$G_{1,\, 0}=\partial _1v(0,\, 0)$ and 
$G_{1,\, 1}$ is equal to $C$ times 
\[
\left.\frac{\partial ^2v(u,\, p)}{\partial u\partial p^-}
\right|_{u=0,\, p=0}
+\alpha\,\left.\frac{\partial v(0,\, p)}{\partial p^-}\right|_{p=0},
\]
because $F_0'(0)$ is equal to $C$ times the second basis vector. 
As the minus part of the second term 
is equal to $\, -\alpha$, it follows from the 
definition of $\alpha$ in Lemma \ref{alphalem} 
that $G_{1,\, 1}^-$ is equal to $C$ times the 
lower left corner of 
${\partial _2}^2v(0,\, 0)\, {L_0}^{-1}\,\partial _1v(0,\, 0)$.  
On the other hand $\Lambda _1$ is equal to $C$ times 
${\partial _2}^2v(0,\, 0)$ applied to the second basis vector, 
when the symmetry of the second order partial 
derivatives yields the solvability condition 
$G_{1,\, 1}^-=(\Lambda _1\, {L_0}^{-1}\, G_{1,\, 0})^-$. 

For $i\geq 2$ the above computations yield that 
$G_{i,\, 1}^-$ depends in an inhomogeneous linear way 
on $F_{i-1,\, 1}^-$ with coefficient equal to $i-1\neq 0$. 
Therefore the solvability condition is satisfied by a unique 
choice of $F_{i-1,\, 1}^-$. It follows that 
the system (\ref{F0eq}), (\ref{Fieq}) has a unique formal 
solution such that $F_0(0)=0$ and $F_0'(0)=(0,\, C)$, 
where for the unique determination of $F_j(\tau )$, $j>0$,  
one needs (\ref{Fieq}) for $i=j$ and $i=j+1$. 

If in a neighborhood of $\tau =0$ the functions 
$F_j(\tau )$ are complex analytic solutions of 
(\ref{F0eq}), (\ref{Fieq}) for $0\leq j\leq i-1$, 
then (\ref{FLambdaG}) is an inhomogeneous linear 
differential equation with complex analytic coefficients 
$\Lambda (\tau )$ and $G_i(\tau )$, with 
$\tau =0$ as a regular singular point. Therefore  
every formal powers series solution of (\ref{FLambdaG}) 
is convergent. For inhomogeneous linear differential 
equations this theorem, which sounds classical, 
seems to be due to G\'erard and Levelt \cite[Lemme 4.2]{gl}.  
For nonlinear higher order scalar ordinary differential equations, 
it has been obtained by Malgrange \cite[Remarque 4.1]{malgrange89}, 
where the proof also works for nonlinear systems near  
a regular singular point. 

If $g(\tau )=f(C\,\tau )$, 
then $\tau\, g'(\tau )=(C\,\tau )\, f'(C\,\tau )$. 
Therefore the function $\Phi _0:\tau\mapsto F_{0,\, 1}(C\,\tau )$ 
satisfies (\ref{F0eq}) with $\Phi _0(0)=0$, ${\Phi _0}'(0)=C$, 
hence $F_{0,\, C}(\tau )=\Phi _0(\tau )=F_{0,\, 1}(C\,\tau )$. 
Furthermore, if $F_j(\tau )=F_{j,\, 1}(C\,\tau )$ 
for $0\leq j\leq i-1$, then 
$\Phi _i:\tau\mapsto F_{i,\, 1}(C\,\tau )$ satisfies 
(\ref{Fieq}), hence $F_i(\tau )=\Phi _i(\tau )
=F_{i,\, 1}(C\,\tau )$. This proves 
$F_{i,\, C}(\tau )=F_{i,\, 1}(C\,\tau )$ 
by induction on $i$. If the coefficients $p_{h,\, i}$ 
for $C=1$ are denoted by $c_{h,\, i}$, then 
the equations $F_{i,\, C}(\tau )=F_{i,\, 1}(C\,\tau )$ 
are equivalent to the equations $p_{h,\, i}=C^h\, c_{h,\, i}$ 
mentioned in the text preceding Lemma \ref{Fjlem}. 

The last statement in Lemma \ref{Fjlem} follows,   
by induction on $i$, from the description of 
$F_i$ as a solution of an inhomogeneous linear 
differential equation (\ref{FLambdaG}), of which the 
coefficients $\Lambda (\tau )$ and $G_i(\tau )$ 
are complex analytic functions of $\tau \in {\mathcal T}_C$. 
\end{proof}
\begin{remark}
It would have been more precise to define the domain 
${\mathcal T}_C$ as the Riemann surface of the maximal 
solution of (\ref{F0eq}) such that $F_0(0)=0$, 
${F_0}'(0)=(0,\, C)$, as follows. 
Let $D_0:=\{ p\in\C ^2\mid (0,\, p)\in D\}$, where $D$ 
is the domain of definition of the vector field $v$ in (\ref{cssystem}). 
In $M:=\C\times D_0\setminus \{ (\tau,\, p)\in\C\times D_0\mid 
\tau =0\;\mbox{\rm and}\; v(0,\, p)=0\}$ we have the 
regular complex one\--dimensional distribution ${\mathcal D}$ defined by 
the equation $\tau\,\op{d}\! p +v(0,\, p)\,\op{d}\!\tau =0$. 
For the local solution $F_0$ of (\ref{F0eq}) with 
$F_0(0)=0$ and ${F_0}'(0)=(0,\, C)$, the set 
$I_{C,\,\epsilon}:=\{ (\tau ,\, F_0(\tau ))\mid 0<|\tau |<\epsilon\}$ 
is an integral manifold of ${\mathcal D}$; let 
$I_C$ denote the maximal integral manifold of 
${\mathcal D}$ which contains $I_{C,\,\epsilon}$. Then 
${\mathcal T}_C$ is canonically identified with the 
Riemann surface $I_C\cup\{ (0,\, 0)\}$. The inverse of the 
projection $I_C\ni (\tau ,\, p)\to\tau$ 
followed by the projection $(\tau ,\, p)\mapsto p$ 
is the, possibly 
multi\--valued, maximal solution $F_{0,\, C}$ of 
(\ref{F0eq}) mentioned in Lemma \ref{Fjlem}. 
\end{remark}

Lemma \ref{bdyasymlem} below follows from  
O. and R. Costin \cite[Th. 2(ii)]{costin2}. 
Our proof is different. Lemma \ref{bdyasymlem} 
implies all the previous asymptotic expansions in 
Section \ref{perturbss}, but not the explicit estimates such as 
in Lemma \ref{stablelem} and \ref{centerlem}. 

The following Lemma \ref{tnlem} will be used in the last part of 
the proof of Lemma \ref{bdyasymlem}, and in Lemma \ref{polelem}
about the poles of the truncated solutions. 
Lemma \ref{tnlem} is a detailed version of  
the estimate in \cite[(24)]{costin2}. 
\begin{lemma}
Let $\tau\in\C\setminus\{ 0\}$ and $\log\tau$ 
a given solution $\lambda$ of the equation 
$\op{e}^{\lambda}=\tau$. 
Then the solutions $t$ of the equation 
$\op{e}^{-t}\, t^{\alpha}=\tau$, 
where $t^{\alpha}=\op{e}^{\alpha\,\log t}$, 
$\log t=\log |t|+\op{i}\,\op{arg}t$, 
$|\op{arg}t|\leq\pi$, and $|t|$ is large, form two sequences 
$t_n$, where $n\in\Z$, $n>>0$ and $n<<0$ respectively, such that 
\begin{equation}
t_n=2\pi\op{i}n+\alpha\, \log (2\pi\op{i}n)
-\log\tau +s(u,\, w)  
\label{tn}
\end{equation}
as $|n|\to\infty$. Here $\log (2\pi\op{i}n) 
:=\log (2\pi |n|)+\op{i}\,\op{sgn}(n)\,\pi /2$, 
\begin{equation}
u:=\frac{1}{2\pi\op{i}n},\quad v:=\frac{\log (2\pi\op{i}n)}{2\pi\op{i}n}, 
\quad w:=\, -u\,\log\tau +\alpha\, v,
\label{uvw}
\end{equation} 
and $s(u,\, w)$ is a convergent power series in $(u,\, w)$. 
More precisely, $s(u,\, w)$ is  
a complex analytic function of $(u,\, w)$ in a neighborhood 
of $(u,\, w)=(0,\, 0)$, equal to the unique small solution $s$ 
of the equation 
\begin{equation}
s=f(u,\, w,\, s):=\alpha\,\log (1+w+\alpha\, u\,\log  (1+u\, s+w))
\label{sfs}
\end{equation} 
for $u$ and $w$ both small. 
It follows that, with the substitutions (\ref{uvw})
and $\lambda :=\log\tau$,  
\begin{eqnarray*}
s(u,\, w)&=&\alpha\, w+\alpha ^2\, u\, w-\alpha\, w^2/2
+\alpha ^3\, u^2\, w-3\,\alpha ^2\, u\, w^2/2
+\alpha\, w^3/3+\op{O}(n^{-4})\\
&=&-\alpha\,\lambda\, u+\alpha ^2\, v
-\alpha\,\lambda\, (\alpha +\lambda /2)\, u^2
+\alpha ^2\, (\alpha +\lambda )\, u\, v-\alpha ^3\, v^2/2\\
&&-\alpha\,\lambda\, (\alpha ^2+3\,\alpha\,\lambda/2+\lambda ^2/3)\, u^3
+\alpha ^2\, (\alpha ^2+3\,\alpha\,\lambda +\lambda ^2)\, u^2\, v\\
&&-\alpha ^3\, (3\,\alpha /2+\lambda )\, u\, v^2
+\alpha ^4\, v^3/3+\op{O}(n^{-4}).
\end{eqnarray*}

The $t_n$ depend in a complex analytic 
way on $\log\tau$, $\tau\in\C\setminus\{ 0\}$. 
If $\tau$ runs around the origin once in the positive 
direction, then $t_n(\tau )$ moves continuously to $t_{n-1}(\tau )$. 
In this sense the $\tau _n(\tau )$ for $n>>0$ and $n<<0$ 
can be viewed as two multi\--valued complex analytic functions 
of $\tau$.  
\label{tnlem}
\end{lemma}
\begin{proof}
The equation $\op{e}^{-t}\, t^{\alpha}=\tau$ 
is equivalent to the equation 
$t-\alpha\,\log t=2\pi\op{i}n-\log\tau$ 
for some $n\in\Z$. For large $|t|$, $\log |t|=\op{o}(|t|)$, 
$|n|$ is large, $2\pi\op{i}n/t=1+o(1)$, and 
$t=2\pi\op{i}n\, (1+\op{o}(1))$. 

With the notation 
$a:=2\pi\op{i}n-\log\tau$, the equation for $t$ is equivalent to 
$0=\phi (t):=(t-a)\, (1-(t-a)^{-1}\,\alpha\,\log t)$. 
If $\gamma :[0,\, 1]\to\C\setminus\{ a\}$ is a 
Jordan curve which winds once around 
$a$ in the positive direction and 
$|\alpha |\,|\log s|<|s-a|$ for all $s\in\gamma ([0,\, 1])$, 
then the the winding number = $(2\pi )^{-1}$ times the increase of the 
argument of $\phi\circ\gamma$ 
$a$ is equal to one, hence the function $\phi$ has 
a unique zero in the interior of $\gamma$, 
which is simple, and therefore 
in view of the implicit function theorem depends in 
a complex analytic way on $a$. If $|a|>1$ and $|s-a|=|a|-1$, 
then $1\leq |s|\geq 1$,  
$|\alpha |\,|\log s|\leq |\alpha |\, (\pi +\log |s|+\pi )
\leq |\alpha |\, (\pi +\log (2|a|-1))$ and therefore 
$|\alpha |\,|\log s|<|t-a|$ for all $s$ on the circle 
$\gamma$ around $a$ with radius $|a|-1$ if 
$|\alpha |\, (\pi +\log (2|a|-1))<|a|-1$, which happens if 
$|a|$ is sufficiently large.  
Because $|n|>>0$ implies that $|a|=|2\pi\op{i}n-\log\tau |>>0$, 
and $t-a=\alpha\log t=\alpha\,\log (2\pi\op{i}n)+\op{o}(1)
=\alpha\log a +\op{o}(1)=\op{o}(|a|)$, the conclusion is that 
every solution $t$ 
of the equation $t-\alpha\,\log t=2\pi\op{i}n-\log\tau$
such that $t=2\pi\op{i}n\, (1+\op{o}(1))$ is equal to the 
unique zero of $\phi$ in the disc around $a$ with radius $|a|-1$. 
Therefore the solution $t=t_n(\tau )$ of the equation 
$t-\alpha\,\log t=2\pi\op{i}n-\log\tau$ with 
$|\arg t|\leq\pi$ and $|t|>>0$ is unique, 
depends in a complex analytic way on $\tau$, 
and satisfies $t=2\pi\op{i}n\, (1+\op{o}(1))$. 

If we substitute the latter estimate for $t$ in the right hand 
side of $t=2\pi\op{i}n-\log\tau +\alpha\,\log t
=2\pi\op{i}n-\log\tau +\alpha\, 
(\log (2\pi\op{i}n)+\log (1+(2\pi\op{i}n)^{-1}\, 
(-\log\tau +\alpha\,\log t))$, we obtain 
that $t=2\pi\op{i}n+\alpha\, \log (2\pi\op{i}n)
-\log\tau +s$ with $s=\op{O}((\log |n|)/|n|)\to 0$ as 
$|n|\to\infty$. With the definitions of $u$, $v$, and $w$ in 
(\ref{uvw}), $s$ satifies the equation (\ref{sfs}). 
Because $f(0,\, 0,\, s)\equiv 0$, 
it follows from the implicit function theorem in the complex analytic 
setting that there exist open neighborhoods
$U$, $W$, and $S$ of the origin in $\C$ such that 
for each $(u,\, w)\in U\times W$ the equation 
(\ref{sfs}) has a unique solution $s=s(u,\, w)\in S$, 
and that $(u,\, w)\mapsto s(u,\, w)$ is a complex analytic 
function on $U\times W$, with $s(0,\, 0)=0$. 
Substitution of (\ref{uvw}) in $s=s(u,\, w)$ 
leads to (\ref{tn}). Subsequent differentiations 
of $s(u,\, w)=f(u,\, w,\, s(u,\, w))$ with respect 
to $(u,\, w)$ at $(u,\, w)=(0,\, 0)$ 
up to the order three and subsequent evaluations of 
the results at $(u,\, w)=(0,\, 0)$ lead to the 
asymptotic formulas for $t_n(\tau )$ modulo terms of order $n^{-4}$. 
With the help of a formula manipulation 
computer program one could go on a little bit further, but 
the complexity of the asymptotic formulas 
increases rapidly with growing order. 

If $\tau$ runs around the origin once in the positive 
direction, then $2\pi\op{i}n-\log\tau$ moves continuously 
to $2\pi\op{i}(n-1)-\log\tau$, and therefore 
$t_n(\tau )$ moves continuously to $t_{n-1}(\tau )$.  
\end{proof}
  
\begin{lemma}
Let $K$ be a compact subset of $\C ^2$ such that 
$\{ 0\}\times K$ is contained in the domain of definition 
$D$ of the vector field $v$ in (\ref{cssystem}). 
Let ${\mathcal T}_0$ be a bounded subdomain of 
the domain ${\mathcal T}$ in Lemma \ref{Fjlem}, such that $F_0(\tau )\in K$ 
for every $\tau\in {\mathcal T}_0$. 
Let $V$ be a $t$\--domain in the complex upper half plane 
where $|t|>>0$ and $\tau (t)\in {\mathcal T}_0$. 
Then the solution $p_C(t)$ in Lemma \ref{alphalem} 
extends to a, possibly 
multi\--valued, solution $p(t)$ of (\ref{cssystem}) on $V$. 
This solution has the asymptotic expansion 
$p(t)\sim\sum_{i\geq 0}\, t^{-i}\, F_i(\tau (t))$ in the sense 
that, for every $m\in\Z _{>0}$, 
\begin{equation}
p(t)=\sum_{i=0}^{m-1}\, t^{-i}\, F_i(\tau (t))
+\op{O}(t^{-m})
\label{bdyasym}
\end{equation}
as $t\in V$, $|t|\to\infty$. 
\label{bdyasymlem}
\end{lemma}
\begin{proof}
Write 
\begin{equation}
p^{[m]}(t):=\sum_{i=0}^{m-1}\, t^{-i}\, F_i(\tau (t)),
\quad p_0^{[m]}(t):=\sum_{i=0}^{m-1}\, t^{-i}\, F_i(0),
\quad\mbox{\rm and}\quad
\delta (t):=p(t)-p^{[m]}(t). 
\label{deltam}
\end{equation}
According to Lemma \ref{alphalem}, 
$p(t)-p_{\uparrow}(t)\sim C\,\tau (t)$ as $\tau (t)\to 0$. 
The expansion (\ref{expasym}) and equation (\ref{Fi0}) imply that 
$p_{\uparrow}(t)-p_0^{[m]}(t)=\op{O}(t^{-m})$. 
Finally $p^{[m]}(t)-p_0^{[m]}(t)
=F_0(\tau (t))-F_0(0)+\sum_{i=1}^{i-1}\, t^{-i}
\, (F_i(\tau (t))-F_i(0))\sim C\,\tau (t)+\op{O}(t^{-1}\,\tau (t))
\sim C\,\tau (t)$  as $\tau (t)\to 0$, because 
$F_0(\tau )\sim C\,\tau$ as $\tau\to 0$ and 
$\tau (t)\to 0$ implies that $t\to 0$. 
Therefore $y^{[m]}(t):=(p(t)-p_{\uparrow}(t))
-(p^{[m]}(t)-p_0^{[m]}(t))=\op{o}(\tau (t))$, 
as $\tau (t)\to 0$; and 
$\delta (t)=y^{[m]}(t)+\op{O}(t^{-m})$, 
hence (\ref{bdyasym}) $\Leftrightarrow y^{[m]}(t)=\op{O}(t^{-m})$. 

The differential equations 
(\ref{F0eq}) and (\ref{Fieq}) 
imply that 
\begin{eqnarray*}
\frac{\op{d}\! p^{[m]}(t)}{\op{d}\! t}
&=&v(0,\, F_0(\tau (t)))
+\sum_{i=1}^{m-1}\, t^{-i}\, 
\left[L_0\, F_i(\tau (t))+P_i(F_0(\tau (t)),\,\ldots ,
\, F_{i-1}(\tau (t)))\right] \\
&&+t^{-m}\, 
\left.\left(\alpha\,\tau\frac{\op{d}\! F_{m-1}(\tau )}{\op{d}\!\tau} 
-(m-1)\, F_{m-1}(\tau )\right)\right| _{\tau =\tau (t)}
=v(t^{-1},\, p^{[m]}(t))+\op{O}(t^{-m}), 
\end{eqnarray*}
hence 
\begin{eqnarray*}
\frac{\op{d}\! y^{[m]}(t)}{\op{d}\! t} 
&=&\frac{\op{d}\! p(t)}{\op{d}\! t} 
-\frac{\op{d}\! p^{[m]}(t)}{\op{d}\! t}
-\frac{\op{d}(p_{\uparrow}(t)-p_0^{[m]}(t))}{\op{d}\! t}\\
&=&v(t^{-1},\, p(t))-(v(t^{-1},\, p^{[m]}(t))
+\op{O}(t^{-m})) 
+\op{O}(t^{-m-1})\\
&=&v(t^{-1},\, p^{[m]}(t)+y^{[m]}(t))-v(t^{-1},\, p^{[m]}(t))
+\op{O}(t^{-m})
\end{eqnarray*}
as long as $\tau (t)\in {\mathcal T}_0$, where in the last identity 
we have used that 
$p(t)=p^{[m]}(t)+y^{[m]}(t)+(p_{\uparrow}(t)-p_0^{[m]}(t)) 
=p^{[m]}(t)+y^{[m]}(t)+\op{O}(t^{-m})$. 

We start with the proof for $t$ in a domain 
$R_{\eta ,\, r}$ as in (\ref{Rdef}) with $\eta >0$ 
sufficiently small, when (\ref{bdyasym}) is \cite[(20)]{costin2} 
in our situation. 
It follows that $\op{d}\! y^{[m]}(t)/\op{d}\! t
=L(t,\, p^{[m]}(t))\, y^{[m]}(t) 
+\op{O}(\| y^{[m]}(t)\| ^2)+\op{O}(t^{-m})$, 
where $L(t,\, p):=\partial v(t^{-1},\, p)/\partial p$. 
Because $p^{[m]}(t)=p_{\uparrow}(t)
-(p_{\uparrow}(t)-p_0^{[m]}(t))+p^{[m]}(t)-p_0^{[m]}(t))
=p_{\uparrow}(t)+\op{O}(t^{-m})+\op{O}(\tau (t))$, 
we have $L(t,\, p^{[m]}(t))
=L(t)+\op{O}(t^{-m})+\op{O}(\tau (t))$ with $L(t)$ as in 
(\ref{Lt}), hence 
\[
\frac{\op{d}\! y^{[m]}(t)}{\op{d}\! t}
=L(t)\, y^{[m]}(t)+\op{O}(|\tau (t)|\,\| y^{[m]}(t)\| ) 
+\op{O}(\| y^{[m]}(t)\| ^2)+\op{O}(t^{-m}).
\] 
Applying a linear substitution of variables 
$y=A(t)\, z$  as in the proof of Lemma \ref{alphalem} 
we arrive at an integral equation 
$z={\mathcal H}(z)$ for $z:t\mapsto A(t)^{-1}\, y^{[m]}(t)$, 
where ${\mathcal H}(z)(t)^+=H(z)(t)^+$ and 
${\mathcal H}(z)(t)^-$ are as in (\ref{inteq+2z}) and 
(\ref{inteq-2zl}), respectively, 
$C=0$ in view of $z(t)=A(t)^{-1}\, y^{[m]}(t)=\op{o}(\tau (t))$, 
and $h(t^{-1},\, z)=\op{O}(|\tau (t)|\,\| z\| ) 
+\op{O}(\| z\| ^2)+\op{O}(t^{-m})$. 
If $\eta$ and $1/r$ are sufficiently small, 
then ${\mathcal H}$ is a contraction on the space ${\mathcal Z}$ 
of all continuous functions $z:R_{\eta ,\, r}\to\C ^2$  
that are  complex analytic on the interior of $R_{\eta ,\, r}$ and  
satisfy a uniform bound $\| z(t)\|\leq C\, |t|^{-m}$, 
where ${\mathcal H}$ is a contraction with respect to 
the metric $(z_1,\, z_2)\mapsto\sup_{t\in R_{\eta ,\, r}}
\, |t|^m\,\| z_1(t)-z_2(t)\|$. It follows that ${\mathcal H}$ 
has a unique fixed point $z$ in ${\mathcal Z}$, 
when $y^{[m]}(t)=A(t)\, z(t)$ 
satisfies $y^{[m]}(t)=\op{O}(t^{-m})$ 
when $t\in R_{\eta ,\, r}$ and $|t|\to\infty$. 
This completes the proof for $t$ in a domain where 
$|\tau (t)|$ remains sufficiently small. 

For $t$ in a subdomain of $V$ 
where $\tau (t):=\op{e}^{-t}\, t^{\alpha}\in {\mathcal T}_0$  
remains bounded away from zero, 
we use $\tau =\tau (t)$ instead of $t$ as the independent variable. 
The domain ${\mathcal T}_0$ can be arranged such that its points 
$\tau$ can be joined with points $\tau _0$ close to the origin 
by smooth paths $\gamma$ in ${\mathcal T}_0$,  
parametrized by arclength, 
with a uniformly bounded length, and 
staying at a uniform distance away from the origin. 
The solution $y^{[m]}(\tau )$ will be estimated along such $\gamma$. 
With $\tau$ along $\gamma$, we view $t=t_n(\tau )$, $n>>0$, 
as the multi\--valued inverse in Lemma \ref{tnlem}
of the function $t\mapsto\tau (t)$ with large $|t|$ and 
$0<\op{arg}t<\pi$. Then $t=t_n(\tau )=2\pi\op{i}n\, (1+\op{o}(1))$,  
uniformly for $\tau$ along $\gamma$. 
Because  
\[
\frac{\op{d}\! y^{[m]}}{\op{d}\!\tau}
=(-1+\alpha /t)^{-1}\,\tau ^{-1}\,\frac{\op{d}\! y^{[m]}}{\op{d}\! t}
=\op{O}(\| y^{[m]}\|)+\op{O}(|t_n(\tau )|^{-m}), 
\]
we have
\[
\frac{\op{d}\! \| y^{[m]}(\gamma (s))\|}{\op{d}\! s}
\leq\left\|\frac{\op{d}\! y^{[m]}(\gamma (s))}{\op{d}\! s}\right\| 
\leq\left\| (y^{[m]})'(\gamma (s))\right\|   
\leq C_1\, \| y^{[m]}(\gamma (s))\| +C_2\, |t_n(\gamma (s))|^{-m}
\]
for some positive constants $C_1$ and $C_2$. Or, 
\[
\frac{\op{d}\! \| y^{[m]}(\gamma (s))\|}{\op{d}\! s}
=C_1\, \| y^{[m]}(\gamma (s))\| +C_2\, |t_n(\gamma (s))|^{-m}-r(s)
\]
for a non\--negative continuous function $r(s)$. 
Lagrange's variation of constants formula yields 
\begin{eqnarray*}
\| y^{[m]}(\gamma (s))\| &=&\op{e}^{C_1\, s}
\| y^{[m]}(\gamma (0))\| 
+\int_{s_0}^s\,\op{e}^{C_1\, (s-\sigma )}
\, (C_2\, |t_n(\gamma (\sigma ))|^{-m}-r(\sigma ))\,\op{d}\!\sigma \\
&\leq& \op{e}^{C_1\, s}
\| y^{[m]}(\gamma (0))\| 
+C_2\, \int_{0}^s\,\op{e}^{C_1\, (s-\sigma )}
\, |t_n(\gamma (\sigma ))|^{-m}\,\op{d}\!\sigma .
\end{eqnarray*}
Write $t=t_n(\gamma (s))$, where $|t|>>0$. 
Then, uniformly for $0\leq\sigma\leq s$, 
$t_n(\gamma (\sigma ))=2\pi\op{i}n\, (1+\op{o}(1))
=t\, (1+\op{o}(1))$, hence 
$|t_n(\gamma (\sigma ))|^{-m}=\op{O}(|t|^{-m})$. 
Furthermore, because $\gamma (0)=\tau _0$ is small, 
$\|y^{[m]}(\gamma (0))\| =\op{O}(|t_n(\gamma (0))|^{-m})
=\op{O}(|t|^{-m})$. Combination of the 
estimates yields $y^{[m]}(\gamma (s))=\op{O}({t}^{-m})$. 
That is, returning to $t$ as the independent variable, 
$y^{[m]}(t)=\op{O}({t}^{-m})$. 
\end{proof}
Lemma \ref{bdyasymlem} can be used in order to obtain asymptotic 
information about the large $t$ where the solution $p_C(t)$ 
becomes singular, where $\tau (t)$ is close to points 
$\tau$ where the solution $F_0(\tau )$ of (\ref{F0eq}) 
becomes singular. This will be done in more detail for the 
Boutroux\--Painlev\'e system in Subsection \ref{ttss}, 
leading to the asymptotic description in Lemma \ref{polelem} 
of the poles in the domain $R_{\eta ,\, r}$ with large $\eta$ 
and $r$ correspondingly large. These poles correspond to the first 
sequence of poles of the truncated and triply truncated solutions of 
the first Painlev\'e equation which appear beyond the boundary 
of the truncated domains.  

\subsection{Truncated and triply truncated solutions}
\label{ttss}
In this subsection we collect the conclusions from Subsection 
\ref{perturbss} for the Painlev\'e equation. 

Boutroux \cite[\S 13]{boutroux} found a family of solutions 
$u(z)$ of the Boutroux\--Painlev\'e system (\ref{u01dot}), 
depending in a complex analytic way on one complex 
variable, such that $u(z)$ converges to the equilibrium point 
$(\epsilon\op{i}\sqrt{6}, 0)$ of (\ref{u01dotlim}) 
as $t=\lambda _+\, z$ runs to infinity in the complex plane in 
the direction of the positive real axis. On 
\cite[p. 346]{boutroux}, he called these solutions 
{\em truncated in the direction of the positive real axis}. 
These solutions correspond to the solutions 
$p_{t_0,\, a^-}(t)$ and $p_l(t)$ described in the 
lemmas \ref{stablelem} and 
\ref{alphalem}. 

In \cite[\S 14]{boutroux}, Boutroux found a solution $u(z)$ 
which converges to the equilibrium point of (\ref{u01dotlim}) 
as $t=\lambda ^+\, z$ runs to infinity in a sector 
in the complex plane which does not contain the 
positive imaginary axis. Because such a sector 
contains the three orthogonal axes equal to the negative real, 
negative imaginary, and positive real one, he called this solution 
the {\em triply truncated solution}. This solution corresponds 
to the solution $p_{\downarrow}(t)$ described in 
the lemmas \ref{centerlem}, \ref{centerVlem}, and \ref{alphalem}. 
In the sequel we will follow Boutroux's terminology.  
\begin{remark}
The proofs in Boutroux \cite[\S 13, 14]{boutroux} contain gaps    
and errors. In the formula for $Y_1$ following 
\cite[(51)]{boutroux} the factor $X^{-1}$ in the first and 
second integral should be replaced by $\, -X^{-2}/\sqrt{12}$ and 
$X^{-2}/\sqrt{12}$, respectively. A similar correction 
is needed in \cite[(54 bis)]{boutroux}. 
In the right hand side of \cite[(54)]{boutroux} the 
term $6\,\sum_{k=1}^{j-1}\, Y_k\, Y_{j-k}$ is missing. 
More seriously, the inductive assumption that 
$Y_k(X)\to 0$ as $X\to\infty$ for every $k<j$ 
implies that  the second limit relation at the bottom 
of \cite[p. 343]{boutroux} is automatically 
satisfied, when the first one is equivalent to the 
condition that $Y_j(X)\to 0$ as $X\to\infty$. 
This determines $Y_j$ only up to the addition of 
a constant times $\op{e}^{-\sqrt{12}\, X}$. Therefore 
the estimate $l_j<K\, l_{j-1}\,\overline{X}^{\alpha -1}$ 
on \cite[p. 345]{boutroux} cannot be proved for the  
solutions $Y_j$  described by Boutroux. 
It is a bit surprising that Boutroux did not use the 
method of Cotton \cite{cotton}, which appeared two years 
earlier in the same journal as the paper of Boutroux. 
\label{bttrem}
\end{remark}
\begin{remark}
On \cite[p. 261]{boutroux}, Boutroux gave a definition of 
truncated solutions which looks quite different from 
the definition in terms of the convergence to the equilibrium 
point of the limit system. In our notation, we understand 
his definition on \cite[p. 261]{boutroux} in the following way. 
For every $z_0\in\C\setminus\{ 0\}$ and $c\in\C$ there is a 
unique solution $U(z)=U_{z_0,\, c}(z)$ of (\ref{boutrouxeq}) 
which has a pole at $z=z_0$ and $c$ as the coefficient 
of $(z-z_0)^4$ in its Laurent expansion at $z=z_0$. 
The parameter $c$ corresponds in an bijective affine manner 
to the position on the pole line $L_9\setminus L_8^{(1)}$ 
in Okamoto's space, see (\ref{laurentu}). Assume that 
$u_{z_0,\, c}(z)$ has another pole at $z=z_1\neq z_0$. 
Because the vector field is regular, nonzero at and transversal 
to the pole line, an application of the implicit function theorem 
yields that for every $({z_0}',\, c')$ near $(z_0,\, c)$ there 
is a unique ${z_1}'=Z_1({z_0}',\, c')$ near $z_1$ such that the 
solution $u_{{z_0}',\, c'}(z)$ has a pole at $z={z_1}'$, 
and the function $({z_0}',\, c')\mapsto Z_1({z_0}',\, c')$ is 
complex analytic. It seems that Boutroux considered this 
to be evident, as he did not provide further proof. Then, according to 
our interpretation of \cite[p. 261]{boutroux}, 
$u_{{z_0}'',\, c''}(z)$ is a truncated solution if the function 
$Z_1({z_0}',\, c')$ has a complex analytic continuation for 
$({z_0}', c')$ on a path approaching $({z_0}'',\, c'')$, but that 
the superior limit of $|Z_1({z_0}',\, c')|$ is infinite if 
$({z_0}',\, c')\to ({z_0}'',\, c'')$ along the path. 
We find this an interesting definition of truncated 
solutions. However, we did not find any statement or 
proof in Boutroux's paper which relates 
the truncated solutions as defined on \cite[p. 261]{boutroux} with the 
truncated solutions defined in term of their convergence to one of the 
equilibrium points of the limit system. We will stay with the 
second definition, and view it as a challenge to find relations 
with the first one. 
\label{trdefrem}
\end{remark}
According to Remark \ref{boutrouxjkconstantsrem}, 
the substitutions $x=\, -2^{-3/5}\, 3^{-1/5}\,\xi$, 
$y(x)=2^{-4/5}\, 3^{-3/5}\,\eta (\xi )$ turn the first Painlev\'e 
equation (\ref{PI}) into 
\begin{equation}
\op{d}^2\eta/\op{d}\!\xi ^2=(\eta ^2-\xi )/2, 
\label{PIetaxi}
\end{equation}
when the substitutions 
$\xi =((5/4)\, t)^{4/5}$, $\eta (\xi )=\xi ^{1/2}\,\pi _1(t)$, 
and $\eta '(\xi )=\xi ^{3/4}\,\pi _2(t)$ turn 
the corresponding first order system for $(\eta ,\,\eta ')$ into 
\begin{equation}
\begin{array}{lll}
\op{d}\!\pi _1/\op{d}\! t&=&\pi _2-2\, (5\, t)^{-1}\,\pi _1,\\
\op{d}\!\pi _2/\op{d}\! t&=&({\pi _1}^2-1)/2
-3\, (5\, t)^{-1}\,\pi _2.
\end{array}
\label{upsilon01dot}
\end{equation}
At the equilibrium point $(\pi _1,\,\pi _2)=(1,\, 0)$ 
of the limit system for $t=\infty$, the linearization 
has the eigenvalues $\pm 1$, and the 
substitution of variables 
$\pi _1=1+p^++p^-$, $\pi _2=p^+-p^-$ lead to the system 
(\ref{cssystem}) with 
\begin{eqnarray*}
v(1/t,\, p)^+&=&p^+-1/5t-p^+/2t+p^-/10t+(p^++p^-)^2/4,\\
v(1/t,\, p)^-&=&-p^--1/5t-p^-/2t+p^+/10t-(p^++p^-)^2/4.
\end{eqnarray*}
It follows that the vector $c_1$ in (\ref{c1eq}) and the matrix 
$L_1$ in (\ref{L1eq}) are equal to 
\[
c_1=\left(\begin{array}{c}1/5\\-1/5\end{array}\right)
\quad\mbox{\rm and}\quad
L_1=\left(\begin{array}{cc}-1/2 & 1/10 \\
1/10 & -1/2\end{array}\right) ,
\]
respectively. Because $c_1^++c_1^-=0$, the 
asymptotic expansion (\ref{powerasym}) implies 
that for the truncated solutions $p(t)$ we have 
$\pi _1(t)=1+\op{O}(t^{-2})$ as $t\in S$, $|t|\to\infty$. 
Furthermore, {\em both the left upper corner 
and the right lower corner $\alpha$ of $L_1$ are equal 
to $-1/2$}, and therefore $p(t)-p_{\uparrow}(t)$ is asymptotically 
equal to $\tau (t)=\op{e}^{-t}\, t^{-1/2}$ times a series in nonnegative 
powers of $t^{-1}$ as $t\in\Sigma$, $|t|\to\infty$, see (\ref{expasym}). 
The fact that $\op{Re}\alpha <0$ implies that 
{\em the large positive and 
negative parts of the imaginary axis are contained in the 
interior of the domains $R_{\eta ,\, r}$ in (\ref{Rdef})}. 
More precisely, the part in $|t|\geq r$ of the boundary 
of $R_{\eta ,\, r}$ is given by 
the equation 
\begin{equation}
\cos (\op{arg}t)=(-(1/2)\,\log |t|-\log\eta )/|t|,\quad 
|t|\geq r,\quad
-\pi <\op{arg}t<\pi ,
\label{boundaryR}
\end{equation}
which is to the left of the imaginary axis if and only if 
$|t|>\eta ^{-2}$. See Figure \ref{regionfig}. 

\begin{figure}[ht]
\centering
\includegraphics[width=15cm]{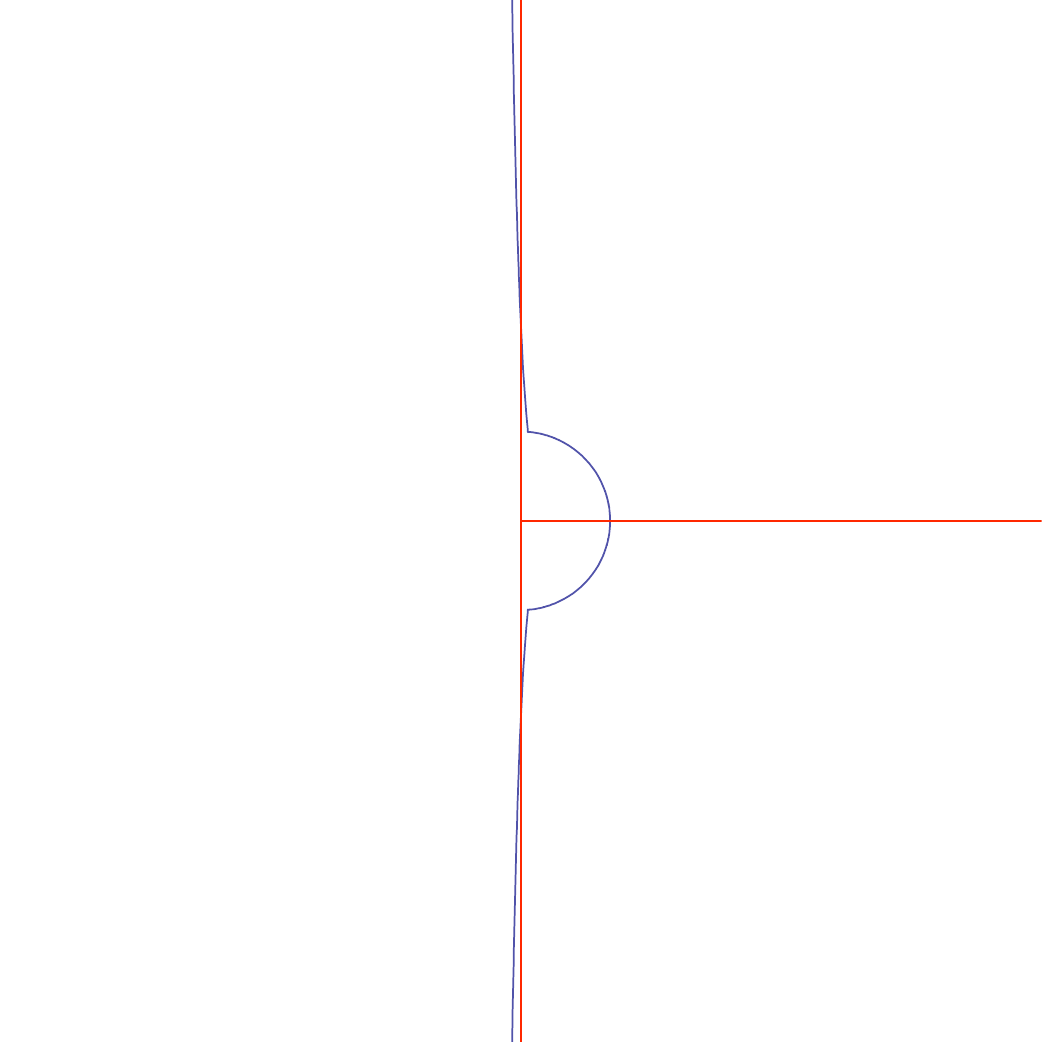}
\caption{The region $R_{\eta,\, r}$, to the right of the curved boundary} 
\label{regionfig}
\end{figure}

Because the Painlev\'e property implies that 
all solutions $y(x)$ of (\ref{PI}) are single\--valued, 
the analytic continuation of 
$(\pi _1(t),\,\pi _2(t))$ along the path $t\,\op{e}^{\scriptop{i}\theta}$, 
where $\theta\in\R$ runs from $0$ to $5/4$ times $2\pi$, 
applied to the substitutions 
\begin{equation}
x=\, -2^{-3/5}\, 3^{-1/5}
\, (5\, t/4)^{4/5},\quad y(x)=2^{-4/5}\, 3^{-3/5}\, 
(-2^{3/5}\, 3^{1/5}\, x)^{1/2}\, \pi _1(t)
\label{xt}
\end{equation}
leaves the 
solution $y(x)$ of (\ref{PI}) invariant. 
We conclude as in (\ref{uangle5/4}) that the analytic continuation of 
$(\pi _1(t),\,\pi _2(t))$ along the aforementioned path 
is equal to $(-\pi _1(t),\,\op{i}\pi _2(t))$. 
This agrees with the symmetry 
$(t,\,\pi _1,\,\pi _2)\mapsto 
(\op{i}t,\, -\pi _1,\,\op{i}\pi _2)$ 
of order four of the system (\ref{upsilon01dot}). 
This analytic continuation converges for 
$t$ running to infinity in the direction of 
the positive imaginary axis to the other equilibrium point 
$(\pi _1,\,\pi _2)=(-1,\, 0)$ of the limit system of 
(\ref{upsilon01dot}) for $t=\infty$. 
Applying this analytic continuation a second and third time,   
one obtains a solution which converges to 
$(\pi _1,\,\pi _2)=(1,\, 0)$ 
and $(\pi _1,\,\pi _2)=(-1,\, 0)$ for $t$ running to 
infinity in the direction of the negative real 
and imaginary axis, respectively. 
In this way each of the four types of truncated solutions, 
for $t$ running to infinity in the direction of the 
positive and negative real axis with 
$(\pi _1(t),\,\pi _2(t))$ converging to 
$(1,\, 0)$, and for $t$ running 
to infinity in the direction of the positive and negative 
imaginary axis 
with $(\pi _1(t),\,\pi _2(t))$ 
converging to $(-1,\, 0)$, 
are obtained by analytic continuation from the solutions 
$p_{t_0,\, a^-}(t)=p_l(t)$ in Lemma \ref{stablelem} and 
\ref{alphalem}. 

Similarly, the analytic continuation of $p_{\uparrow}(t)$ 
along the path $t\,\op{e}^{\scriptop{i}\theta}$ 
is equal to $p_{\downarrow}(t)$ if 
$\theta\in\R$ runs from $0$ to $5\pi$. 
As this agrees with the symmetry 
$(t,\,\pi _1,\,\pi _2)\mapsto 
(-t,\,\pi _1,\, -\pi _2)$ of the system 
(\ref{upsilon01dot}), it follows that 
$p_{\downarrow}(t)^+=p_{\uparrow}(-t)^-$ and 
$p_{\downarrow}(t)^-=p_{\uparrow}(-t)^+$. 
If $\theta$ runs from $0$ to $5\pi /2$ and $15\pi /2$, then 
the analytic continuation of $p_{\uparrow}(t)$ is equal to 
the two triply truncated solutions near $(\pi _1,\,\pi _2)=(-1,\, 0)$. 
In this way each of the four triply truncated solutions, 
for $t$ running to infinity in the upper and lower half plane 
with $(\pi _1(t),\,\pi _2(t))$ converging to $(1,\, 0)$ 
and for $t$ running to infinity in the left and right half plane 
with $(\pi _1(t),\,\pi _2(t))$ converging to $(-1,\, 0)$, 
is obtained by analytic continuation from the 
solution $p_{\uparrow}(t)$ in Lemma \ref{centerlem}, \ref{centerVlem}, 
and \ref{alphalem}.  

Because the system (\ref{upsilon01dot}) is real, 
it has the symmetry $(t,\,\pi _1,\,\pi _2)\mapsto 
(\overline{t},\,\overline{\pi _1},\,\overline{p_2})$. 
Therefore, if $(\pi _1(t),\,\pi _2(t))$ is a truncated 
solution of (\ref{upsilon01dot}) in the sense that it converges to 
$(1,\, 0)$ as $t$ runs to infinity in the direction of the 
positive real axis, then 
$t\mapsto (\overline{\pi _1(\overline{t})}, 
\, \overline{\pi _2(\overline{t})})$ is a solution 
of (\ref{upsilon01dot}) with the same limit behavior, 
and therefore is truncated in the same way. 
The triply truncated solutions satisfy 
$p_{\downarrow}(t)=\overline{p_{\uparrow}(\overline{t})}$, 
which in combination with  
$p_{\downarrow}(t)^{\pm}=p_{\uparrow}(-t)^{\mp}$ 
implies that $p_{\uparrow}(t)^{\pm}
=\overline{p_{\uparrow}(-\overline{t})^{\mp}}$. 
The corresponding triply truncated 
solution of (\ref{upsilon01dot}) satisfies 
$\pi _1(t)=\overline{\pi _1(-\overline{t})}$ 
and $\pi _2(t)=\, -\overline{\pi _2(-\overline{t})}$.

\begin{lemma}
The Stokes constant $S$ in (\ref{Sdef}) is nonzero and purely imaginary. 
\label{Slem}
\end{lemma}
\begin{proof}
Because $p_{\downarrow}(t)=\overline{p_{\uparrow}(\overline{t})}$, 
the asymptotic identity for $t\in\R$, $t\to\infty$ 
implies that $S$ is purely imaginary. 
If $S=0$, then 
$p_{\downarrow}(t)=p_{\uparrow}(t)$ in the right half plane, 
and therefore $p_{\uparrow}(t)$ and $p_{\downarrow}(t)$ 
would have a common extension to a small solution $p(t)$ 
for $\,- (3/2)\pi\leq\op{arg}(t)\leq (3/2)\pi$, 
$|t|\geq r$, when the corresponding solution 
$y(x)$ would be bounded by a constant times $|x|^{1/2}$ 
for $\, -(6/5)\pi\leq\op{arg}(x)\leq (6/5)\pi $
and $|x|\geq R$ for some $R$. As this implies that 
the single valued function $y(x)$ has no poles for large $|x|$, 
it would follow that $y(x)$ has only finitely many poles, 
in contradiction with Corollary \ref{infpolecor}. 
\end{proof} 

\begin{remark}
Because $\tau (t)=\op{e}^{-t}\, t^{-1/2}$, 
the positive and negative imaginary axis are 
transitional domains as in the paragraphs preceding 
Lemma \ref{Fjlem}, the Stokes constant $S$ is equal to the one 
in Costin \cite[(2.8)]{costin}. 
According to \cite[Note (3) on p. 7]{costin}, 
$S=\op{i}\sqrt{6/5\pi}$.
The formula $p_{\downarrow}(t)-p_{\uparrow}(t)\sim 
\op{i}\sqrt{6/5\pi}\,
\op{e}^{-t}\, t^{-1/2}$ as $-\pi /2\leq\op{arg}t\leq\pi /2$, 
$|t|\to\infty$  
agrees, at least up to the sign, with 
Kapaev \cite[Cor. 2.4]{kapaevstokes}. 
\label{stokesrem}
\end{remark}
Figure \ref{regionfigp} shows a region 
in the $x$\--plane, the plane of definition of the solution of 
(\ref{PI}), where the corresponding 
truncated solution of (\ref{u01dot}) is close 
to one of the two equilibrium points of the limit system 
(\ref{u01dotlim}). This region is the image 
under the mapping $t\mapsto -2^{-3/5}\, 3^{-1/5}\, ((5/4)\, t)^{4/5}$ 
of a domain of the form $R_{\eta ,\, r}$. 
\begin{figure}[ht]
\centering
\includegraphics[width=15cm]{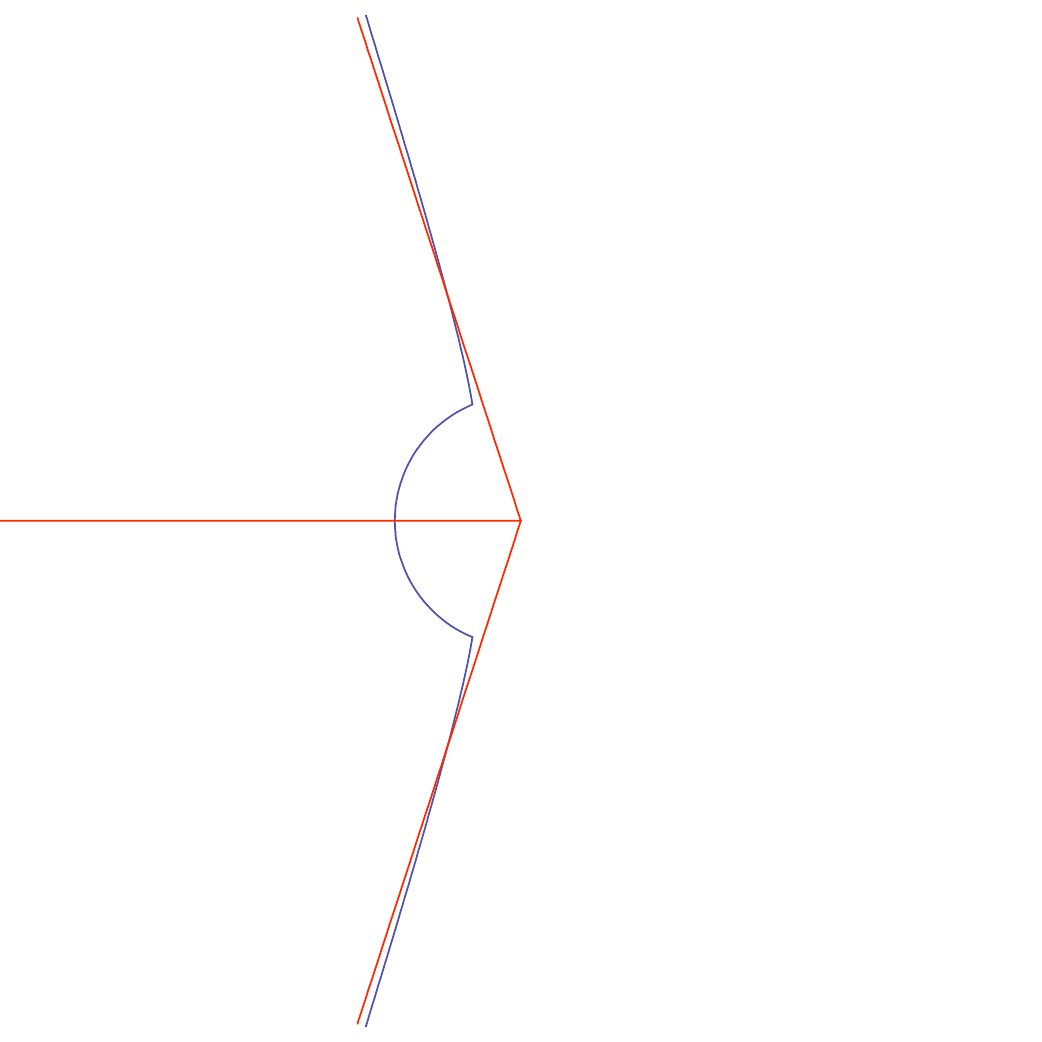}
\caption{Truncated region in the $x$\--plane,  
to the left of the curved boundary} 
\label{regionfigp}
\end{figure}

The triply truncated 
solution $p_{\uparrow}(t)$ 
of (\ref{u01dot}) is defined on a domain determined by the 
inequalities $|t|\geq r$ and  
\[
-\op{arccos}((-(1/2)\,\log |t|-\log\eta )/|t|)\leq\op{arg}t\leq  
\pi +\op{arccos}((-(1/2)\,\log |t|-\log\eta )/|t|) .
\]
Here $r$ and $\eta$ are sufficiently large and small strictly 
positive real numbers, and $\op{arccos}$ denotes the inverse of 
the bijective function $\cos :[0,\,\pi ]\mapsto [-1,\, 1]$. 
As the inequalities for $|t|$ and $\op{arg}t$ allow points 
$t,\, t'$ such that $|t'|=|t|$ and $\op{arg}t'=\op{arg}t+2\pi$, 
the function $p_{\uparrow}(t)$ is interpreted as multi\--valued. 
The properties $p_{\uparrow}(t)^{\pm}
=\overline{p_{\uparrow}(-\overline{t})^{\mp}}$, 
(\ref{Sdef}), and $p_{\downarrow}(t)=\op{O}(t^{-1})$ 
when $t$ runs to infinity in the direction of the 
negative imaginary axis imply that for 
$\op{arg}t=\, -\pi/2$ and $\op{arg}t=(3/2)\,\pi$ 
we have $p_{\uparrow}(t)\sim \op{e}^{-t}\, t^{-1/2}\, (0,\, -S)$ 
and $p_{\uparrow}(t)\sim\op{e}^t\, t^{-1/2}\, (-\op{i}\,\overline{S},\, 0)$
as $|t|\to\infty$, respectively. As Lemma \ref{Slem} implies that $S\neq 0$, 
{\em the two branches of $p_{\uparrow}(t)$ do not 
coincide on the overlap}. 

The image under the mapping 
$t\mapsto x=\, -2^{-3/5}\, 3^{-1/5}\, ((5/4)\, t)^{4/5}$ 
of the aforementioned domain where 
$p_{\uparrow}(t)$ is small is a domain in the $x$\--plane  
where $|x|$ is large and $\op{arg}(x)$ runs from 
slightly smaller than $(3/5)\,\pi$ to slightly larger 
than $(11/5)\,\pi$. The other truncated 
and triply truncated regions are obtained from these by 
applying a rotation in the $x$\--plane 
over $k\, 2\pi/5$, $1\leq k\leq 4$. 
Because the truncated and triply truncated solutions 
of (\ref{u01dot}) are bounded in their domains of definition, 
they have no pole there, and therefore 
{\em the corresponding truncated and triply truncated 
solutions of (\ref{PI}) have no poles in the 
aforementioned truncated and and triply truncated regions 
in the $x$\--plane}.  

Figure \ref{regionfigpt} shows the unique triply 
truncated region in the $x$\--plane which is invariant 
under complex conjugation. If $y(x)$ denotes the 
corresponding triply truncated solution of 
(\ref{PI}), the function $x\mapsto\overline{y(\overline{x})}$ 
is a solution of (\ref{PI}) which is triply truncated 
in the same domain. The uniqueness of 
triply truncated solutions of (\ref{u01dot}) leads to  
the following observation of Joshi and Kitaev 
\cite[Cor. 3]{jkit}. 
\begin{lemma}
Let $D$ denote the triply truncated domain in 
Fig. \ref{regionfigpt} which 
is invariant under complex conjugation.  
Then the  solution $y(x)$ of (\ref{PI}) 
which is triply truncated on $D$ is real in the sense 
that $y(x)=\overline{y(\overline{x})}$ for every $x\in\C$. 
In particular $y(x)\in\R$ for every $x\in\R$ not equal to 
a pole point of $y$. 
\end{lemma}

\begin{figure}[ht]
\centering
\includegraphics[width=15cm]{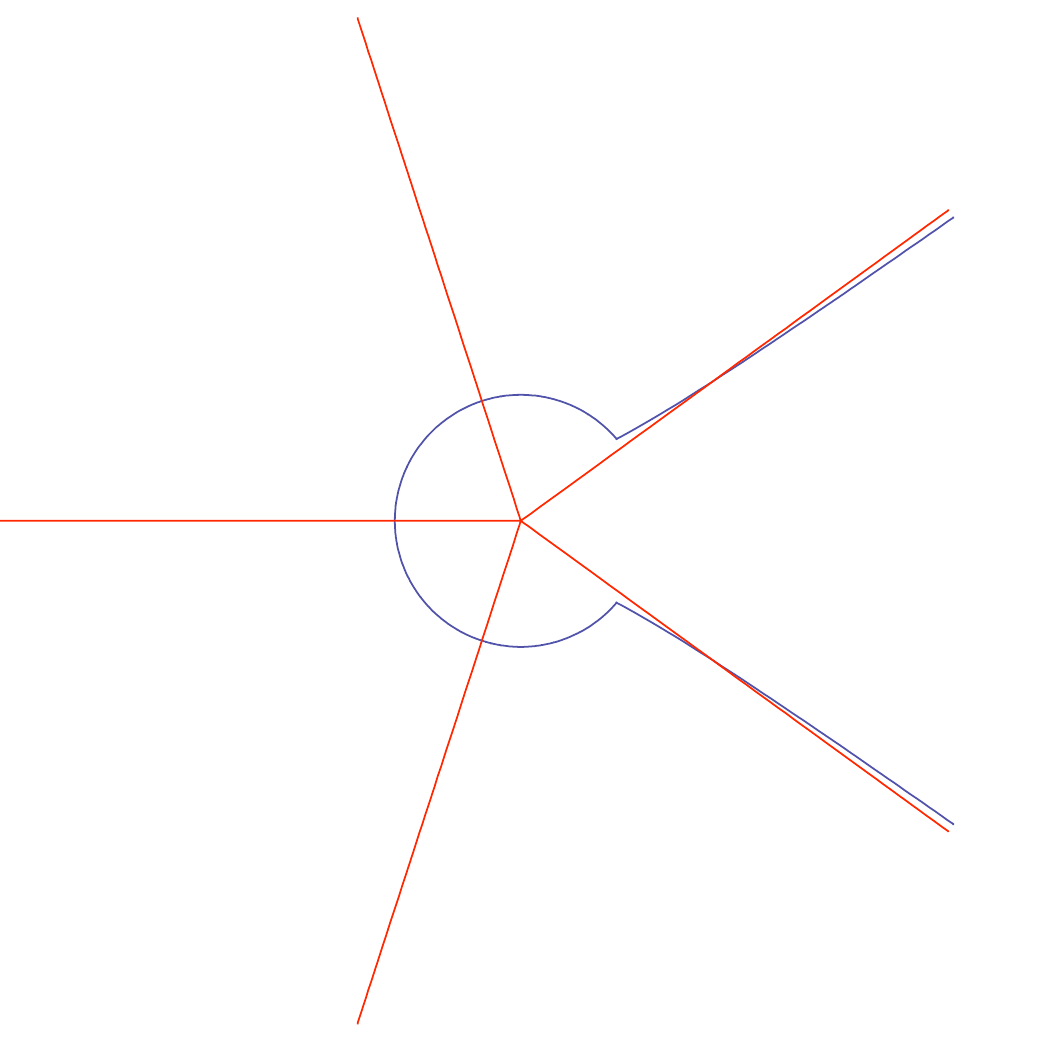}
\caption{Triply truncated region in the $x$\--plane, 
to the left of the curved boundary} 
\label{regionfigpt}
\end{figure}

The next lemma discusses what happens with Lemma \ref{bdyasymlem} 
in the case of the Boutroux\--Painlev\'e system. 
Our results correspond to \cite[(113)]{costin2} 
with the formulas for $H_0(\xi )$ and $H_1(\xi )$ 
on \cite[p. 38]{costin2}, as 
$\pi _1(x)=Y(x)=1-4/25\, x^2+h(x)$, 
see \cite[p. 36]{costin2}, hence 
$\pi _{1,\, 0}(\xi )=H_0(\xi )+1$, 
$\pi _{1,\, 2}(\xi )=H_2(\xi )-4/25$ and 
$\pi _{1,\, l}(\xi )=H_(\xi )$ for all 
$l\notin \{0,\, 2\}$. O. and R. Costin 
\cite[p. 39]{costin2} wrote: \lq\lq We omit the straightforward 
but quite lengthy inductive proof that all $H_k$ are 
rational functions of $\xi$.\rq\rq\  And on p. 40: \lq\lq For large 
$\xi$ induction shows that $H_n\sim\op{Const}_n\,\xi ^n$, 
\dots \rq\rq\ , but did not give further details of the proof.   

\begin{lemma}
With the notation of Lemma \ref{bdyasymlem}, 
the solution $(\pi _1(t),\,\pi _2(t))$ of (\ref{upsilon01dot}) 
corresponding to 
$p_C(t)$ has the asymptotic expansion 
\begin{equation}
\pi _k(t)=\sum_{l=0}^{m-1}\, t^{-l}\, \pi _{k,\, l}(C\,\tau (t))
+\op{O}(t^{-m})
\label{bdyasymh}
\end{equation}
as $t\in V$, $|t|\to\infty$. Here 
$\pi _{1,\, l}(\xi )=(\xi -12)^{-l-2}\, P_{1,\, l}(\xi )$ and 
$\pi _{2,\, l}(\xi )=(\xi -12)^{-l-3}\, P_{2,\, l}(\xi )$, where 
$P_{1,\, l}(\xi )$ and $P_{2,\, l}(\xi )$ is a polynomial 
in $\xi :=C\,\tau$ of degree $\leq 2\, l+2$ and $\leq 2\, l+3$, respectively. 
We have 
\begin{eqnarray}
P_{1,\, 0}(\xi )&=&(\xi -12)^2+144\,\xi ,\label{P10}\\
P_{2,\, 0}(\xi )&=&144\,\xi (\xi +12),\label{P20}\\
P_{1,\, 1}(\xi )&=&\xi\, (216+210\,\xi +3\,\xi ^2-\xi ^3/60),\label{P11}\\
P_{2,\, 1}(\xi )&=&
(497664-134784\,\xi +266112\,\xi ^2+25704\,\xi ^3-24\,\xi ^4+\xi ^5)/60.
\label{P21}
\end{eqnarray}
\label{hlem}
\end{lemma}
\begin{proof}
It follows from the last paragraph in Lemma \ref{Fjlem} 
that it suffices to prove all the formulas for $C=1$, when $\tau =\xi$. 

The shortest proof of the formulas for 
$\pi _{1,\, 0}(\xi )$ and $\pi _{2,\, 0}(\xi )$ is 
to verify that these functions 
satisfy the differential equations 
$-\xi\,\op{d}\! P_{1,\, 0}/\op{d}\!\xi =P_{2,\, 0}$, 
$-\xi\,\op{d}\! P_{2,\, 0}/\op{d}\!\xi =({P_{1,\, 0}}^2-1)/2$
corresponding to (\ref{F0eq}), with the initial conditions 
$P_{1,\, 0}(0)=1$, $P_{2,\, 0}(0)=0$, and the derivatives 
with respect to $\xi$ at $\xi =0$ of $p^+=(P_{1,\, 0}+P_{2,\, 0}-1)/2$ 
and $p^-=(P_{1,\, 0}-P_{2,\, 0}-1)/2$ equal to $0$ and $C$, 
respectively. The longer proof below explains how 
the formulas for $\pi _{1,\, 0}(\xi )$ and $\pi _{2,\, 0}(\xi )$ 
could have been found. 

The system (\ref{upsilon01dot}) is equivalent to 
the second order differential equation 
\begin{equation}
\frac{\op{d}^2\pi}{\op{d}\! t^2}
=\, -\frac{1}{t}\,\frac{\op{d}\!\pi}{\op{d}\! t}
+\frac12\, (\pi ^2-1)+\frac{4}{25\, t^2}\,\pi 
\label{pidd}
\end{equation}
for $\pi (t)=\pi _1(t)$, when $\pi _2(t)$ is given in terms 
of $\pi (t)$ by means of the formula
\begin{equation}
\pi _2(t)=\frac{\op{d}\!\pi (t)}{\op{d}\! t}+\frac{2}{5\, t}\,\pi (t).
\label{pi2}
\end{equation}
The autonomous limit equation of (\ref{pidd}) for $t\to\infty$ 
is $\Pi ''=(\Pi ^2-1)/2$, a Newton equation with potential 
energy $V(\Pi ):=\, -\Pi ^3/6+\Pi/2$. It follows that the total energy 
$E=(\Pi ')^2/2+V(\Pi )$ is a constant of motion, and the solution 
which converges to the equilibrium point $(\Pi ,\,\Pi ')=(1,\, 0)$ 
has energy $E=V(1)=1/3$. This leads to the first order 
differential equation 
$(\Pi ')^2=2\, (1/3-V(\Pi ))=(\Pi -1)^2\, (\Pi +2)/3$, hence, 
if the independent variable is denoted by $s$,  
$\op{d}\! s/\op{d}\!\Pi =3^{1/2}\, (\Pi -1)^{-1}\, (\Pi +2)^{1/2}$, 
when the substitution $\Pi +2=\Psi ^2$ leads to 
\begin{eqnarray*}
s&=&2\,\sqrt{3}\,\int^{\Psi (s)}\, 
(\Psi ^2-3)^{-1}\,\op{d}\!\Psi +c
=\log \frac{\Psi (s)-\sqrt{3}}{\Psi (s)+\sqrt{3}}+c\\
\Leftrightarrow&&
\frac{\Psi (s)+\sqrt{3}}{\Psi (s)-\sqrt{3}}=c\,\op{e}^{-s}=:c\,\xi 
\Leftrightarrow
\Psi (s)=\sqrt{3}\,\frac{c\,\xi +1}{c\,\xi -1}\\
\Rightarrow &&
\Pi (s)=\Psi (s)^2-2=1+\frac{12\, c\,\xi}{(c\,\xi -1)^2}, 
\end{eqnarray*}
where $c$ denotes a constant which at every place might 
be a different one. 
The function $s\mapsto\Pi (s)$ is the \lq\lq degenerate elliptic 
function\rq\rq\  of \cite[p. 38]{costin2}. 
With the substitution 
$\xi =\op{e}^{-s}$, the differential equation (\ref{F0eq}) 
is equivalent to $\op{d}\! F_0/\op{d}\! s= v(0,\, F_0)$. 
As the derivative of $F_0^+(\xi )$ and $F_0^-(\xi )$ at 
$\xi =0$ have to be equal to $0$ and $1$, respectively, 
the derivative of $\Pi =\pi _1=1+p^++p^-$ with respect to $\xi$ 
at $\xi =0$ has to be equal to $1$.  
Therefore $c=1/12$ and 
$\Pi =1+\xi /(\xi /12-1)^2$, which proves 
$\pi _{1,\, 0}(\xi )=P_{1,\, 0}(\xi )/(\xi -12 )^2$ 
with $P_{1,\, 0}(\xi )$ as in (\ref{P10}). 
The formula (\ref{pi2}) with 
$\op{d}\! /\op{d}\! t=\, (-1-1/2t)\,\xi\,\op{d}\! /\op{d}\!\xi$ 
yields 
\[
\pi _{2,\, 0}=\, -\xi\,\op{d}\!\pi _{1,\, 0}/\op{d}\!\xi 
=P_{2,\, 0}(\xi )/(\xi -12)^3, 
\] 
with $P_{2,\, 0}(\xi )$ as in (\ref{P20}). 

Because $\pi _{1,\, 0}$ and $\pi _{2,\, 0}$ are only singular 
at $\xi =12$, it follows from Lemma \ref{Fjlem} and 
Lemma \ref{bdyasymlem} that the 
functions $\pi _{1,\, l}$ and $\pi _{2,\, l}$ have a complex 
analytic continuation along any path in $\C\setminus\{ 12\}$, 
and that the asymptotic expansion (\ref{bdyasymh}) 
holds along these paths. If $\xi$ runs around 
$12$ along a small circle, then the $t_n(\xi )$ with large 
modulus, see Lemma \ref{tnlem}, return to the same value. 
As the Painlev\'e property implies that the function 
$p_C(t)$ is single valued, the asymptotic expansion 
(\ref{bdyasymh}) implies by induction on $l$ that 
$\pi _{1,\, l}$ and $\pi _{2,\, l}$ are single valued 
complex analytic functions on $\C\setminus\{ 12\}$. 

For our system (\ref{upsilon01dot}) the differential 
equations (\ref{Fieq}) for $i\in\Z _{>0}$ take the form 
\begin{eqnarray}
-\xi\,\frac{\op{d}\!\pi _{1,\, i}}{\op{d}\!\xi} 
&=&\pi _{2,\, i}+\frac12\,\xi\,\frac{\op{d}\!\pi _{1,\, i-1}}{\op{d}\!\xi}
+(i-\frac75)\,\pi _{1,\, i-1}
\label{pi1i}\\
-\xi\,\frac{\op{d}\!\pi _{2,\, i}}{\op{d}\!\xi}
&=&\pi _{1,\, 0}\,\pi _{1,\, i}
+\frac12\,\xi\,\frac{\op{d}\!\pi _{2,\, i-1}}{\op{d}\!\xi}
+(i-\frac85)\,\pi _{2,\, i-1}
+\frac12\,\sum_{j=1}^{i-1}\,\pi _{1,\, j}\,\pi _{1,\, i-j}.
\label{pi2i}
\end{eqnarray}
Given $\pi _{1,\, j}$ and $\pi _{2,\, j}$ for $j<i$, 
(\ref{pi1i}), (\ref{pi2i}) is an inhomogeneous linear system 
of first order differential equations for 
$(\pi _{1,\, i},\, \pi _{2,\, i})$, equivalent to the 
inhomogenous linear second order differential equation 
\begin{eqnarray*}
&&\xi ^2\,\frac{\op{d}^2\pi _{1,\, i}}{\op{d}\!\xi ^2}
+\xi\,\frac{\op{d}\!\pi _{1,\, i}}{\op{d}\!\xi}
=-\xi\,\frac{\op{d}}{\op{d}\!\xi}
\left(-\xi\,\frac{\op{d}\!\pi _{1,\, i}}{\op{d}\!\xi}\right)\\
&=&\pi _{1,\, 0}\,\pi _{1,\, i}
+\frac12\,\xi\,\frac{\op{d}\!\pi _{2,\, i-1}}{\op{d}\!\xi}
+(i-\frac85)\,\pi _{2,\, i-1}
+\frac12\,\sum_{j=1}^{i-1}\,\pi _{1,\, j}\,\pi _{1,\, i-j}\\
&&-\xi\,\frac{\op{d}}{\op{d}\!\xi }
\left(\frac12\,\xi\,\frac{\op{d}\!\pi _{1,\, i-1}}{\op{d}\!\xi}
+(i-\frac75)\,\pi _{1,\, i-1}\right)
\end{eqnarray*}
for $\pi _{1,\, i}$, when $\pi _{2,\, i}$ can be solved 
from (\ref{pi1i}) in terms of $\pi _{1,\, i}$ and 
$\pi _{1,\, i-1}$. 

Let $\varphi _1(\xi )$ and $\varphi _2(\xi )$ 
be a basis of solutions of the homogeneous linear  
second order differential equation 
$\pi ''=a(\xi )\,\pi '+b(\xi )\,\pi$.  
Lagrange's method of variations of constants 
yields that the solutions of the inhomogeneous 
equation $\pi ''=a(\xi )\,\pi '+b(\xi )\,\pi +f(\xi )$ 
are of the form 
\begin{equation}
\pi (\xi )=c_1\,\varphi _1(\xi )+c_2\,\varphi _2(\xi )
+\int_{\xi _0}^{\xi}\, 
(-\varphi _1(\xi )\,\varphi _2(\eta )
+\varphi _2(\xi )\,\varphi _1(\eta )\,\frac{f(\eta )}{w(\eta )}
\,\op{d}\!\eta .
\label{inhomdd}
\end{equation}
Here $w=\varphi _1\,\varphi _2'-\varphi _1'\,\varphi _2$ 
is the Wronskian determinant, which satisfies 
$w'=a\, w$, hence 
\[
w(\xi )
=w(\xi _0)\,\op{e}^{\int_{\xi _0}^{\xi}\, a(\eta )\,\op{d}\!\eta}. 
\]
The differential equation $\pi ''=a(\xi )\,\pi '+b(\xi )\,\pi$ 
has a regular singular point at $\xi =\Xi$ if 
$a(\xi )$ and $b(\xi )$ have a pole of order $\leq 1$ and 
$\leq 2$ at $\xi =\Xi$. If $a(\xi )=A\, (\xi -\Xi )^{-1}
+\op{O}(1)$ and $b(\xi )=B\, (\xi -\Xi )^{-2}+\op{O}((\xi -\Xi )^{-1})$ 
as $\xi\to\Xi$, and the indicial equation 
$\lambda\, (\lambda -1)=A\,\lambda +B$ has to distinct 
solutions $\lambda _1$ and $\lambda _2$, 
then there is a basis of solutions 
$\varphi _1(\xi )$ and $\varphi _2(\xi )$ such that 
$\varphi _k(\xi )=(\xi -\Xi )^{\lambda _k}(1+\op{o}(1))$, 
and $w(\xi )=(\lambda _2-\lambda _1)\, 
(\xi -\Xi )^{\lambda _1+\lambda _2-1}$. 
See for instance Coddington and Levinson \cite[Chap. 4, Sec. 8]{cl}. 
Therefore, if $b(\xi )=\op{O}(\xi -\Xi )^b$, then the integral 
in (\ref{inhomdd}) is of order $(\xi -\Xi )^{b+2}$ as $\xi\to\Xi$. 
If $\lambda _1=\lambda _2$, then one has to replace 
$\varphi _2(\xi )=(\xi -\Xi )^{\lambda _2}(1+\op{o}(1))$ 
by $\varphi _2(\xi )=((\xi -\Xi )^{\lambda _1}\,\log 
(\xi -\Xi ))\, (1+\op{o}(1))$, with a corresponding change in the 
estimate for the integral in (\ref{inhomdd}). 
 
At $\xi =12$ we have $A=0$ and $B=(12)^{-2}\, (12)^3=12$, 
when the solutions of the indicial equation are 
$\lambda _1=\, -3$ and $\lambda _2=4$. 
If, for every $0\leq j\leq i-1$, 
$\pi _{1,\, j}$ and $\pi _{2,\, j}$ are meromorphic 
at $\xi =12$ with a pole of order $\leq j+2$ and 
$\leq j+3$, respectively, then an inspection of the 
inhomogenous terms in the second order diffrential equation 
for $\pi _{1,\, i}$ yields $b=\, -i-4$, and because 
$b+2=-i-2\leq -3$ it follows that 
$\pi _{1,\, i}$ has a pole of order $\leq i+2$ at $\xi =12$. 
For $i=1$ we have $b=\, -4$, but then $\lambda _1=\, -3$ 
yields that $\pi _{1,\, 1}$ has a pole of order $\leq 3$. 
Subsequently (\ref{pi1i}) implies that $\pi _{2,\, i}$ 
has a pole of order $\leq i+3$ at $\xi =12$. 
It follows by induction on $l$ that, at $\xi =12$, 
$\pi _{1,\, l}$ and $\pi _{2,\, l}$ have a pole of order $\leq l+2$ and 
$\leq l+3$, respectively.  

At $\xi =\infty$ we have $A=\, -1$ and $B=1$, when the solutions of 
the indicial equation are $\lambda _1=1$ and $\lambda _2=\, -1$. 
If, for every $0\leq j\leq i-1$, 
$\pi _{1,\, j}$ and $\pi _{2,\, j}$ are meromorphic 
at $\xi =\infty$ with exponents $\leq j$, then an inspection of the 
inhomogenous terms in the second order diffrential equation 
for $\pi _{1,\, i}$ yields $b=i-2$, and because 
$b+2=i\geq 1$ it follows that at $\xi =\infty$ the function  
$\pi _{1,\, i}$ has an exponent $\leq i$. 
Subsequently (\ref{pi1i}) implies that also 
$\pi _{2,\, i}$ has an exponent $\leq i$ at $\xi =\infty$. 
It follows by induction on $l$ that the functions 
$\pi _{1,\, l}$ and $\pi _{2,\, l}$ have exponents $\leq l$ 
at $\xi =\infty$, and therefore are rational functions   
of the form   $\pi _{1,\, l}(\xi )=(\xi -12)^{-l-2}\, P_{1,\, l}(\xi )$ and 
$\pi _{2,\, l}(\xi )=(\xi -12)^{-l-3}\, P_{2,\, l}(\xi )$, 
where $P_{1,\, l}$ and $P_{2,\, l}$ are polynomials  
of degree $\leq 2\, l+2$ and $\leq 2\, l+3$, respectively. 

The functions $\pi _{1,\, 1}(\xi )=(\xi -12)^{-3}\, P_{1,\, 1}(\xi )$ and 
$\pi _{2,\, 1}(\xi )=(\xi -12)^{-4}\, P_{2,\, 1}(\xi )$, with 
the respective polynomials 
$P_{1,\, 1}(\xi )$ and $P_{2,\, 1}(\xi )$ as in 
(\ref{P11}) and (\ref{P21}), 
have been found with the help of a formula manipulation 
computer program, in the following way. 
The solutions of the system (\ref{pi1i}), (\ref{pi2i}) 
for $i=1$ which are complex analytic in a neighborhood 
of $\xi =0$ are of the form 
\begin{eqnarray*}
\pi _{1,\, 1}(\xi )&=&\xi\, ((720\, c-84672)
+(60\, c+4464)\,\xi +180\,\xi ^2-\xi ^3)/(60\, (\xi -12)^3),\\
\pi _{2,\, 1}(\xi )&=&(497664+(8640\, c-1306368)\,\xi 
+(2880\, c-124416)\,\xi ^2+(60\, c+17568)\,\xi ^3\\
&&-24\,\xi ^4+\xi ^5)/
(60\, (\xi -12)^4),
\end{eqnarray*}
where $c$ is free constant. With these functions 
$\pi _{1,\, 1}$ and $\pi _{2,\, 1}$, an investigation 
of the explicit solutions $\pi _{1,\, 2}$, $\pi _{2,\, 2}$ of 
the system (\ref{pi1i}), (\ref{pi2i}) for $i=2$ 
yields that there exist solutions which are complex analytic 
in a neighborhood of $\xi =0$ if and only if $c=678/5$. 
Therefore $\pi _{1,\, 1}(\xi )=(\xi -12)^{-3}\, P_{1,\, 1}(\xi )$ and 
$\pi _{2,\, 1}(\xi )=(\xi -12)^{-4}\, P_{2,\, 1}(\xi )$ with 
$P_{1,\, 1}(\xi )$ and $P_{2,\, 1}(\xi )$ as in 
(\ref{P11}) and (\ref{P21}), respectively. 
\end{proof}

As the system (\ref{upsilon01dot}) is just a rescaled version 
of (\ref{u01dot}), passing to the complex projective plane and 
successively blowing up the base points of the vector fields, 
as in Section \ref{boutrouxsec}, 
leads to surface $S_9$, with a locus $I$ where the vector field 
is infinite equal to the union of nine complex projective lines 
$L_i^{(9-i)}$, $0\leq i\leq 8$, and a pole line 
$L_9\setminus I$, where $L_9$ is the complex projective line 
appearing at the last blowup. In the common domain 
of definition of the coordinate systems $(\pi _{ij1},\,\pi _{ij2})$ 
and $(\pi _1,\,\pi _2)$, an application of the birational transformation 
$(\pi _1,\,\pi _2)\mapsto (\pi _{ij1},\,\pi _{ij2})$ to the asymptotic 
expansion (\ref{bdyasymh}) leads to an asymptotic expansion 
\begin{equation}
\pi _{ijk}(t)=\sum_{l=0}^{m-1}\, t^{-l}\, \pi _{ijk,\, l}(\tau (t))
+\op{O}(t^{-m})
\label{bdyasymij}
\end{equation}
as $t\in V$, $|t|\to\infty$, where the functions $\pi _{ijk,\, l}(\tau )$ are 
rational expressions in the functions 
$\pi _{1, l'}(\tau )$, $\pi _{2,\, l'}(\tau )$ for 
$0\leq l'\leq l$, and therefore are rational functions of $\tau$. 
Because the differential equations for $\pi _{ijk,\, 0}(\tau )$, 
analogous to (\ref{F0eq}), correspond to the autonomous 
limit system, these differential equations are regular 
and its solutions have a complex analytic extension as 
long as they remain in in the complement of 
the inifinity set $I$ in the coordinate system 
$(\pi _{ij1},\, \pi _{ij2})$. Because also the non\--autonomous 
vector field is regular in $S\setminus I$, the functions 
$\pi _{ijk,\, l}(\tau )$ with $l\in\Z _{>0}$ satisfy   
inhomogeneous linear differential equations as 
(\ref{Fieq}), variational equations of the differential equations 
for $\pi _{ijk,\, 0}(\tau )$, where the inhomogeneous term is 
a regular expression in the $\pi _{ijk,\, m}(\tau )$ with $m<l$.  
It follows by induction on $l$ that all the rational functions 
$\pi _{ijk,\, l}(\tau )$ are regular, when the 
perturbation argument in the last paragraph 
of the proof of Lemma \ref{bdyasymlem} yields that the 
asymptotic expansion (\ref{bdyasymij}) extends to the whole 
complement of $I$ in the coordinate system $(\pi _{ij1},\,\pi _{ij2})$. 

The pole line is visible in the coordinate system  
$(\pi _{911},\,\pi _{912})$ as the line $\pi _{912}=0$, 
and therefore the poles are the solutions $T$ of 
the equation $\pi _{912}(T)=0$, where 
$\pi _{912}=\pi _1/\pi _2$. 
This leads to the following asymptotic results for the 
poles, where in view of (\ref{xt}) the poles of the corresponding 
truncated solution 
of (\ref{PI}) are given by $X_n=\, -2^{-3/5}\, 3^{-1/5}\, 
(5\, T_n/4)^{4/5}$. 
\begin{lemma}
There exist universal 
sequences of coefficients $c_j$, $d_{k,\, l}$
$j,\, k,\, l\in\Z _{\geq 0}$, 
with $c_0=12$, $c_1=109/10$, and $d_{0,\, 0}=0$, 
such that following holds. 
Let $C\neq 0$, and let $p_C(t)$ be the solution 
in Lemma \ref{bdyasymlem}, of 
the system (\ref{cssystem}) obtained from (\ref{upsilon01dot}) 
by means of the substitutions $\pi _1=1+p^++p^-$, 
$\pi _2=p^+-p^-$. Let $(\pi _1(t),\,\pi _2(t))$ 
be the corresponding solution of (\ref{upsilon01dot}). 
Then there is a sequence of 
poles $T_n$, $n\in\Z$, $n>>0$ of $\pi _1(t)$, such that 
$T_n=2\pi\op{i}\, n\, (1+\op{o}(1))$ as $n\to\infty$, and 
\begin{equation}
\tau (T_n):=\op{e}^{-T_n}\, {T_n}^{-1/2}
\sim \frac1C\,\sum_{j=0}^{\infty}\, c_j\, {T_n}^{-j}
\quad\mbox{\rm as}\quad n\to\infty .
\label{asymTn}
\end{equation}
Furthermore, with the notations 
\[
u:=\frac{1}{2\pi\op{i}n}, 
\quad v:=\frac{\log (2\pi\op{i}n)}{2\pi\op{i}n},
\quad\mbox{\rm and}\quad 
W:=u\,\log\frac{C}{12}-\frac{v}{2}, 
\]
where $\log (2\pi\op{i}n)=\log (2\pi\, n)+\pi\op{i}/2$, 
we have the more explicit but more complicated asymptotic 
expansion 
\begin{equation}
T_n\sim 2\pi\op{i}n-\frac12\,\log (2\pi\op{i}n)
+\log\frac{C}{12}+\sum_{k,\, l\geq 0}\, d_{k,\, l}\, u^k\, W^l
\quad\mbox{\rm as}\quad n\to\infty , 
\label{asymTn2}
\end{equation}
of which the leading terms yield  
\begin{eqnarray}
T_n&=&2\pi\op{i}n-\frac12\,\log (2\pi\op{i}n)
+\log\frac{C}{12}
+\frac{v}{4}-(\frac12\,\log\frac{C}{12}+\frac{109}{120})\, u
\nonumber\\
&&+\frac{1}{16}\, v^2-(\frac14\,\log\frac{C}{12}
+\frac{139}{240})\, u\, v
+\op{O}(n^{-2})
\label{asymTn3}
\end{eqnarray}
as $n\to\infty$. 
If $C\in\C\setminus\{ 0\}$ runs once around 
the origin in the positive direction, then $T_n$ moves 
continuously to $T_{n+1}$. 

Conversely, for every $\eta >0$ there exists an $r>0$, such that 
these $T_n$ are the only poles $T$ of $\pi _1(t)$ 
such that $|T|\geq r$, $\op{Im}T\geq 0$, and 
$|\op{e}^{-T}\, T^{-1/2}|\leq\eta$.
\label{polelem}
\end{lemma}
\begin{proof}
It follows from Lemma \ref{hlem} that the poles 
$T$ of $\pi _1(t)$ with bounded 
$\Xi (T):=C\,\op{e}^{-T}\, T^{-1/2}$ satisfy 
$\Xi (T)\to 12$ as $|T|\to\infty$. 

There exists a sequence of rational functions 
$\pi _{912,\, l}$, $l\in\Z _{\geq 0}$, such that, 
with the notation $\xi =C\,\tau =C\,\op{e}^{-t}\, t^{-1/2}$, 
\[
\pi _{912}(t)=\frac{\pi _1(t)}{\pi _2(t)}
=\sum_{l=0}^{m-1}\, t^{-l}\,\pi _{912,\, l}(\xi )+\op{O}(t^m),
\]
for every $m\in\Z _{>0}$. Lemma \ref{hlem} implies that 
\begin{eqnarray*}
\pi _{912,\, 0}(\xi )&=&(\xi -12)\, (144+120\,\xi +\xi ^2)
/(144\,\xi (\xi +12))\quad\mbox{\rm and}\\
\pi _{912,\, 1}(\xi )&=&
(-71663616-40310784\,\xi-248832\,\xi ^2-11860992\,\xi ^3
-1221696\,\xi ^4\\
&&+1224\,\xi ^5-240\,\xi ^6-\xi ^7)
/((1244160\,\xi ^2\, (\xi +12)^2).
\end{eqnarray*}
Because $\pi _{912,\, 0}(12)=0$ and $\pi _{912,\, 0}'(12)
=1/24\neq 0$, an application of the implicit function 
theorem yields that there exist open neighbourhoods 
$A$ and $B$ of $12$ and $0$, respectively, such that 
for every $t^{-1}\in A$ the equation 
\[
\sum_{l=0}^{m-1}\, t^{-l}\,\pi _{912,\, l}(\xi )=0
\]
has a unique solution 
$\Xi _m=\Xi _m(t^{-1})\in B$, which moreover depends in a 
complex analytic fashion on $t^{-1}$, and satisfies 
$\Xi _m(0)=12$. 
Furthermore, $\Xi (T)=\Xi _m(T^{-1})+\op{O}(T^{-m})$, and   
as the left hand side does not depend on $m$, it follows that 
the coefficients $c_j$ for $0\leq j\leq m-1$ in the Taylor 
expansion of the function $\Xi _m$ at the origin 
do not depend on $m$. Because this holds for every 
$m\in\Z _{>0}$, it follows that there is a sequence 
of complex numbers $c_j$, $j\in\Z _{>0}$, such that 
the poles $T$ of $\pi _1(t)$ with bounded 
$\Xi (T):=C\,\op{e}^{-T}\, T^{-1/2}$ satisfy 
\[
C\,\tau (T)=C\,\op{e}^{-T}\, T^{-1/2}\sim 12+\sum_{j>0}\, c_j\, T^{-j}
\quad\mbox{\rm as}\quad |T|\to\infty . 
\]
Because $\pi _{912,\, 1}(12)=\, -109/240$, 
we have $c_1=\, -\pi _{912,\, 1}(12)/\pi _{912,\, 0}'(12)
=109/10$. This completes the proof of (\ref{asymTn}). 

The proof of (\ref{asymTn2}) is analogous to the proof of 
(\ref{tn}). For any $m\in\Z _{>0}$, 
(\ref{asymTn}) yields that 
\[
\tau =\tau (T_n)=\frac{12}{C}\, 
(1+\frac{1}{12}\sum_{j=1}^{m-1}\, c_j\, {T_n}^{-j}+r),
\]
where $r=\op{O}({T_n}^{-m})=\op{O}(n^{-m})$, 
hence 
\[
\log\tau =\log\frac{12}{C}
+\log (1+\sum_{j=1}^{m-1}\,\frac{c_j}{12}\, {T_n}^{-j}+r).
\]
Upon the substitution 
\begin{eqnarray*}
T_n&=&2\pi\op{i}n-\frac12\,\log (2\pi\op{i}n)+\log\frac{C}{12}+S\\
&=&2\pi\op{i}n\, (1-\frac12\, v+u\log (C/12)+u\, S)
=2\pi\op{i}n\, (1+W+u\, S),
\end{eqnarray*}
which implies 
\[
\log T_n=\log (2\pi\op{i}n)+\log (1+W+u\, S)
\]
and
\[ 
{T_n}^{-1}=u\, (1+W+u\, S)^{-1},
\] 
the equation 
$T_n=2\pi\op{i}n-(1/2)\,\log T_n-\log\tau$ 
is equivalent to the equation 
\[
S=F(u,\, W,\, r,\, S):=
-\frac12\,\log (1+W+u\, S)
-\log (1+\sum_{j=0}^{m-1}\, \frac{c_j}{12}\, u^j
\, (1+W+u\, S)^{-j}+r).
\]
Because $F(0,\, 0,\, 0,\, S)\equiv 0$, it follows from the 
implicit function theorem in the complex analytic setting 
that there exist open neighborhoods $U$, ${\mathcal W}$, $R$, 
and ${\mathcal S}$ of the origin in $\C$ such that for every 
$(u,\, W,\, r)\in U\times {\mathcal W}\times R$ the equation 
has a unique solution $S=S_m(u,\, W,\, r)\in {\mathcal S}$, 
and that $S_m$ is a complex analytic function on 
$U\times {\mathcal W}\times R$. In our setting 
\[
S_m(u,\, W,\, r)=S_m(u,\, W,\, 0)+\op{O}(r)
=\sum_{k=0}^{m-1}\,\sum_{l=0}^{m-k}\, d_{k,\, l}\, u^k\, W^l
+\op{O}(n^{-m}). 
\]
Here the coefficients $d_{k,\, l}$ do not depend on $m$ 
because $S$ does not depend on $m$. 
This completes the proof of (\ref{asymTn2}). 

The equation for $m=2$ yields 
\[
S=\, -\frac12\, (W+u\, S)+\frac14\, (W^2+2\, W\, u\, S)
-\frac{c_1}{12}\, u\, (1-W)+\op{O}(n^{-2}),
\]
hence 
\[
S=\, -\frac12\, W-\frac{c_1}{12}\, u
+\frac14\, W^2+(\frac14+\frac{c_1}{12})\, W\, u+\op{O}(n^{-2}),
\]
which implies (\ref{asymTn3}). 

If $C\in\C\setminus\{ 0\}$ runs once around the origin 
in the positive direction, then (\ref{asymTn}) implies that 
$\tau (T_n)$ runs once around the origin in the negative 
direction, when Lemma \ref{tnlem} implies that $T_n$ 
moves continuously to $T_{n+1}$.  
\end{proof}

Figure \ref{polefig} illustrates the asymptotic approximations  
of the poles in (\ref{asymTn3}) in the complex 
$t$\--plane, of Boutroux's triply truncated solution 
$p_{\downarrow}(t)$, when, according to Remark \ref{stokesrem}, 
$C=\op{i}\,\sqrt{6/5\,\pi}$. Shown are the points in the right 
hand side of (\ref{asymTn3}) without the remainder 
term $\op{O}(n^{-2})$, for 
$1\leq n\leq 20$. For clarity of the picture, 
the imaginary part has been multiplied by $1/24$ 
in comparison to the real part. It would be interesting 
to compare the approximate poles in Figure \ref{polefig} 
with the numerical approximations of the 
actual poles of $p_{\downarrow}(t)$. 

\begin{figure}[ht]
\centering
\includegraphics[width=12cm]{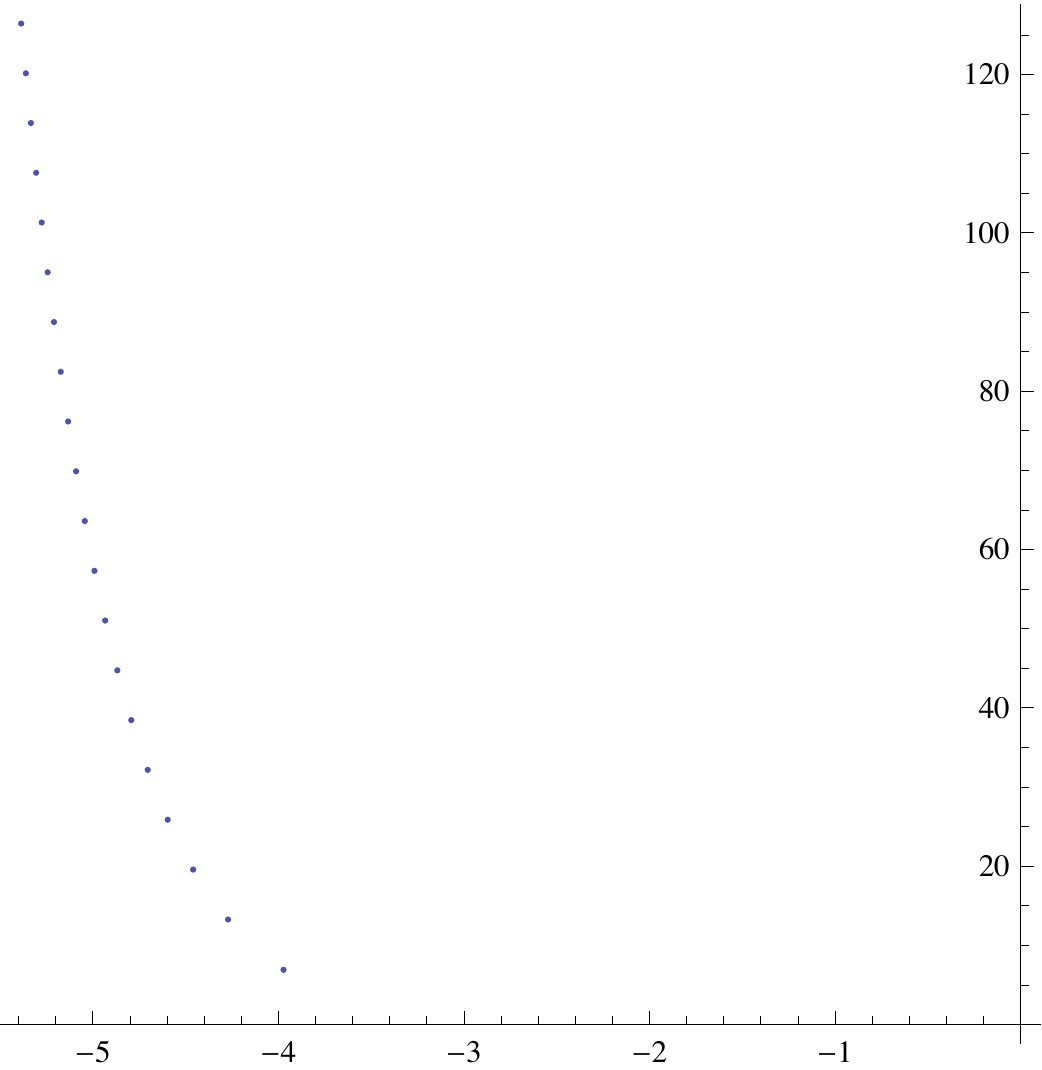}
\caption{The first twenty of the asymptotic approximations 
in (\ref{asymTn3}) of the poles of $p_{\downarrow}(t)$}  
\label{polefig}
\end{figure}

\begin{remark}
The intriguing \lq\lq General comments 2\rq\rq\  of \cite[p. 40]{costin2} say: 

\medskip\noindent 
\lq\lq The pole structure can be explored beyond the first array, 
in much of the same way: For large $\xi$ induction shows that 
$H_n\sim\op{Const}\,\xi ^n$, suggesting a reexpansion for large 
$\xi$ in the form 
\[
h\sim\sum_{k=0}^{\infty}\,\frac{H_k^{[1]}(\xi _2)}{x^k};
\quad\xi _2=C^{[1]}\,\xi\, x^{-1}=C\, C^{[1]}\, x^{-3/2}\,\op{e}^{-x}.
\quad\quad\quad (118)
\]
By the same technique it can be shown that (118) holds and, 
by matching with 
\[
h\sim\sum_{k=0}^{\infty}\, x^{-k}\, H_k(\xi (x))
\quad\quad\quad\quad\quad\quad\quad\quad\quad\quad\quad
\quad\quad\quad\quad\quad\quad (113)
\]
at $\xi _2\sim x^{-2/3}$, we get $H_0^{[1]}=H_0$ 
with $C^{[1]}=\, -1/60$. Hence, if $x_s$ belongs to the 
first line of poles, i.e. 
\[
\xi (x_s)=\xi _s=12+\frac{109}{10\, x}+\op{O}(x^{-2}),
\quad\quad\quad\quad\quad\quad\quad\quad\quad\quad\quad\quad (116) 
\]
the second line of poles is given by the condition
\[
x_1^{-3/2}\,\op{e}^{-x_1}=\, -60\cdot 12\, c
\]
i.e., it is situated at a logarithmic distance from the first one: 
\[
x_1-x_s=\, -\op{ln}x_s+(2\, n+1)\,\pi\op{i}-\op{ln}(60)+\op{o}(1).
\]
Similarly one finds $x_{s,\, 3}$ and in general 
$x_{s,\, n}$. The second scale for the $n$\--array 
is $x^{-n-1/2}\,\op{e}^{-x}$. 

The expansion (113) can however matched directly to an 
{\em adiabatic invariant}\--like expansion valid 
throughout the sector where $h$ has poles, similar to the one in 
Joshi and Kruskal \cite{jk92}. In this language, the successive 
expansions of the form (118) pertain to the separatrix crossing 
region. We will not pursue this issue here.\rq\rq\  

\medskip\noindent

The word \lq\lq suggesting\rq\rq\ preceding (118) indicates that 
(118) is a conjecture, but in the sequel all the statements, 
including (118), are treated as facts, with only some hints
of proofs. The phrase \lq\lq matching with (113) at $\xi _2\sim x^{-2/3}$\rq\rq\ 
suggests that $H_n\sim\op{Const}_n\,\xi ^n$ implies that the 
expansion (113) extends to domains where 
$\xi (x)$ is of order $x^{1/3}$, 
thus allowing a matching with 
(118) for $\xi _2(x)=C^{[1]}\,\xi (x)\, x^{-1}$ 
of order $x^{-2/3}$. 

Write $\tau ^{[N]}(t):=\tau (t)\, t^{-N}=
\op{e}^{-t}\, t^{-N-1/2}$, the second scale for 
the $(N+1)$\--st array of poles. 
Because $\pi _{k,\, l}(\xi )\sim\op{Const}_{k,\, l}\,\xi ^l$ 
for $\xi\to\infty$, see Lemma \ref{hlem}, the 
$t^{-l}\,\pi _{k,\, l}(C\,\tau (t)$, $l\in\Z _{\geq 0}$ 
form an asymptotic sequence 
for $|t|\to\infty$, $|\tau (t)|=\op{o}(|t|)$ and $\tau (t)$ 
bounded away from $12/C$, in the sense that 
$t^{-l}\,\pi _{k,\, l}(C\,\tau (t)=\op{O}(\tau ^{[1]}(t)^l)$, 
where $\tau ^{[1]}(t)\to 0$. A stronger conjecture would be 
that (\ref{bdyasymh}) extends as an asymptotic expansion 
for $|t|\to\infty$ in the aforementioned domain. 
Because (\ref{P10}) and (\ref{P20}) imply that 
$(\pi _{1,\, 0}(\xi ),\,\pi _{2,\, 0}(\xi ))\to 
(1,\, 0)$ as $\xi\to\infty$, it would follow that 
$(\pi _1(t),\,\pi _2(t))$ converges to the equilibrium 
point $(1,\, 0)$ of the autonomous limit system if 
$|t|\to\infty$, $|\tau (t)|\to\infty$, and 
$|\tau (t)|=\op{o}(|t|)$. 
In view of (\ref{P10}), (\ref{P20}), (\ref{P11}), and 
(\ref{P21}), the first two terms of the extended asymptotic expansion 
yield 
\begin{eqnarray*}
\pi _1(t)&=&1+144\, (C\,\tau (t))^{-1}-C\, \tau ^{[1]}(t)/60
+\op{o}(\tau (t)^{-1})+\op{o}(\tau ^{[1]}(t)),\\
\pi _2(t)&=&144\, (C\,\tau (t))^{-1}+C\,\tau ^{[1]}(t)/60
+\op{o}(\tau (t)^{-1})+\op{o}(\tau ^{[1]}(t)).
\end{eqnarray*}
Therefore, if we restrict to $\tau (t)^{-1}=\op{o}(\tau ^{[1]}(t))
=\op{o}(\tau (t)/t)$, that is 
$|\tau (t)|/|t|^{1/2}\to\infty$, then 
\[
(\pi _1(t),\,\pi _2(t))=
(1,\, 0)+\xi ^{[1]}\, (1,\, -1)
+\op{o}((\xi ^{[1]}))\quad\mbox{\rm if}\quad \xi ^{[1]}(t)=\, 
-C\, \tau ^{[1]}(t)/60,
\]
the leading term of an asymptotic expansion 
\[
(\pi _1(t),\, \pi _2(t)\sim\sum _{l\geq 0}\, t^{-l}\, 
(\pi _{1,\, l}^{[1]}(\xi ^{[1]}(t)),
\,\pi _{2,\, l}^{[1]}(\xi ^{[1]}(t))).
\]
As in Lemma \ref{Fjlem}, the function 
$\xi ^{[1]}\mapsto (\pi _{1,\, 0}^{[1]}(\xi ^{[1]}),\, 
\pi _{2,\, 0}^{[1]}(\xi ^{[1]}))$ satisfies the 
same differential equation 
(\ref{F0eq}) as the function 
$\xi\mapsto (\pi _{1,\, 0}(\xi ),\, \pi _{2,\, 0}(\xi ))$, 
where both functions have the same value and first order derivative 
at the origin. Therefore 
$\pi _{k,\, 0}^{[1]}=\pi _{k,\, 0}$, 
with $\pi _{1,\, 0}$ and $\pi _{2,\, 0}$ 
as in Lemma \ref{Fjlem}. It follows, if the 
aforementioned statements about the 
asymptotic expansions hold, that the second sequence of poles 
occurs at points $2\pi\op{i}\, n 
-(3/2)\,\log (2\pi\op{i}\, n)-\log (-720/C)+\op{o}(1)$, 
equal to the first sequence 
$2\pi\op{i}\, n 
-(1/2)\,\log (2\pi\op{i}\, n)-\log (12/C)+\op{o}(1)$ 
plus $-\log (2\pi\op{i}n)-\log (-60)+\op{o}(1)$ as $n\to\infty$. 

The text in \lq\lq General comments 2\rq\rq  of \cite[p. 40]{costin2} 
continues with the statement that for each $N$ there 
is an asymptotic expansion of the form 
\[
(\pi _1(t),\,\pi _2(t))\sim\sum_{l=0}^{\infty}
\, t^{-l}\, (\pi _{1,\, l}^{[N]}(C^{[N]}\,\tau ^{[N]}(t)),
\, \pi _{2,\, l}^{[N]}(C^{[N]}\,\tau ^{[N]}(t))), 
\]
valid, if interpreted in the strong sense, for 
$|t|\to\infty$, $|\tau ^{[N]}(t)|=\op{o}(t)$, 
and $|\tau ^{[N-1]}(t)|\to\infty$, 
where the constant $C^{[N]}$ depends linearly on $C$. 
This would lead to an asymptotic description of the 
$(N+1)$\--st sequence of poles, equal to the $N$\--th sequence 
plus $\, -\log n+\gamma _N+\op{o}(1)$ as $n\to\infty$, 
where the constant $\gamma _N$ neither depends on 
$n$ nor on $C$. For $|t|\to\infty$ and $t$ between 
the $N$\-th and the $(N+1)$\--st sequence of poles, 
the solution $(\pi _1(t),\,\pi _2(t))$ converges to 
the equilibrium point $(1,\, 0)$ of the autonomous 
system, if and only the distance from $t$ to both 
sequences of poles tends to infinity. 
Furthermore, for every $M>0$ we have that 
the energy $E={\pi _2}^2/2-{\pi _1}^3/6+\pi_1/2$ 
converges to the critical level $1/3$, meaning that the 
solution converges to the solution of the autonomous 
limit system at the critical energy level, if 
$\op{Im}t\to\infty$, $\op{Re}t\geq -M\,\log (\op{Im}t)$, 
and $|t|$ times the distance from $t$ to the poles tends to infinity. 
The latter condition is related to the description 
of the energy near the poles in (\ref{qexpansion}). 

We would like to prove statements like those in the 
second paragraph in the \lq\lq General comments 2\rq\rq\  
of \cite[p. 40]{costin2} 
by means of the averaging method. 
This is not a trivial matter, as 
all the asymptotic expansions up till now 
are near one of the critical values of the energy function, 
where solutions of the averaged differential equation 
for the energy function are not uniquely determined. 
One might expect that the energy function 
acquires different limit values from the critical 
value $1/3$ of the energy function, if $t$ runs to infinity 
in the direction of $\op{e}^{\,\scriptop{i}\,\theta}$  
with $\pi/2<\theta <3\,\pi /2$. On the other hand   
the truncated solution converges to the equilibrium point 
of the autonomous limit system (with energy equal to 
the critical value $1/3$) if $\, -\pi/2<\theta <\pi/2$. 
\label{nextpolesrem}
\end{remark}

\appendix
\section{Okamoto's Space}
We construct Okamoto's space 
of initial conditions \cite{okamoto79}  
in the Boutroux rescaling. (See also \cite{d} for 
the original Painlev\'e equation (\ref{PI}).) Recall that the notation $(u_{ij1}, u_{ij2})$ will be used to denote
the coordinates in the $j$-th chart of the $i$-th blowup and that
in each coordinate chart, the Jacobian of the coordinate change from $(u_1, u_2)$ to 
$(u_{ij1}, u_{ij2})$ will be denoted by
\begin{equation*}
 w_{ij}=\frac{\partial u_{ij1}}{\partial u_1}\,\frac{\partial u_{ij2}}{\partial u_2}-\frac{\partial u_{ij1}}{\partial u_2}\,\frac{\partial u_{ij2}}{\partial u_1}.
\end{equation*}

Up to and including the seventh 
blowup, the function $z\, \dot{E}$ is rational with 
$w_{ij}$ in the denominator, and we have 
added the formula for $\dot{E}\, w_{ij}$ in 
each coordinate chart. 
\subsection{Embedding into $\Proj^{2}$}
Recall the embedding of $(u_{1}, u_{2})$ into projective space given in Section \ref{notation}. We have the second affine chart in $\Proj ^2$: 
\begin{eqnarray*}
u_{021}&=&{u_1}^{-1},\\ 
u_{022}&=&{u_1}^{-1}\, u_2,\\
u_1&=&{u_{021}}^{-1},\\
u_2&=&{u_{021}}^{-1}\, u_{022},\\ 
\dot{u}_{021}&=&u_{021}\, (-u_{022}
+2\, (5\, z)^{-1}),\\
\dot{u}_{022}&=&{u_{021}}^{-1}
\, (6+{u_{021}}^2-u_{021}\, {u_{022}}^2
-(5\, z)^{-1}\, u_{021}\, u_{022}),\\
w_{02}&=&-{u_{021}}^3,\\
\dot{w}_{02}&=&3\, {u_{021}}^3\, (u_{022}-2\, (5\, z)^{-1}),\\
E\, w_{02}&=&2+{u_{021}}^2-2^{-1}\, u_{021}\, {u_{022}}^2,\\
\dot{E}\, w_{02}&=&-(5\, z)^{-1}
\, (12+2\, {u_{021}}^2-3\, u_{021}\, {u_{022}}^2).
\end{eqnarray*}
The line at infinity $L_0$ corresponds to $u_{021}=0$. 
In this chart 
there are no base points for the Painlev\'e vector field 
or the anticanonical pencil.  

Third affine chart in $\Proj ^2$: 
\begin{eqnarray*}
u_{031}&=&{u_2}^{-1},\\ 
u_{032}&=&u_1\, {u_2}^{-1},\\ 
u_1&=&{u_{031}}^{-1}\, u_{032},\\
u_2&=&{u_{031}}^{-1},\\ 
\dot{u}_{031}&=&-{u_{031}}^2-6\, {u_{032}}^2
+3\, (5\, z)^{-1}\, u_{031},\\
\dot{u}_{032}&=&{u_{031}}^{-1}\, 
(u_{031}-{u_{031}}^2\, u_{032}-6\, {u_{032}}^3
+(5\, z)^{-1}\, u_{031}\, u_{032}),\\
w_{03}&=&{u_{031}}^3,\\
\left[ w_{03}\, {u_{032}}^{-3}\right] ^{\bullet}&=&
3\, {u_{031}}^3\, (-1+2\, (5\, z)^{-1}\, u_{032})
\, {u_{032}}^{-4},\\
E\, w_{03}&=&2^{-1}\, u_{031}-{u_{031}}^2\, u_{032}
-2\, {u_{032}}^3,\\
\dot{E}\, w_{03}&=&(5\, z)^{-1}
\, (-3 u_{031}+2\, {u_{031}}^2\, u_{032}+12\, {u_{032}}^3).
\end{eqnarray*}
The line at infinity $L_0$ corresponds to 
$u_{031}=0$. Both the 
The Painlev\'e vector field and the anticanonical 
pencil both have a base point $b_0$ given by $u_{031}=0$, $u_{032}=0$. 

\subsection{Resolution of the flow at $b_{0}$}
Blowing up $\Proj ^2$ at $b_0$ leads to $S_1$. 

First coordinate chart: 
\begin{eqnarray*}
u_{031}&=&u_{111}\, u_{032},\\
u_{032}&=&u_{112},\\
u_{111}&=&{u_1}^{-1},\\
u_{112}&=&u_1\, {u_2}^{-1},\\ 
u_1&=&{u_{111}}^{-1},\\ 
u_2&=&{u_{111}}^{-1}\, {u_{112}}^{-1},\\
\dot{u}_{111}&=&{u_{112}}^{-1}\, u_{111}
\, (-1+2\, (5\, z)^{-1}\, u_{112}),\\
\dot{u}_{112}&=&{u_{111}}^{-1}
\, (u_{111}-6\, {u_{112}}^2-{u_{111}}^2\, {u_{112}}^2
+(5\, z)^{-1}\, u_{111}\, u_{112}),\\
w_{11}&=&{u_{111}}^3\, {u_{112}}^2,\\
\left[ w_{11}\, {u_{112}}^{-2}\right] ^{\bullet}&=&
3\, {u_{111}}^3
\, (-1+2\, (5\, z)^{-1}\, u_{112})\, {u_{112}}^{-1},\\
\left[ w_{11}\, {u_{111}}^{-1}\right] ^{\bullet}&=&
2\, u_{111}\, {u_{112}}^2\, 
(-6\, u_{112}-{u_{111}}^2\, u_{112}+3\, (5\, z)^{-1}\, u_{111}),\\
E\, w_{11}&=&2^{-1}\, u_{111}-2\, {u_{112}}^2
-{u_{111}}^2\, {u_{112}}^2,\\
\dot{E}\, w_{11}&=&(5\, z)^{-1}\, 
(-3\, u_{111}+12\, {u_{112}}^2+2\, {u_{111}}^2\, {u_{112}}^2).
\end{eqnarray*}
Then $u_{112}=0$ defines $L_1$ and $u_{111}=0$ defines 
$L_0^{(1)}$. 
The Painlev\'e vector field and the anticanonical pencil 
both have a base point $b_1$ given by $u_{111}=0$, $u_{112}=0$.

The second coordinate chart after the first blowup 
is defined by 
\begin{eqnarray*}
u_{031}&=&u_{121},\\
u_{032}&=&u_{122}\, u_{031},\\
u_{121}&=&{u_2}^{-1}=u_{111}\, u_{112},\\
u_{122}&=&u_1={u_{111}}^{-1},\\
u_1&=&u_{122},\\
u_2&=&{u_{121}}^{-1},\\
\dot{u}_{121}&=&
u_{121}\, (-u_{121}-6\, u_{121}\, {u_{122}}^2
+3\, (5\, z)^{-1}),\\
\dot{u}_{122}&=&{u_{121}}^{-1}
\, (1-2\, (5\, z)^{-1}\, u_{121}\, u_{122}),\\
w_{12}&=&{u_{121}}^2,\\
\dot{w}_{12}&=&2\, {u_{121}}^2
\, (-u_{121}-6\, u_{121}\, {u_{122}}^2+3\, (5\, z)^{-1}),\\
E\, w_{12}&=&2^{-1}-{u_{121}}^2\, u_{122}-2\, {u_{121}}^2\, {u_{122}}^3,\\
\dot{E}\, w_{12}&=&(5\, z)^{-1}\, 
(-3+2\, {u_{121}}^2\, u_{122}+12\, {u_{121}}^2\, {u_{122}}^3). 
\end{eqnarray*}
The equation $u_{121}=0$ defines $L_1$. 
The line $L_0^{(1)}$ is not visible, 
and there are no base points in this chart. 

\subsection{Resolution of the flow at $b_{1}$}
Blowing up $S_1$ at $b_1$ leads to $S_2$. 
First coordinate chart: 
\begin{eqnarray*}
u_{111}&=&u_{211}\, u_{112},\\
u_{112}&=&u_{212},\\
u_{211}&=&{u_1}^{-2}\, u_2,\\
u_{212}&=&{u_2}^{-1}\,  u_1,\\ 
u_1&=&{u_{211}}^{-1}\, {u_{212}}^{-1},\\ 
u_2&=&{u_{211}}^{-1}\, {u_{212}}^{-2},\\
\dot{u}_{211}&=&{u_{212}}^{-1}
\, (-2\, u_{211}+6\, u_{212}
+{u_{211}}^2\, {u_{212}}^3
+(5\, z)^{-1}\, u_{211}\, u_{212}),\\
\dot{u}_{212}&=&{u_{211}}^{-1}
\,  (u_{211}-6\, u_{212}
-{u_{211}}^2\, {u_{212}}^3
+(5\, z)^{-1}\, u_{211}\, u_{212}),\\
w_{21}&=&{u_{211}}^3\, {u_{212}}^4,\\
\left[ w_{21}\, {u_{212}}^{-1}\right] ^{\bullet}&=&
3\, {u_{211}}^3\, {u_{212}}^2
\, (-1+2\, (5\, z)^{-1}\, u_{212}),\\
\left[ w_{21}\, {u_{211}}^{-1}\right] ^{\bullet}&=&
2\, u_{211}\, {u_{212}}^4\, 
(-6-{u_{211}}^2\, {u_{212}}^2+3\, (5\, z)^{-1}\, u_{211}),\\
E\, w_{21}&=&2^{-1}\, u_{211}-2\, u_{212}-{u_{211}}^2\, {u_{212}}^3,\\
\dot{E}\, w_{21}&=&(5\, z)^{-1}\, 
(-3\, u_{211}+12\, u_{212}+2\, {u_{211}}^2\, {u_{212}}^3).
\end{eqnarray*}
Then $u_{212}=0$ defines $L_2$ and $u_{211}=0$ defines 
the proper transform $L_0^{(2)}$ of 
$L_0^{(1)}$. The proper transform $L_1^{(1)}$ 
of $L_1$ is not visible in this chart. 
The Painlev\'e vector field and the anticanonical pencil 
both have a base point $b_2$ given by $u_{211}=0$, $u_{212}=0$.

The second coordinate chart after the second blowup 
is defined by 
\begin{eqnarray*}
u_{111}&=&u_{221},\\
u_{112}&=&u_{222}\, u_{111},\\
u_{221}&=&{u_1}^{-1}=u_{211}\, u_{212},\\
u_{222}&=&{u_1}^2\, {u_2}^{-1}={u_{211}}^{-1},\\
u_1&=&{u_{221}}^{-1},\\
u_2&=&{u_{221}}^{-2}\, {u_{222}}^{-1},\\
\dot{u}_{221}&=&{u_{222}}^{-1}
\, (-1+2\, (5\, z)^{-1}\, u_{221}\, u_{222}),\\
\dot{u}_{222}&=&{u_{221}}^{-1}\, 
(2-6\, u_{221}\, {u_{222}}^2-{u_{221}}^3\, {u_{222}}^2
-(5\, z)^{-1}\, u_{221}\, u_{222}),\\
w_{22}&=&{u_{221}}^4\, {u_{222}}^2,\\
\dot{w}_{22}&=&
2\, {u_{221}}^4\, {u_{222}}^2\, 
(-6\, u_{222}-{u_{221}}^2\, u_{222}+3\, (5\, z)^{-1}),\\
E\, w_{22}&=&2^{-1}-2\, u_{221}\, {u_{222}}^2
-{u_{221}}^3\, {u_{222}}^2,\\
\dot{E}\, w_{22}&=&(5\, z)^{-1}\, 
(-3+12\, u_{221}\, {u_{222}}^2+2\, {u_{221}}^3\, {u_{222}}^2). 
\end{eqnarray*}
The equations $u_{221}=0$ and $u_{222}=0$ define $L_2$ 
and $L_1^{(1)}$, respectively. 
The line $L_0^{(2)}$ is not visible, 
and there are no base points in this chart. 

\subsection{Resolution of the flow at $b_{2}$}
Blowing up $S_2$ at $b_2$ leads to $S_3$. 
First coordinate chart: 
\begin{eqnarray*}
u_{211}&=&u_{311}\, u_{212},\\
u_{212}&=&u_{312},\\
u_{311}&=&{u_1}^{-3}\, {u_2}^2,\\
u_{312}&=&u_1\, {u_2}^{-1},\\ 
u_1&=&{u_{311}}^{-1}\, {u_{312}}^{-2},\\ 
u_2&=&{u_{311}}^{-1}\, {u_{312}}^{-3},\\
\dot{u}_{311}&=&{u_{312}}^{-1}
\, (12-3\, u_{311}+2\, {u_{311}}^2\, {u_{312}}^4),\\
\dot{u}_{312}&=&{u_{311}}^{-1}(-6+u_{311}-{u_{311}}^2\, {u_{312}}^4
+(5\, z)^{-1}\, u_{311}\, u_{312}),\\
w_{31}&=&{u_{311}}^3\, {u_{312}}^6,\\
\left[ w_{31}\, (u_{311}-4)^{-1}\right] ^{\bullet}&=&
2\, {u_{311}}^3\, {u_{312}}^6\, (u_{311}-4)^{-2}
\, (-{u_{311}}^2\, {u_{312}}^3+3\, (5\, z)^{-1}\, (u_{311}-4)),\\
E\, w_{31}&=&-2+2^{-1}\, u_{311}-{u_{311}}^2\, {u_{312}}^4,\\
\dot{E}\, w_{31}&=&(5\, z)^{-1}\, 
(12-3\, u_{311}+2\, {u_{311}}^2\, {u_{312}}^4).
\end{eqnarray*}
Then $u_{312}=0$ defines $L_3$ and $u_{311}=0$ defines 
the proper transform $L_0^{(3)}$ of 
$L_0^{(2)}$. The proper transforms $L_2^{(1)}$ 
of $L_2$ and $L_1^{(2)}$ of $L_1^{(1)}$ 
are not visible in this chart. 
The Painlev\'e vector field and the anticanonical pencil 
both have a base point $b_3$ given by $u_{311}=4$, $u_{312}=0$. 

The second coordinate chart after the third blowup 
is defined by 
\begin{eqnarray*}
u_{211}&=&u_{321},\\
u_{212}&=&u_{322}\, u_{211},\\
u_{321}&=&{u_1}^{-2}\, u_2=u_{311}\, u_{312},\\
u_{322}&=&{u_1}^3\, {u_2}^{-2}={u_{311}}^{-1},\\
u_1&=&{u_{321}}^{-2}\, {u_{322}}^{-1},\\
u_2&=&{u_{321}}^{-3}\, {u_{322}}^{-2},\\
\dot{u}_{321}&=&{u_{322}}^{-1}
\, (-2+6\, u_{322}+{u_{321}}^4\, {u_{322}}^3
+(5\, z)^{-1}\, u_{321}\, u_{322}),\\
\dot{u}_{322}&=&{u_{321}}^{-1}\, (3-12\, u_{322}
-2\, {u_{321}}^4\, {u_{322}}^3),\\
w_{32}&=&{u_{321}}^6\, {u_{322}}^4,\\
\left[ w_{32}\, (1-4\, u_{322})^{-1}\right] ^{\bullet}&=&
2\, {u_{321}}^6\, {u_{322}}^4\, (1-4\, u_{322})^{-2}
\, (-{u_{321}}^3\, {u_{322}}^2+3\, (5\, z)^{-1}
\, (1-4\, u_{322})),\\
E\, w_{32}&=&2^{-1}-2\, u_{322}-{u_{321}}^4\, {u_{322}}^3,\\
\dot{E}\, w_{32}&=&(5\, z)^{-1}\, 
(-3+12\, u_{322}+2\, {u_{321}}^4\, {u_{322}}^3). 
\end{eqnarray*}
The equations $u_{321}=0$ and $u_{322}=0$ define $L_3$ 
and $L_2^{(1)}$, respectively. 
The lines $L_0^{(3)}$ and $L_1^{(2)}$ 
are not visible. The Painlev\'e vector field 
and the anticanonical pencil 
both have a base point $b_3$ given by $u_{321}=0$, $u_{322}=1/4$ 
in this chart. 

\subsection{Resolution of the flow at $b_{3}$}
Blowing up $S_3$ at $b_3$ leads to $S_4$. 
First coordinate chart: 
\begin{eqnarray*}
u_{311}-4&=&u_{411}\, u_{312},\\
u_{312}&=&u_{412},\\
u_{411}&=&{u_1}^{-4}\, u_2
\, (-4\, {u_1}^3+{u_2}^2),\\
u_{412}&=&u_1\, {u_2}^{-1},\\ 
u_1&=&{u_{412}}^{-2}\, (4+u_{411}\, u_{412})^{-1},\\ 
u_2&=&{u_{412}}^{-3}\, (4+u_{411}\, u_{412})^{-1},\\
\dot{u}_{411}&=&{u_{412}}^{-1}\, (4+u_{411}\, u_{412})^{-1}\\
&&\times\, 
(-10\, u_{411}-4\, {u_{411}}^2\, u_{412}
+128\, {u_{412}}^3+112\, u_{411}\, {u_{412}}^4
+32\, {u_{411}}^2 \, {u_{412}}^5\\
&&+3\, {u_{411}}^3\, {u_{412}}^6
-(5\, z)^{-1}\, u_{411}\, u_{412}\, (4+u_{411}\, u_{412})),\\
\dot{u}_{412}&=&(4+u_{411}\, u_{412})^{-1}\\
&&\times\, (-2+u_{411}\, u_{412}-16\, {u_{412}}^4
-8\, u_{411}\, {u_{412}}^5-{u_{411}}^2\, {u_{412}}^6\\
&&+\, (5\, z)^{-1}\, u_{412}\, (4+u_{411}\, u_{412})),\\
w_{41}&=&{u_{412}}^5\, (4+u_{411}\, u_{412})^3,\\
\left[ w_{41}\, {u_{411}}^{-1}\right] ^{\bullet}&=&
2\, {u_{412}}^5\, (4+u_{411}\, {u_{412}})^3
\, {u_{411}}^{-2}
\, (-{u_{412}}^2\, (4+u_{411}\, u_{412})^2
+3 (5\, z)^{-1}\, u_{411}),\\
E\, w_{41}&=&2^{-1}\, u_{411}
-{u_{412}}^3\, (4+u_{411}\, u_{412})^2,\\
\dot{E}\, w_{41}&=&(5\, z)^{-1}\, 
(-3\, u_{411}+2\, {u_{412}}^3\, (4+u_{411}\, u_{412})^2).
\end{eqnarray*}
Then $u_{412}=0$ defines $L_4$ and $4+u_{411}\, u_{412}=0$ 
defines 
the proper transform $L_0^{(4)}$ of 
$L_0^{(3)}$. The proper transforms 
of the other lines on which the Painlev\'e 
vector field is infinite 
are not visible in this chart. 
The Painlev\'e vector field and the anticanonical pencil 
both have a base point $b_4$ given by $u_{411}=0$, $u_{412}=0$. 

The second coordinate chart after the fourth blowup 
is defined by 
\begin{eqnarray*}
u_{311}-4&=&u_{421},\\
u_{312}&=&u_{422}\, (u_{311}-4),\\
u_{421}&=&{u_1}^{-3}\, (-4\, {u_1}^3+{u_2}^2)
=u_{411}\, u_{412},\\
u_{422}&=&{u_1}^4\, {u_2}^{-1}\, (-4\, {u_1}^3+{u_2}^2)^{-1}
={u_{411}}^{-1},\\
u_1&=&{u_{421}}^{-2}\, (4+u_{421})^{-1}\, {u_{422}}^{-2},\\
u_2&=&{u_{421}}^{-3}\, (4+u_{421})^{-1}\, {u_{422}}^{-3},\\
\dot{u}_{421}&=&{u_{422}}^{-1}
\, (-3+32\, {u_{421}}^3\, {u_{422}}^4
+16\, {u_{421}}^4\, {u_{422}}^4+2\, {u_{421}}^5\, {u_{422}}^4),\\
\dot{u}_{422}&=&{u_{421}}^{-1}\, (4+u_{421})^{-1}
\,  (10+4\, u_{421}-128\, {u_{421}}^3\, {u_{422}}^4
-112\, {u_{421}}^4\, {u_{422}}^4\\
&&-32\, {u_{421}}^5\, {u_{422}}^4
-3\, {u_{421}}^6\, {u_{422}}^4
+(5\, z)^{-1}\, u_{421}\, (4+u_{421})\, u_{422}),\\
w_{42}&=&{u_{421}}^5\, (4+u_{421})^3\, {u_{422}}^6,\\
\dot{w}_{42}&=&2\, {u_{421}}^5\, 
(4+u_{421})^3\, {u_{422}}^6
\, (-{u_{421}}^2\, (4+u_{421})^2\, {u_{422}}^3
+3\, (5\, z)^{-1}),\\
E\, w_{42}&=&2^{-1}-{u_{421}}^3\, (4+u_{421})^2\, {u_{422}}^4,\\
\dot{E}\, w_{42}&=&(5\, z)^{-1}\, 
(-3+2\, {u_{421}}^3\, (4+u_{421})^2\, {u_{422}}^4). 
\end{eqnarray*}
The equations $u_{421}=0$, $4+u_{421}=0$, 
and $u_{422}=0$ define $L_4$, $L_0^{(4)}$, and  
and $L_3^{(1)}$, respectively. The proper transforms 
of the other lines on which the Painlev\'e 
vector field is infinite 
are not visible in this chart. 
Both the Painlev\'e vector field 
and the anticanonical pencil 
have no base point in this chart. 

\subsection{Resolution of the flow at $b_{4}$}
Blowing up $S_4$ at $b_4$ leads to $S_5$. 
First coordinate chart: 
\begin{eqnarray*}
u_{411}&=&u_{511}\, u_{412},\\
u_{412}&=&u_{512},\\
u_{511}&=&{u_2}^2\, (-4\, {u_1}^3+{u_2}^2)\, {u_1}^{-5},\\
u_{512}&=&u_1\, {u_2}^{-1},\\ 
u_1&=&{u_{512}}^{-2}\, (4+u_{511}\, {u_{512}}^2)^{-1},\\ 
u_2&=&{u_{512}}^{-3}\, (4+u_{511}\, {u_{512}}^2)^{-1},\\
\dot{u}_{511}&=&{u_{512}}^{-1}\, (4+u_{511}\, {u_{512}}^2)^{-1}\\
&&\times\, (-8\, u_{511}+128\, {u_{512}}^2
-5\, {u_{511}}^2\, {u_{512}}^2
+128\, u_{511}\, {u_{512}}^4
+40\, {u_{511}}^2\, {u_{512}}^6\\
&&+4\, {u_{511}}^3\, {u_{512}}^8
-2\, (5\, z)^{-1}
\, u_{511}\, u_{512}\, (4+u_{511}\, {u_{512}}^2)),\\
\dot{u}_{512}&=&(4+u_{511}\, {u_{512}}^2)^{-1}\\
&&\times\, (-2+u_{511}\, {u_{512}}^2-16\, {u_{512}}^4
-8\, u_{511}\, {u_{512}}^6-{u_{511}}^2\, {u_{512}}^8\\
&&+(5\, z)^{-1}\, u_{512}\, (4+u_{511}\, {u_{512}}^2)),\\
w_{51}&=&{u_{512}}^4\, (4+u_{511}\, {u_{512}}^2)^3,\\
\left[ w_{51}\, {u_{511}}^{-1}\right] ^{\bullet}
&=&2\, {u_{512}}^4\, (4+u_{511}\, {u_{512}}^2)^3\, {u_{511}}^{-2}
\, (-u_{512}\, (4+u_{511}\, {u_{512}}^2)^2
+3\, (5\, z)^{-1}\, u_{511}),\\
E\, w_{51}&=&2^{-1}\, u_{511} 
\, -{u_{512}}^2\, (4+u_{511}\, {u_{512}}^2)^2,\\
\dot{E}\, w_{51}&=&(5\, z)^{-1}\, 
(-3\, u_{511}+2\, {u_{512}}^2\, (4+u_{511}\, {u_{512}}^2)^2).
\end{eqnarray*}
Then $u_{512}=0$ defines $L_5$ and 
$4+u_{511}\, {u_{512}}^2=0$ 
defines 
the proper transform $L_0^{(5)}$ of 
$L_0^{(4)}$. The proper transforms 
of the other lines on which the Painlev\'e 
vector field is infinite 
are not visible in this chart. 
The Painlev\'e vector field and the anticanonical pencil 
both have a base point $b_5$ given by $u_{511}=0$, $u_{512}=0$. 

The second coordinate chart after the fifth blowup 
is defined by 
\begin{eqnarray*}
u_{411}&=&u_{521},\\
u_{412}&=&u_{522}\, u_{411},\\
u_{521}&=&{u_1}^{-4}\, (-4\, {u_1}^3+{u_2}^2)\, u_2
=u_{511}\, u_{512},\\
u_{522}&=&{u_1}^5\, (-4\, {u_1}^3+{u_2}^2)^{-1}\, {u_2}^{-2}
={u_{511}}^{-1},\\
u_1&=&{u_{521}}^{-2}\, {u_{522}}^{-2}
\, (4+{u_{521}}^2\, u_{522})^{-1},\\
u_2&=&{u_{521}}^{-3}\, {u_{522}}^{-3}
\, (4+{u_{521}}^2\, u_{522})^{-1},\\
\dot{u}_{521}&=&{u_{522}}^{-1}
\, (4+{u_{521}}^2\, u_{522})^{-1}\\
&&\times\, (-10-4\, {u_{521}}^2\, u_{522}
+128\, {u_{521}}^2\, {u_{522}}^3
+112\, {u_{521}}^4\, {u_{522}}^4
+32\, {u_{521}}^6\, {u_{522}}^5\\
&&+3\, {u_{521}}^8\, {u_{522}}^6
-(5\, z)^{-1}\, u_{521}\, u_{522}
\, (4+{u_{521}}^2\, u_{522})),\\
\dot{u}_{522}&=&{u_{521}}^{-1}\, (4+{u_{521}}^2\, u_{522})^{-1}\\
&&\times\,
(8+5\, {u_{521}}^2\, u_{522}
-128\, {u_{521}}^2\, {u_{522}}^3
-128\, {u_{521}}^4\, {u_{522}}^4
-40\, {u_{521}}^6\, {u_{522}}^5\\
&&-4\, {u_{521}}^8\, {u_{522}}^6
+2\, (5\, z)^{-1}\, u_{521}\, u_{522}
\, (4+{u_{521}}^2\, u_{522})),\\
w_{52}&=&{u_{521}}^4\, {u_{522}}^5
\, (4+{u_{521}}^2\, u_{522})^3,\\
\dot{w}_{52}&=&2\, {u_{521}}^4\, {u_{522}}^5
\, (4+{u_{521}}^2\, u_{522})^3\, 
(-u_{521}\, {u_{522}}^2
\, (4+{u_{521}}^2\, u_{522})^2+3\, (5\, z)^{-1}),\\
E\, w_{52}&=&2^{-1}-{u_{521}}^2\, {u_{522}}^3
\, (4+{u_{521}}^2\, u_{522})^2,\\
\dot{E}\, w_{52}&=&(5\, z)^{-1}\, 
(-3+2\, {u_{521}}^2\, {u_{522}}^3
\, (4+{u_{521}}^2\, u_{522})^2). 
\end{eqnarray*}
The equations $u_{521}=0$, $4+{u_{521}}^2\, u_{522}=0$, 
and $u_{522}=0$ define $L_5$, $L_0^{(5)}$, and  
and $L_4^{(1)}$, respectively. The proper transforms 
of the other lines on which the Painlev\'e 
vector field is infinite 
are not visible in this chart. 
Both the Painlev\'e vector field 
and the anticanonical pencil 
have no base point in this chart. 

\subsection{Resolution of the flow at $b_{5}$}
Blowing up $S_5$ at $b_5$ leads to $S_6$. 
First coordinate chart: 
\begin{eqnarray*}
u_{511}&=&u_{611}\, u_{512},\\
u_{512}&=&u_{612},\\
u_{611}&=&{u_1}^{-6}\, {u_2}^3\, (-4\, {u_1}^3+{u_2}^2),\\
u_{612}&=&u_1\, {u_2}^{-1},\\ 
u_1&=&{u_{612}}^{-2}\, (4+u_{611}\, {u_{612}}^3)^{-1},\\ 
u_2&=&{u_{612}}^{-3}\, (4+u_{611}\, {u_{612}}^3)^{-1},\\
\dot{u}_{611}&=&{u_{612}}^{-1}\, (4+u_{611}\, {u_{612}}^3)^{-1}\\
&&\times\, (-6\, u_{611}+128\, u_{612}-6\, {u_{611}}^2\, {u_{612}}^3\\
&&\ +144\, u_{611}\, {u_{612}}^4+48\, {u_{611}}^2\, {u_{612}}^7
+5\, {u_{611}}^3\, {u_{612}}^{10}\\
&&\ -3\, (5\, z)^{-1}\, 
u_{611}\, u_{612}\, (4+u_{611}\, {u_{612}}^3)),\\
\dot{u}_{612}&=&(4+u_{611}\, {u_{612}}^3)^{-1}\\
&&\times\, (-2+u_{611}\, {u_{612}}^3-16\, {u_{612}}^4
-8\, u_{611}\, {u_{612}}^7\\
&&-{u_{611}}^2\, {u_{612}}^{10}
+(5\, z)^{-1}\, u_{612}\, (4+u_{611}\, {u_{612}}^3)))\\
w_{61}&=&{u_{612}}^3\, (4+u_{611}\, {u_{612}}^3)^3,\\
\left[ w_{61}\, {u_{611}}^{-1}\right] ^{\bullet}&=&
2\, {u_{612}}^3\, (4+u_{611}\, {u_{612}}^3)^3
\, {u_{611}}^{-2}\\
&&\times\, (-(4+u_{611}\, {u_{612}}^3)^2+3\, (5\, z)^{-1}\, u_{611}),\\
E\, w_{61}&=&2^{-1}\, u_{611}-u_{612}\, (4+u_{611}\, {u_{612}}^3)^2,\\
\dot{E}\, w_{61}&=&(5\, z)^{-1}\, 
(-3\, u_{611}+2\, u_{612}\, (4+u_{611}\, {u_{612}}^3)^2).
\end{eqnarray*}
Then $u_{612}=0$ defines $L_6$ and 
$4+u_{611}\, {u_{612}}^3=0$ 
defines 
the proper transform $L_0^{(6)}$ of 
$L_0^{(5)}$. The proper transforms 
of the other lines on which the Painlev\'e 
vector field is infinite 
are not visible in this chart. 
The Painlev\'e vector field and the anticanonical pencil 
both have a base point $b_6$ given by $u_{611}=0$, $u_{612}=0$. 

The second coordinate chart after the sixth blowup 
is defined by 
\begin{eqnarray*}
u_{511}&=&u_{621},\\
u_{512}&=&u_{622}\, u_{511},\\
u_{621}&=&{u_1}^{-5}\, {u_2}^2\, (-4\, {u_1}^3+{u_2}^2)
=u_{611}\, u_{612},\\
u_{622}&=&{u_1}^6\, {u_2}^{-3}\, (-4\, {u_1}^3+{u_2}^2)^{-1}
={u_{611}}^{-1},\\
u_1&=&{u_{621}}^{-2}\, {u_{622}}^{-2}
\, (4+{u_{621}}^3\, {u_{622}}^2)^{-1},\\
u_2&=&{u_{621}}^{-3}\, {u_{622}}^{-3}
\, (4+{u_{621}}^3\, {u_{622}}^2)^{-1},\\
\dot{u}_{621}&=&{u_{622}}^{-1}
\, (4+{u_{621}}^3\, {u_{622}}^2)^{-1}\\
&&\times\, (-8+128\, u_{621}\, {u_{622}}^2
-5\, {u_{621}}^3\, {u_{622}}^2\\
&&\ 
+128\, {u_{621}}^4\, {u_{622}}^4
+40\, {u_{621}}^7\, {u_{622}}^6+4\, {u_{621}}^{10}\, {u_{622}}^8\\
&&\ 
-2\, (5\, z)^{-1}\, u_{621}\, u_{622}
\, (4+{u_{621}}^3\, {u_{622}}^2)),\\
\dot{u}_{622}&=&
{u_{621}}^{-1}\, (4+{u_{621}}^3\, {u_{622}}^2)^{-1}\\
&&\times\, (6-128\, u_{621}\, {u_{622}}^2
+6\, {u_{621}}^3\, {u_{622}}^2\\
&&\ 
-144\, {u_{621}}^4\, {u_{622}}^4
-48\, {u_{621}}^7\, {u_{622}}^6-5\, {u_{621}}^{10}\, {u_{622}}^8\\
&&\ 
+3\, (5\, z)^{-1}\, u_{621}\, u_{622}
\, (4+{u_{621}}^3\, {u_{622}}^2)),\\
w_{62}&=&{u_{621}}^3\, {u_{622}}^4
\, (4+{u_{621}}^3\, {u_{622}}^2)^3,\\
\dot{w}_{622}&=&2\, {u_{621}}^3\, {u_{622}}^4
\, (4+{u_{621}}^3\, {u_{622}}^2)^3\\
&&\times\, 
\, (-u_{622}\, (4+{u_{621}}^3\, {u_{622}}^2)^2
+3\, (5\, z)^{-1}),\\
E\, w_{62}&=&2^{-1}-u_{621}\, {u_{622}}^2
\, (4+{u_{621}}^3\, {u_{622}}^2)^2,\\
\dot{E}\, w_{62}&=&(5\, z)^{-1}\, 
(-3+2\, u_{621}\, {u_{622}}^2
\, (4+{u_{621}}^3\, {u_{622}}^2)^2). 
\end{eqnarray*}
The equations $u_{621}=0$, $4+{u_{621}}^3\, {u_{622}}^2=0$, 
and $u_{622}=0$ define $L_6$, $L_0^{(6)}$, and  
and $L_5^{(1)}$, respectively. The proper transforms 
of the other lines on which the Painlev\'e 
vector field is infinite 
are not visible in this chart. 
Both the Painlev\'e vector field 
and the anticanonical pencil 
have no base point in this chart. 

\subsection{Resolution of the flow at $b_{6}$}
Blowing up $S_6$ at $b_6$ leads to $S_7$. 
First coordinate chart: 
\begin{eqnarray*}
u_{611}&=&u_{711}\, u_{612},\\
u_{612}&=&u_{712},\\
u_{711}&=&{u_1}^{-7}\, {u_2}^4\, (-4\, {u_1}^3+{u_2}^2),\\
u_{712}&=&u_1\, {u_2}^{-1},\\ 
u_1&=&{u_{712}}^{-2}\, (4+u_{711}\, {u_{712}}^4)^{-1},\\ 
u_2&=&{u_{712}}^{-3}\, (4+u_{711}\, {u_{712}}^4)^{-1},\\
\dot{u}_{711}&=&{u_{712}}^{-1}\, (4+u_{711}\, {u_{712}}^4)^{-1}\\
&&\times\, (128-4\, u_{711}
+160\, u_{711}\, {u_{712}}^4
-7\, {u_{711}}^2\, {u_{712}}^4\\
&&\ +56\, {u_{711}}^2\, {u_{712}}^8+6\, {u_{711}}^3\, {u_{712}}^{12}\\
&&\ -4\, (5\, z)^{-1}\, u_{711}\, u_{712}
\, (4+u_{711}\, {u_{712}}^4)),\\
\dot{u}_{712}&=&(4+u_{711}\, {u_{712}}^4)^{-1}\\
&&\times\, (-2-16\, {u_{712}}^4+u_{711}\, {u_{712}}^4
-8\, u_{711}\, {u_{712}}^8\\
&&-{u_{711}}^2\, {u_{712}}^{12} 
+(5\, z)^{-1}\, u_{712}\, (4+u_{711}\, {u_{712}}^4)),\\
w_{71}&=&{u_{712}}^2\, (4+u_{711}\, {u_{712}}^4)^3,\\
\left[ w_{71}\, (u_{711}-32)^{-1}\right] ^{\bullet}&=&
-2\, {u_{712}}^2\, (4+u_{711}\, {u_{712}}^4)^2
\, (u_{711}-32)^{-2}\\
&&\times\, 
({u_{712}}^3\, 
(1024-64\, u_{711}+512\, u_{711}\, {u_{712}}^4
+12\, {u_{711}}^2\, {u_{712}}^4\\
&&+64\, {u_{711}}^2\, {u_{712}}^8
+{u_{711}}^3\, {u_{712}}^8)\\
&&+(5\, z)^{-1}\, (32-3\, u_{711})\, (4+u_{711}\, {u_{712}}^4)),\\
E\, w_{71}&=&2^{-1}\, u_{711}-(4+u_{711}\, {u_{712}}^4)^2,\\
\dot{E}\, w_{71}&=&(5\, z)^{-1}\, 
(-3\, u_{711}+2\, (4+u_{711}\, {u_{712}}^4)^2).
\end{eqnarray*}
Then $u_{712}=0$ defines $L_7$ and 
$4+u_{711}\, {u_{712}}^4=0$ 
defines 
the proper transform $L_0^{(7)}$ of 
$L_0^{(6)}$. The proper transforms 
of the other lines on which the Painlev\'e 
vector field is infinite 
are not visible in this chart. 
The Painlev\'e vector field and the anticanonical pencil 
both have a base point $b_7$ given by $u_{711}=32$, $u_{712}=0$. 
Remarkably this base point in the Boutroux coordinates 
does not depend on the independent variable $z$, 
whereas the seventh base point in the unscaled system is 
given by $y_{711}=32\, x$, $y_{712}=0$. 

The second coordinate chart after the seventh blowup 
is defined by 
\begin{eqnarray*}
u_{611}&=&u_{721},\\
u_{612}&=&u_{722}\, u_{611},\\
u_{721}&=&{u_1}^{-6}\, {u_2}^3\, (-4\, {u_1}^3+{u_2}^2)
=u_{711}\, u_{712},\\
u_{722}&=&{u_1}^7\, {u_2}^{-4}\, (-4\, {u_1}^3+{u_2}^2)^{-1}
={u_{711}}^{-1},\\
u_1&=&{u_{721}}^{-2}\, {u_{722}}^{-2}
\, (4+{u_{721}}^4\, {u_{722}}^3)^{-1},\\
u_2&=&{u_{721}}^{-3}\, {u_{722}}^{-3}
\, (4+{u_{721}}^4\, {u_{722}}^3)^{-1}
\end{eqnarray*}
\begin{eqnarray*}
\dot{u}_{721}&=&{u_{722}}^{-1}\, (4+{u_{721}}^4\, {u_{722}}^3)^{-1}
\\&&\times\,
(-6+128\, u_{722}-6\, {u_{721}}^4\, {u_{722}}^3
+144\, {u_{721}}^4\, {u_{722}}^4
+48\, {u_{721}}^8\, {u_{722}}^7\\
&&+5\, {u_{721}}^{12}\, {u_{722}}^{10}
-3\, (5\, z)^{-1}\, u_{721}\, u_{722}
\, (4+{u_{721}}^4\, {u_{722}}^3)),\\
\dot{u}_{722}&=&-{u_{721}}^{-1}
\, (4+{u_{721}}^4\, {u_{722}}^3)^{-1}
\\&&\times\,
(4-128\, u_{722}+7\, {u_{721}}^4\, {u_{722}}^3
-160\, {u_{721}}^4\, {u_{722}}^4
-56\, {u_{721}}^8\, {u_{722}}^7\\
&&-6\, {u_{721}}^{12}\, {u_{722}}^{10}
+4\, (5\, z)^{-1}\, u_{721}\, u_{722}
\, (4+{u_{721}}^4\, {u_{722}}^3)),\\
w_{72}&=&{u_{721}}^2\, {u_{722}}^3
\, (4+{u_{721}}^4\, {u_{722}}^3)^3,\\
\left[ w_{72}\, (32\, u_{722}-1)^{-1}\right] ^{\bullet}&=&
2\, {u_{721}}^2\, {u_{722}}^3
\, (4+{u_{721}}^4\, {u_{722}}^3)^2\, (32\, u_{722}-1)^{-2}\\
&&\times\, 
({u_{721}}^3\, {u_{722}}^3\, 
(-64+1024\, u_{722}+12\, {u_{721}}^4\, {u_{722}}^3\\
&&+512\, {u_{721}}^4\, {u_{722}}^4
+{u_{721}}^8\, {u_{722}}^6
+64\, {u_{721}}^8\, {u_{722}}^7)\\
&&+(5\, z)^{-1}\, 
(-3+32\, u_{722})\, (4+{u_{721}}^4\, {u_{722}}^3)),\\
E\, w_{72}&=&2^{-1}-u_{722}\, (4+{u_{721}}^4\, {u_{722}}^3)^2,\\
\dot{E}\, w_{72}&=&(5\, z)^{-1}\, 
(-3+2\, u_{722}\, (4+{u_{721}}^4\, {u_{722}}^3)^2). 
\end{eqnarray*}
The equations $u_{721}=0$, 
$4+{u_{721}}^4\, {u_{722}}^3=0$, 
and $u_{722}=0$ define $L_7$, $L_0^{(7)}$, and  
and $L_6^{(1)}$, respectively. The proper transforms 
of the other lines on which the Painlev\'e 
vector field is infinite 
are not visible in this chart. 
The Painlev\'e vector field 
and the anticanonical pencil 
have the base point $u_{721}=0$, $u_{722}=1/32$ 
in this chart, which is equal to 
the previously found base point $b_7$. 

\subsection{Resolution of the flow at $b_{7}$}
Blowing up $S_7$ at $b_7$ leads to $S_8$. 
First coordinate chart: 
\begin{eqnarray*}
u_{711}-32&=&u_{811}\, u_{712},\\
u_{712}&=&u_{812},\\
u_{811}&=&-{u_1}^{-8}\, u_2\, 
(32\, {u_1}^7+4\, {u_1}^3\, {u_2}^4-{u_2}^6),\\
u_{812}&=&u_1\, {u_2}^{-1},\\ 
u_1&=&{u_{812}}^{-2}
\, (4+32\, {u_{812}}^4+u_{811}\, {u_{812}}^5)^{-1},\\ 
u_2&=&{u_{812}}^{-3}
\, (4+32\, {u_{812}}^4+u_{811}\, {u_{812}}^5)^{-1},\\
\dot{u}_{811}&=&{u_{812}}^{-1}
\, (4+32\, {u_{812}}^4+u_{811}\, {u_{812}}^5)^{-1}\\
&&\times\, 
[-2\, (u_{811}+1024\, {u_{812}}^3
+152\, u_{811}\, {u_{812}}^4
+4\, {u_{811}}^2\, {u_{812}}^5)\\
&&+{u_{812}}^7\, (32+u_{811}\, u_{812})
\, (1792+64\, u_{811}\, u_{812}+6144\, {u_{812}}^4
+416\, u_{811}\, {u_{812}}^5\\
&&+7\, {u_{811}}^2\, {u_{812}}^6)
-(5\, z)^{-1}
\, (128+5\, u_{811}\, u_{812})
\, (4+32\, {u_{812}}^4+u_{811}\, {u_{812}}^5)],\\
\dot{u}_{812}&=&
-(4+32\, {u_{812}}^4+u_{811}\, {u_{812}}^5)^{-1}\\
&&\times\,
[2-16\, {u_{812}}^4-u_{811}\, {u_{812}}^5
+256\, {u_{812}}^8+8\, u_{811}\, {u_{812}}^9
+1024\, {u_{812}}^{12}\\
&&+64\, u_{811}\, {u_{812}}^{13}
+{u_{811}}^2\, {u_{812}}^{14}
-(5\, z)^{-1}\, u_{812}
\, (4+u_{811}\, {u_{812}}^5+32\, {u_{812}}^9)],\\
w_{81}&=&u_{812}
\, (4+32\, {u_{812}}^4+u_{811}\, {u_{812}}^5)^3,\\
E\, w_{81}&=&2^{-1}\, u_{811}-{u_{812}}^3
\, (32+u_{811}\, u_{812})
\, (8+32\, {u_{812}}^4+u_{811}\, {u_{812}}^5),\\
\dot{E}\, w_{81}&=&(5\, z)^{-1}\, 
{u_{812}}^{-1}\, 
(-64-3\, {u_{811}}\, u_{812}\\
&&+
2\, {u_{812}}^4
\, (32+u_{811}\, u_{812})
\, (8+32\, {u_{812}}^4+u_{811}\, {u_{812}}^5)).
\end{eqnarray*}
Furthermore 
\begin{eqnarray*}
&&\left[ w_{81}\, 
(u_{811}+256\, (5\, z)^{-1})^{-1}\right] ^{\bullet}=
2\, u_{812}\, (4+32\, {u_{812}}^4+u_{811}\, {u_{812}}^5)^2
\, (u_{811}+256\, (5\, z)^{-1})^{-2}\\
&&\times\, 
[-{u_{812}}^2\, 
(-2^{10}-2^6\, u_{811}\, u_{812}
+2^{12}\cdot 7\, {u_{812}}^4
+2^8\cdot 5\, u_{811}\, {u_{812}}^5
+2^2\cdot 3\, {u_{811}}^2\, {u_{812}}^6\\
&&+2^{15}\cdot 3\, {u_{812}}^8
+2^{10}\cdot 7\, u_{811}\, {u_{812}}^9
+2^5\cdot 5\, {u_{811}}^2\, {u_{812}}^{10}
+{u_{811}}^3\, {u_{812}}^{11})\\
&&+(5\, z)^{-1}\, 
(2^2\cdot 3\, u_{811}
-2^{12}\cdot 5\, {u_{812}}^3
-2^5\cdot 3^3\, u_{811}\, {u_{812}}^4
+3\, {u_{811}}^2\, {u_{812}}^5
+2^{15}\cdot 5\, {u_{812}}^7\\
&&+2^{10}\cdot 5\, u_{811}\, {u_{812}}^8
+2^{17}\cdot 5\, {u_{812}}^{11}
+2^{13}\cdot 5\, u_{811}\, {u_{812}}^{12}
+2^7\cdot 5\, {u_{811}}^2\, {u_{812}}^{13})\\
&&+768\, (5\, z)^{-2}
\, (4+32\, {u_{812}}^4+u_{811}\, {u_{812}}^5)].
\end{eqnarray*}
The equation $u_{812}=0$ defines $L_8$ and 
$4+32\, {u_{812}}^4+u_{811}\, {u_{812}}^5=0$ 
defines 
the proper transform $L_0^{(8)}$ of 
$L_0^{(7)}$. The proper transforms 
of the other lines on which the Painlev\'e 
vector field is infinite 
are not visible in this chart. 
The Painlev\'e vector field 
has a base point $b_8$ given by 
$u_{811}=\, -256\, (5\, z)^{-1}$, $u_{812}=0$. 
In the Boutroux coordinates, this is the first base 
point which depends on the independent variable $z$. 
The anticanonical pencil has a base point 
$b_8^{\,\scriptop{ell}}$ given by 
$u_{811}=$, $u_{812}=0$. We have 
$b_8^{\,\scriptop{ell}}\neq b_8$ 
with a distance between both base 
point vanishing of order $1/z$ as $z\to\infty$. 

The second coordinate chart after the eighth blowup 
is defined by 
\begin{eqnarray*}
u_{711}-32&=&u_{821},\\
u_{712}&=&u_{822}\, (u_{711}-32),\\
u_{821}&=&-{u_1}^{-7}
\, (32\, {u_1}^7+4\, {u_1}^3\, {u_2}^4-{u_2}^6)
=u_{811}\, u_{812},\\
u_{822}&=&-{u_1}^8\, {u_2}^{-1}
\, (32\, {u_1}^7+4\, {u_1}^3\, {u_2}^4-{u_2}^6)^{-1}
={u_{811}}^{-1},\\
u_1&=&{u_{821}}^{-2}\, {u_{822}}^{-2}\, 
(4+32\, {u_{821}}^4\, {u_{822}}^4+{u_{821}}^5\, {u_{822}}^4)^{-1},\\
u_2&=&{u_{821}}^{-3}\, {u_{822}}^{-3}\, 
(4+32\, {u_{821}}^4\, {u_{822}}^4+{u_{821}}^5\, {u_{822}}^4)^{-1},\\
\dot{u}_{821}&=&{u_{822}}^{-1}\, 
(4+32\, {u_{821}}^4\, {u_{822}}^4+{u_{821}}^5\, {u_{822}}^4)^{-1}
\\
&&\times\, [-4-2048\, {u_{821}}^3\, {u_{822}}^4
-288\, {u_{821}}^4\, {u_{822}}^4
-7\, {u_{821}}^5\, {u_{822}}^4\\
&&+2\, {u_{821}}^7\, 
(32+u_{821})^2\, {u_{822}}^8
\, (28+96\, {u_{821}}^4\, {u_{822}}^4
+3\, {u_{821}}^5\, {u_{822}}^4)\\
&&-4\, (5\, z)^{-1}\, (32+u_{821})\, 
u_{822}\, (4+32\, {u_{821}}^4\, {u_{822}}^4
+{u_{821}}^5\, {u_{822}}^4)],\\
\dot{u}_{822}&=&-{u_{821}}^{-1}\, 
(4+32\, {u_{821}}^4\, {u_{822}}^4+{u_{821}}^5\, {u_{822}}^4)^{-1}
\\&&\times\, [-2
-8\, {u_{821}}^3\, (256+38\, u_{821}+{u_{821}}^2)\, {u_{822}}^4\\
&&+{u_{821}}^7\, (32+u_{821})\, {u_{822}}^8
\, (1792+64\, u_{821}+6144\, {u_{821}}^4\, {u_{822}}^4
+416\, {u_{821}}^5\, {u_{822}}^4\\
&&+7\, {u_{821}}^6\, {u_{822}}^4)-(5\, z)^{-1}
\, (128+5\, u_{821})\, u_{822}\, 
(4+32\, {u_{821}}^4\, {u_{822}}^4+{u_{821}}^5\, {u_{822}}^4)],\\
w_{82}&=&u_{821}\, {u_{822}}^2\, 
(4+32\, {u_{821}}^4\, {u_{822}}^4+{u_{821}}^5\, {u_{822}}^4)^3,\\
E\, w_{82}&=&
2^{-1}-{u_{821}}^3\, (32+u_{821})
\, {u_{822}}^4\, (8+32\, {u_{821}}^4\, {u_{822}}^4
+{u_{821}}^5\, {u_{822}}^4),\\
\dot{E}\, w_{82}&=&(5\, z)^{-1}\, 
{u_{821}}^{-1}\, 
(-64-3\, u_{821}\\
&&+2\, {u_{821}}^4\, (32+u_{821})
\, {u_{822}}^4\, (8+32\, {u_{821}}^4\, {u_{822}}^4
+{u_{821}}^5\, {u_{822}}^4)). 
\end{eqnarray*}
Furthermore, 
\begin{eqnarray*}
&&\left[ w_{82}\, (256\, (5\, z)^{-1}\, u_{822}+1)^{-1}\right] ^{\bullet}=\\
&&2\, {u_{822}}^2\, 
(4+32\, {u_{821}}^4\, {u_{822}}^4+{u_{821}}^5\, {u_{822}}^4)^2
\, (256\, (5\, z)^{-1}\, u_{822}+1)^{-2}\\
&&\times\, [-{u_{821}}^3\, {u_{822}}^3
\, (-2^{10}-2^6\, u_{821}
+2^{17}\cdot 7\, {u_{821}}^4\, {u_{822}}^4
+2^8\cdot 5\, {u_{821}}^5\, {u_{822}}^4
+2^2\cdot 3\, {u_{821}}^6\, {u_{822}}^4\\
&&+2^{15}\cdot 3\, {u_{821}}^8\, {u_{822}}^8
+2^{10}\cdot 7\, {u_{821}}^9\, {u_{822}}^8
+2^5\cdot 5\, {u_{821}}^{10}\, {u_{822}}^8
+{u_{821}}^{11}\, {u_{822}}^8)\\
&&-(5\, z)^{-1}
\, (-2^8-2^2\cdot 3\, u_{821}
-2^{18}\, {u_{821}}^3\, {u_{822}}^4
-2^{11}\cdot 3^2\, {u_{821}}^4\, {u_{822}}^4\\
&&-2^5\cdot 5\, {u_{821}}^5\, {u_{822}}^4
-3\, {u_{821}}^6\, {u_{822}}^4
+2^{20}\cdot 7\, {u_{821}}^7\, {u_{822}}^8\\
&&+2^{16}\cdot 5\, {u_{821}}^8\, {u_{822}}^8
+2^{10}\cdot 3\, {u_{821}}^9\, {u_{822}}^8
+2^{23}\cdot 3\, {u_{821}}^{11}\, {u_{822}}^{12}\\
&&+2^{18}\cdot 7\, {u_{821}}^{12}\, {u_{822}}^{12}
+2^{13}\cdot 5\, {u_{821}}^{13}\, {u_{822}}^{12}
+2^8\, {u_{821}}^{14}\, {u_{822}}^{12})\\
&&+2^7\, (5\, z)^{-2}\, (2^7+11\, u_{821})
\, (4+32\, {u_{821}}^4\, {u_{822}}^4+{u_{821}}^5\, {u_{822}}^4)
\, u_{822}].
\end{eqnarray*}
The equations $u_{821}=0$, 
$4+32\, {u_{821}}^4\, {u_{822}}^4+{u_{821}}^5\, {u_{822}}^4=0$, 
and $u_{822}=0$ define $L_8$, $L_0^{(8)}$, and  
and $L_7^{(1)}$, respectively. The proper transforms 
of the other lines on which the Painlev\'e 
vector field is infinite 
are not visible in this chart. 
The Painlev\'e vector field 
has the base point $b_8$ defined by the 
equations $u_{821}=0$, $256\, (5\, z)^{-1}\, u_{822}+1=0$. 
The base point $b_8^{\,\scriptop{ell}}$ 
of the anticanonical pencil is not visible in this chart. 

\subsection{Resolution of the flow at $b_{8}$}\label{lastblowup}
Blowing up $S_8$ at $b_8=b_8(z)$ leads to $S_9=S_9(z)$. 
First coordinate chart: 
\begin{eqnarray*}
u_{811}+256\, (5\,z)^{-1}&=&u_{911}\, u_{812},\\
u_{812}&=&u_{912},\\
u_{911}&=&{u_1}^{-9}\, u_2\, 
(-32\, {u_1}^7\, u_2-4\, {u_1}^3\, {u_2}^5
+{u_2}^7+256\, (5\, z)^{-1}\, {u_1}^8),\\
u_{912}&=&u_1\, {u_2}^{-1},\\ 
u_1&=&{u_{912}}^{-2}\, 
(4+32\, {u_{912}}^4+u_{911}\, {u_{912}}^6
-256\, (5\, z)^{-1}\, {u_{912}}^5)^{-1},\\ 
u_2&=&{u_{912}}^{-3}\, 
(4+32\, {u_{912}}^4+u_{911}\, {u_{912}}^6
-256\, (5\, z)^{-1}\, {u_{912}}^5)^{-1},\\
\dot{u}_{911}&=&(4+32\, {u_{912}}^4+u_{911}\, {u_{912}}^6
-256\, (5\, z)^{-1}\, {u_{912}}^5)^{-1}\\
&&\times\, 
[u_{912}\, 
(-2^{11}-2^6\cdot 5\, u_{911}\, {u_{912}}^2
+2^{13}\cdot 7\, {u_{912}}^4
-3^2\, {u_{911}}^2\, {u_{912}}^4\\
&&+2^{12}\, u_{911}\, {u_{912}}^6
+2^{16}\cdot 3\, {u_{912}}^8
+2^3\cdot 3^2\, {u_{911}}^2\, {u_{912}}^8\\
&&+2^{12}\cdot 5\, u_{911}\, {u_{912}}^{10}
+2^6\cdot 11\, {u_{911}}^2\, {u_{912}}^{12}
+2^3\, {u_{911}}^3\, {u_{912}}^{14})\\
&&-2\, (5\, z)^{-1}\, 
(2^2\cdot 3\, u_{911}
-2^{12}\cdot 3^2\, {u_{912}}^2
-2^5\cdot 3^2\cdot 7\, u_{911}\, {u_{912}}^4\\
&&+2^{15}\cdot 3\cdot 5\, {u_{912}}^6
+3\, {u_{911}}^2\, {u_{912}}^6
+2^{10}\cdot 17\, u_{911}\, {u_{912}}^8\\
&&+2^{17}\cdot 19\, {u_{912}}^{10}
+2^{13}\cdot 3\cdot 7\, u_{911}\, {u_{912}}^{12}
+2^7\cdot 23\, {u_{911}}^2\, {u_{912}}^{14})\\
&&+2^9\, (5\, z)^{-2}\, {u_{912}}^3
\, (-2^6\cdot 3\cdot 5
+3\, u_{911}\, {u_{912}}^2\\
&&+2^{13}\, {u_{912}}^4
+2^{14}\cdot 5\, {u_{912}}^8
+2^8\cdot 11\, u_{911}\, {u_{912}}^{10})\\
&&-2^{24}\cdot 7\, (5\, z)^{-3}\, {u_{912}}^{12}]
\end{eqnarray*}
\begin{eqnarray*}
\dot{u}_{912}&=&-(4+32\, {u_{912}}^4+u_{911}\, {u_{912}}^6
-256\, (5\, z)^{-1}\, {u_{912}}^5)^{-1}\\
&&\times\, [2-2^4\, {u_{912}}^4
-u_{911}\, {u_{912}}^6+2^8\, {u_{912}}^8
+2^3\, u_{911}\, {u_{912}}^{10}\\
&&+2^{10}\, {u_{912}}^{12}
+2^6\, u_{911}\, {u_{912}}^{14}
+{u_{911}}^2\, {u_{912}}^{16}\\
&&-(5\, z)^{-1}\, u_{912}\, 
(2^2-2^5\cdot 7\, {u_{912}}^4
+u_{911}\, {u_{912}}^6\\
&&+2^{11}\, {u_{912}}^8
+2^{14}\, {u_{912}}^{12}
+2^9\, u_{911}\, {u_{912}}^{14})\\
&&+2^8\, (5\, z)^{-2}\, {u_{912}}^6\, (1+2^8\, {u_{912}}^8)],\\
w_{91}&=&(4+32\, {u_{912}}^4+u_{911}\, {u_{912}}^6
-256\, (5\, z)^{-1}\, {u_{912}}^5)^3,\\
\dot{w}_{91}&=&3\, {u_{912}}^3\, 
(4+32\, {u_{912}}^4+u_{911}\, {u_{912}}^6
-256\, (5\, z)^{-1}\, {u_{912}}^5)^2)\\
&&\times\, [-2^6-3\, u_{911}\, {u_{912}}^2
+2^9\, {u_{912}}^4+2^4\, u_{911}\, {u_{912}}^6
+2^{11}\, {u_{912}}^8\\
&&+2^7\, u_{911}\, {u_{912}}^{10}
+2\, {u_{911}}^2\, {u_{912}}^{12}\\
&&-2^8\, (5\, z)^{-1}\, u_{912}\, 
(-3+2^4\, {u_{912}}^4
+2^7\, {u_{912}}^8
+2^2\, u_{911}\, {u_{912}}^{10})\\
&&+2^{17}\, (5\, z)^{-2}\, {u_{912}}^{10}],\\
E\, w_{91}&=&-2^{-1}\, {u_{912}}^{-1}
\, [u_{912}\, (-u_{911}
+2^9\, {u_{912}}^2
+2^4\, u_{911}\, {u_{912}}^4\\
&&+2^{11}\, {u_{912}}^6
+2^7\, u_{911}\, {u_{912}}^8
+2\, {u_{911}}^2\, {u_{912}}^{10})\\
&&-2^8\, (5\, z)^{-1}\, 
(-1+2^4\, {u_{912}}^4
+2^7\, {u_{912}}^8
+4\, u_{911}\, {u_{912}}^{10})\\
&&+2^{17}\, (5\, z)^{-2}\, {u_{912}}^9],\\
\dot{E}\, w_{91}&=&(5\, z)^{-1}\, {u_{912}}^{-2}
\, [-2^6-3\, u_{911}\, {u_{912}}^2
+2^9\, {u_{912}}^4
+2^4\, u_{911}\, {u_{912}}^6\\
&&+2^{11}\, {u_{912}}^8
+2^7\, u_{911}\, {u_{912}}^{10}
+2\, {u_{911}}^2\, {u_{912}}^{12}\\
&&-2^8\, (5\, z)^{-1}\, u_{912}\, 
(-3+2^4\, {u_{912}}^4
+2^7\, {u_{912}}^8
+2^2\, u_{911}\, {u_{912}}^{10})\\
&&2^{17}\, (5\, z)^{-2}\, {u_{912}}^{10}].
\end{eqnarray*}
As the change of coordinates from $(u_1,\, u_2)$ 
to all previous coordinate systems $(u_{ij1},\, u_{ij2})$ 
for $i\leq 8$ do not depend on $z$ the limiting 
system of differential equations $^0\dot{u}_1=u_2$, 
$^0\dot{u}_2=6\, {u_1}^2+1$ in the coordinate systems 
$(u_{ij1},\, u_{ij2})$ for $i\leq 8$ is obtained by 
deleting the term in $\dot{u}_{ij1}$ and $\dot{u}_{ij2}$ 
which have a factor $1/z$ in front. This is no longer true 
in the coordinate system $(u_{911},\, u_{912})$. 
However, the difference between the Painlev\'e-Boutroux system and the 
limiting system still has a relatively simple expression: 
\begin{equation}
\begin{array}{ccc}
\dot{u}_{911}-^0\dot{u}_{911}&=&-2\, (5\, z)^{-1}\, {u_{912}}^{-2}
\, (64-640\, (5\, z)^{-1}\, u_{912}+3\, u_{911}\, {u_{912}}^2)\\
\dot{u}_{912}-^0\dot{u}_{912}&=&(5\, z)^{-1}\, u_{912}. 
\end{array}
\label{u-0u}
\end{equation}

The equation $u_{912}=0$ defines $L_9$ and 
$4+32\, {u_{912}}^4+u_{911}\, {u_{912}}^6
-256\, (5\, z)^{-1}\, {u_{912}}^5=0$ 
defines 
the proper transform $L_0^{(9)}$ of 
$L_0^{(8)}$. The proper transforms 
of the other lines on which the Painlev\'e 
vector field is infinite 
are not visible in this chart. 
The Painlev\'e vector field 
is regular along $L_9$, nonzero, and transversal to it. 
Moreover, the Painleve\'{e} vector field 
has no base points in this chart. 
On the other hand the blowing up of $S_8$ in 
the point $b_8$, which is not the base point of the 
anticanonical pencil, causes $E\, w_{91}$ to be 
infinite along 
$L_9$, the line determined by the equation $u_{912}=0$. 
The image $^{(1)}b_8^{\,\scriptop{ell}}$ 
of $b_8^{\,\scriptop{ell}}$ in $S_9$ is not 
visible in this coordinate chart.  

The second coordinate chart after the ninth blowup 
is defined by 
\begin{eqnarray*}
u_{921}&=&
u_{811}+2^8\, (5\, z)^{-1},\\
u_{812}&=&u_{922}\, (u_{811}+2^8\, (5\, z)^{-1}),\\
u_{921}&=&{u_1}^{-8}\, 
(-2^5\, {u_1}^7\, u_2-2^2\, {u_1}^3\, {u_2}^5
+{u_2}^7+2^8\, (5\, z)^{-1}\, {u_1}^8)
=u_{911}\, u_{912},\\
u_{922}&=&{u_1}^9\, {u_2}^{-1}
\, (-2^5\, {u_1}^7\, u_2-2^2\, {u_1}^3\, {u_2}^5
+{u_2}^7+2^8\, (5\, z)^{-1}\, {u_1}^8)^{-1}
={u_{911}}^{-1},\\
u_1&=&{u_{921}}^{-2}\, {u_{922}}^{-2}
\, (2^2+2^5\, {u_{921}}^4\, {u_{922}}^4
+{u_{921}}^6\, {u_{922}}^5-2^8\, (5\,z)^{-1}
\, {u_{921}}^5\, {u_{922}}^5)^{-1},\\
u_2&=&{u_{921}}^{-3}\, {u_{922}}^{-3}
\, (2^2+2^5\, {u_{921}}^4\, {u_{922}}^4
+{u_{921}}^6\, {u_{922}}^5-2^8\, (5\,z)^{-1}
\, {u_{921}}^5\, {u_{922}}^5)^{-1},\\
\dot{u}_{921}&=&-{u_{922}}^{-1}\, 
(2^2+2^5\, {u_{921}}^4\, {u_{922}}^4
+{u_{921}}^6\, {u_{922}}^5-2^8\, (5\,z)^{-1}
\, {u_{921}}^5\, {u_{922}}^5)^{-1}\\
&&\times\, [2+2^{11}\, {u_{921}}^2\, {u_{922}}^3
+2^4\cdot 19\, {u_{921}}^4\, {u_{922}}^4
+2^3\, {u_{921}}^6\, {u_{922}}^5\\
&&-2^{13}\cdot 7\, {u_{921}}^6\, {u_{922}}^7
-2^8\cdot 3\cdot 5\, {u_{921}}^8\, {u_{922}}^8
-2^6\, {u_{921}}^{10}\, {u_{922}}^9\\
&&-2^{16}\cdot 3\, {u_{921}}^{10}\, {u_{922}}^{11}
-2^{10}\cdot 19\, {u_{921}}^{12}\, {u_{922}}^{12}
-2^7\cdot 5\, {u_{921}}^{14}\, {u_{922}}^{13}\\
&&-7\, {u_{921}}^{16}\, {u_{922}}^{14}
+(5\, z)^{-1}\, u_{921}\, u_{922}\\ 
&&\times\, (2^2\cdot 5-2^{13}\, 3^2\, {u_{921}}^2\, {u_{922}}^3
-2^5\cdot 7\cdot 17\, {u_{921}}^4\, {u_{922}}^4\\
&&+5\, {u_{921}}^6\, {u_{922}}^5
+2^{16}\cdot 3\cdot 5\, {u_{921}}^6\, {u_{922}}^7
+2^{15}\, {u_{921}}^8\, {u_{922}}^8\\
&&+2^{18}\cdot 19\, {u_{921}}^{10}\, {u_{922}}^{11}
+2^{16}\cdot 5\, {u_{921}}^{12}\, {u_{922}}^{12}
+2^8\cdot 3\cdot 7\, {u_{921}}^{14}\, {u_{922}}^{13})\\
&&-2^8\, (5\, z)^{-2}\, {u_{921}}^4\, {u_{922}}^5
\, (-2^7\cdot 3\cdot 5
+5\, {u_{921}}^2\, u_{922}\\
&&
+2^{14}\, {u_{921}}^4\, {u_{922}}^4
+2^{15}\cdot 5\, {u_{921}}^8\, {u_{922}}^8
+2^8\cdot 3\cdot 7\, {u_{921}}^{10}\, {u_{922}}^9)\\
&&2^{24}\cdot 7\,  (5\, z)^{-3}
\, {u_{921}}^{13}\, {u_{922}}^{14}],\\
\dot{u}_{922}&=&-(2^2+2^5\, {u_{921}}^4\, {u_{922}}^4
+{u_{921}}^6\, {u_{922}}^5-2^8\, (5\,z)^{-1}
\, {u_{921}}^5\, {u_{922}}^5)^{-1}\, u_{922}\\
&&\times\,
[u_{921}\, {u_{922}}^2\, 
(-2^{11}\, -2^6\cdot 5\, {u_{921}}^2\, u_{922}
-3^2\, {u_{921}}^4\, {u_{922}}^2\\
&&+2^{13}\cdot 7\, {u_{921}}^4\, {u_{922}}^4
+2^{12}\, {u_{921}}^6\, {u_{922}}^5
+2^3\cdot 3^2\, {u_{921}}^8\, {u_{922}}^6\\
&&+2^{16}\cdot 3\, {u_{921}}^8\, {u_{922}}^8
+2^{12}\cdot 5\, {u_{921}}^{10}\, {u_{922}}^9
+2^6\cdot 11\, {u_{921}}^{12}\, {u_{922}}^{10}\\
&&+2^3\, {u_{921}}^{14}\, {u_{922}}^{11})
-2\, (5\, z)^{-1}\, 
(2^2\cdot 3-2^{12}\cdot 3^2\, {u_{921}}^2\, {u_{922}}^3\\
&&-2^5\cdot 3^2\cdot 7\, {u_{921}}^4\, {u_{922}}^4
+3\, {u_{921}}^6\, {u_{922}}^5
+2^{15}\cdot 3\cdot 5\, {u_{921}}^6\, {u_{922}}^7\\
&&+2^{10}\cdot 17\, {u_{921}}^8\, {u_{922}}^8
+2^{17}\cdot 19\, {u_{921}}^{10}\, {u_{922}}^{11}
+2^{13}\cdot 3\cdot 7\, {u_{921}}^{12}\, {u_{922}}^{12}\\
&&+2^7\cdot 23\, {u_{921}}^{14}\, {u_{922}}^{13})
+2^9\, (5\, z)^{-2}\, {u_{921}}^3\, {u_{922}}^4
\, (-2^6\cdot 3\cdot 5+3\, {u_{921}}^2\, u_{922}\\
&&+2^{13}\, {u_{921}}^4\, {u_{922}}^4
+2^{14}\cdot 5\, {u_{921}}^8\, {u_{922}}^8
+2^8\cdot 11\, {u_{921}}^{10}\, {u_{922}}^9)\\
&&-2^{24}\cdot 7\, (5\, z)^{-3}\, {u_{921}}^{12}\, {u_{922}}^{13}]\\
w_{92}&=&u_{922}\, (2^2+2^5\, {u_{921}}^4\, {u_{922}}^4
+{u_{921}}^6\, {u_{922}}^5-2^8\, (5\,z)^{-1}
\, {u_{921}}^5\, {u_{922}}^5)^3,\\
\dot{w}_{92}&=&
2\, u_{922}\, (2^2+2^5\, {u_{921}}^4\, {u_{922}}^4
+{u_{921}}^6\, {u_{922}}^5-2^8\, (5\,z)^{-1}
\, {u_{921}}^5\, {u_{922}}^5)^2\\
&&\times\, [-u_{921}\, {u_{922}}^2\, 
(-2^{10}-2^6\, {u_{921}}^2\, u_{922}
+2^{12}\cdot 7\, {u_{921}}^4\, {u_{922}}^4\\
&&+2^8\cdot 5\, {u_{921}}^6\, {u_{922}}^5
+2^2\cdot 3\, {u_{921}}^8\, {u_{922}}^6
+2^{15}\cdot 3\, {u_{921}}^8\, {u_{922}}^8\\
&&+2^{10}\cdot 7\, {u_{921}}^{10}\, {u_{922}}^9
+2^5\cdot 5\, {u_{921}}^{12}\, {u_{922}}^{10}
+{u_{921}}^{14}\, {u_{922}}^{11})\\
&&+(5\, z)^{-1}\, 
(2^2\cdot 3-2^{12}\cdot 3^2\, {u_{921}}^2\, {u_{922}}^3
-2^5\cdot 3^3\, {u_{921}}^4\, {u_{922}}^4\\
&&+3\, {u_{921}}^6\, {u_{922}}^5
+2^{15}\cdot 3\cdot 5\, {u_{921}}^6\, {u_{922}}^7
+2^{10}\cdot 11\, {u_{921}}^8\, {u_{922}}^8\\
&&+2^{17}\cdot 19\, {u_{921}}^{10}\, {u_{922}}^{11}
+2^{13}\cdot 3\cdot 5\, {u_{921}}^{12}\, {u_{922}}^{12}
+2^7\cdot 11\, {u_{921}}^{14}\, {u_{922}}^{13})\\
&&-2^8\, (5\, z)^{-2}\, {u_{921}}^3\, {u_{922}}^4
\, (-2^6\cdot 3\cdot 5+3\, {u_{921}}^2\, u_{922}
+2^{13}\, {u_{921}}^4\, {u_{922}}^4\\
&&+2^{14}\cdot 5\, {u_{921}}^8\, {u_{922}}^8
+2^{11}\, {u_{921}}^{10}\, {u_{922}}^9)
+2^{23}\cdot 7\, (5\, z)^{-3}\, {u_{921}}^{12}\, {u_{922}}^{13}]
\end{eqnarray*}
\begin{eqnarray*}
E\, w_{92}&=&-2^{-1}\, {u_{921}}^{-1}\, [u_{921}\, 
(-1+2^9\, {u_{921}}^2\, {u_{922}}^3
+2^4\, {u_{921}}^4\, {u_{922}}^4\\
&&+2^{11}\, {u_{921}}^6\, {u_{922}}^7
+2^7\, {u_{921}}^8\, {u_{922}}^8
+2\, {u_{921}}^{10}\, {u_{922}}^9)\\
&&-2^8\, (5\, z)^{-1}\, (-1+2^4\, {u_{921}}^4\, {u_{922}}^4
+2^7\, {u_{921}}^8\, {u_{922}}^8
+2^2\, {u_{921}}^{10}\, {u_{922}}^9)\\
&&+2^{17}\, (5\, z)^{-2}\, {u_{921}}^9\, {u_{922}}^9],\\
\dot{E}\, w_{92}&=&(5\, z)^{-1}\, 
{u_{921}}^{-2}\, {u_{922}}^{-1}\\
&&\times\, [-2^6-3\, {u_{921}}^2\, u_{922}
+2^9\, {u_{921}}^4\, {u_{922}}^4
+2^4\, {u_{921}}^6\, {u_{922}}^5\\
&&+2^{11}\, {u_{921}}^8\, {u_{922}}^8
+2^7\, {u_{921}}^{10}\, {u_{922}}^9
+2\, {u_{921}}^{12}\, {u_{922}}^{10}\\
&&-2^8\, (5\, z)^{-1}\, u_{921}\, u_{922}
\, (-3+2^4\, {u_{921}}^4\, {u_{922}}^4
+2^7\, {u_{921}}^8\, {u_{922}}^8\\
&&+2^2\, {u_{921}}^{10}\, {u_{922}}^9)
+2^{17}\, (5\, z)^{-2}\, {u_{921}}^{10}\, {u_{922}}^{10}]. 
\end{eqnarray*}
The difference between the system and the 
limiting system is given by  
\begin{eqnarray*}
\dot{u}_{921}-^0\dot{u}_{921}&=&-(5\, z)^{-1}
\, {u_{921}}^{-1}\, {u_{922}}^{-1}
\, (128-1280\, (5\, z)^{-1}\, u_{921}\, u_{922}+5\, {u_{921}}^2\, u_{922})\\
\dot{u}_{922}-^0\dot{u}_{922}&=&
2\, (5\, z)^{-1}\, {u_{921}}^{-2}
\, (64-640\, (5\, z)^{-1}\, u_{921}\, u_{922}+3\, {u_{921}}^2\, u_{922}).
\end{eqnarray*}

The equations $u_{921}=0$, 
$2^2+2^5\, {u_{921}}^4\, {u_{922}}^4
+{u_{921}}^6\, {u_{922}}^5-2^8\, (5\,z)^{-1}
\, {u_{921}}^5\, {u_{922}}^5=0$, 
and $u_{922}=0$ define $L_9$, $L_0^{(9)}$, and  
and $L_8^{(1)}$, respectively. The proper transforms 
of the other lines on which the Painlev\'e 
vector field is infinite 
are not visible in this chart.

\end{document}